\documentclass[a4paper,11pt]{article}
\usepackage{LatexDefinitions}
\usepackage{color, booktabs}

\floatstyle{ruled}
\newfloat{algorithmfloat}{t}{lop}
\floatname{algorithmfloat}{Algorithm}

\newcommand{\kernel}{K}
\newcommand{\sKernel}{k}
\newcommand{\proj}{\tn{P}}

\newcommand{\piOpt}{\pi^\ast}
\DeclareMathOperator{\dist}{dist}
\usepackage{esint}

\newcommand{\partA}{\mc{J}_A}
\newcommand{\partB}{\mc{J}_B}
\newcommand{\partC}{\mc{J}_C}
\newcommand{\partGeneric}{\mc{J}}
\newcommand{\muMin}{\mu_{\tn{min}}}

\title{Domain decomposition for entropy regularized optimal transport}
\author{Mauro Bonafini, Bernhard Schmitzer}
\date{\today}

\begin{document}
\maketitle
\begin{abstract}
We study Benamou's domain decomposition algorithm for optimal transport in the entropy regularized setting.
The key observation is that the regularized variant converges to the globally optimal solution under very mild assumptions.
We prove linear convergence of the algorithm with respect to the Kullback--Leibler divergence and illustrate the (potentially very slow) rates with numerical examples. 

On problems with sufficient geometric structure (such as Wasserstein distances between images) we expect much faster convergence.
We then discuss important aspects of a computationally efficient implementation, such as adaptive sparsity, a coarse-to-fine scheme and parallelization, paving the way to numerically solving large-scale optimal transport problems. We demonstrate efficient numerical performance for computing the Wasserstein-2 distance between 2D images and observe that, even without parallelization, domain decomposition compares favorably to applying a single efficient implementation of the Sinkhorn algorithm in terms of runtime, memory and solution quality.
\end{abstract}
\section{Introduction}
\subsection{Motivation}
\label{sec:IntroMotivation}
\paragraph{(Computational) optimal transport.} Optimal transport is a fundamental optimization problem with applications in various branches of mathematics.
Let $\mu$ and $\nu$ be probability measures over spaces $X$ and $Y$ and let $\Pi(\mu,\nu)$ be the set of transport plans, i.e.~probability measures on $X \times Y$ with $\mu$ and $\nu$ as first and second marginal.
Further, let $c : X \times Y \to \R$ be a \emph{cost function}.
The Kantorovich formulation of optimal transport is then given by
\begin{align}
	\label{eq:IntroOT}
	\inf \left\{ \int_{X \times Y} c(x,y)\,\diff \pi(x,y) \middle| \pi \in \Pi(\mu,\nu) \right\}.
\end{align}
We refer to the monographs \cite{Villani-OptimalTransport-09} and \cite{SantambrogioOT} for a thorough introduction and historical context.
Due to its geometric intuition and robustness it is becoming particularly popular in image and data analysis.
Therefore, one main challenge is the development of efficient numerical methods, which has seen immense progress in recent years, such as solvers for the Monge--Amp\`ere equation \cite{ObermanMongeAmpere2014}, semi-discrete methods \cite{LevySemiDiscrete2015,KiMeThi2019}, entropic regularization \cite{Cuturi2013}, and multi-scale methods \cite{MultiscaleTransport2011,SchmitzerSchnoerr-SSVM2013}.
An introduction to computational optimal transport, an overview on available efficient algorithms, and applications can be found in \cite{PeyreCuturiCompOT}.

\paragraph{Domain decomposition.}
Benamou introduced a domain decomposition algorithm for Was\-ser\-stein-2 optimal transport on $\R^d$ \cite{BenamouPolarDomainDecomposition1994}, based on Brenier's polar factorization \cite{MonotoneRerrangement-91}.
In the language of \eqref{eq:IntroOT} the algorithm works as follows: Let $X$ and $Y$ be subsets of $\R^d$ and $c(x,y)=\|x-y\|^2$ the squared Euclidean distance. Let now $(X_1,X_2,X_3)$ be a partition of $X$ and write $X_{\{1,2\}} \assign X_1 \cup X_2$, $X_{\{2,3\}} \assign X_2 \cup X_3$ and let $\iterz{\pi} \in \Pi(\mu,\nu)$ be some initial feasible transport plan.
The next iterate $\iter{\pi}{1}$ is then obtained by optimizing $\iterz{\pi}$ on the set $X_{\{1,2\}} \times Y$, while keeping it fixed on $X_3 \times Y$.
Next, $\iter{\pi}{2}$ is obtained by optimizing $\iter{\pi}{1}$ on $X_{\{2,3\}} \times Y$, while keeping it fixed on $X_1 \times Y$. These two steps are then repeated.
This is illustrated (for entropic optimal transport) in Figure \ref{fig:IntroIllustration}.
\begin{figure}[t]
	\centering
	{\def\imgw{2.3cm}
	\begin{tikzpicture}[x=\imgw,y=\imgw,img/.style={inner sep=0pt,anchor=south west}]
	\foreach \x/\y/\l in {0/0/0,1/0/1,2/0/2,3/0/3,4/0/4,5/0/5} {
		\node[img] at (1.1*\x,-1.15*\y)[label={below:$\iter{\pi}{\l}$}]{\includegraphics[width=\imgw]{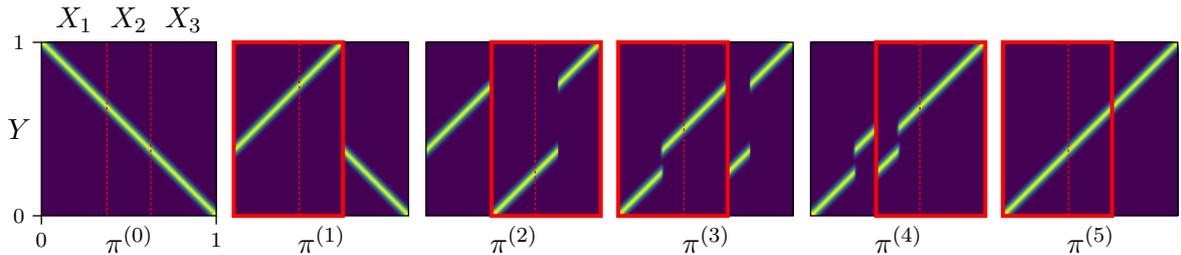}};
	}
	{\scriptsize
	\begin{scope}[line width=0.5pt,black]
		\draw (0,0) -- ++(-0.05,0) node[anchor=east]{0};
		\draw (0,1) -- ++(-0.05,0) node[anchor=east]{1};
		\draw (0,0) -- ++(0,-0.05) node[anchor=north]{0};
		\draw (1,0) -- ++(0,-0.05) node[anchor=north]{1};
	\end{scope}
	}
	\node at (0,0.5) [anchor=east]{$Y$};
	\node at (24/128,1) [anchor=south]{$X_1$};
	\node at (0.5,1) [anchor=south]{$X_2$};
	\node at (104/128,1) [anchor=south]{$X_3$};
	\foreach \x in {0,1,2,3,4,5} {
		\begin{scope}[shift={(1.1*\x,0)},x={(\imgw/128,0)},y={(0,\imgw/128)}]
		\draw[red,line width=0.5pt,dash pattern=on 1pt off 1pt] (48,0) -- (48,128);
		\draw[red,line width=0.5pt,dash pattern=on 1pt off 1pt] (80,0) -- (80,128);
		\end{scope}
	}
	\foreach \x in {0,1,2,3,4,5} {
		\begin{scope}[shift={(1.1*\x,0)},x={(\imgw/128,0)},y={(0,\imgw/128)}]
		\draw[black,line width=0.5pt] (0,0) -- (128,0) -- (128,128) -- (0,128) -- cycle;
		\end{scope}
	}
	\foreach \x in {1,3,5} {
		\begin{scope}[shift={(1.1*\x,0)},x={(\imgw/128,0)},y={(0,\imgw/128)}]
		\draw[red,line width=1.5pt] (0,0) -- (80,0)-- (80,128) -- (0,128) -- cycle;
		\end{scope}
	}
	\foreach \x in {2,4} {
		\begin{scope}[shift={(1.1*\x,0)},x={(\imgw/128,0)},y={(0,\imgw/128)}]
		\draw[red,line width=1.5pt] (48,0) -- (128,0) -- (128,128) -- (48,128) -- cycle;
		\end{scope}
	}
	\end{tikzpicture}
	}
\caption{Iterates of the domain decomposition algorithm for entropic optimal transport in one dimension (yellow represents high, blue low mass density). %
$\mu=\nu$ is the uniform probability measure on $X=Y=[0,1]$, equipped with cost function $c(x,y)=\|x-y\|^2$. The optimal coupling is a blurred `diagonal' measure (approximately equal to $\iter{\pi}{5}$). We initialize with an `anti-diagonal' feasible $\iterz{\pi}$, i.e.~the order must be approximately reversed. $X$ is divided into three cells (dashed lines, first panel) and $\iterl{\pi}$ is alternatingly optimized over $X_{\{1,2\}}$ and $X_{\{2,3\}}$. Iterates $\iter{\pi}{1}$ to $\iter{\pi}{5}$ were obtained by optimizing the previous iterate over the red rectangle.}
\label{fig:IntroIllustration}
\end{figure}

In \cite{BenamouPolarDomainDecomposition1994} it is shown that this algorithm converges to the globally optimal solution when $\mu$ is Lebesgue absolutely continuous and $(X_1,X_2,X_3)$ satisfy a `convex overlap' condition, which roughly states that `if a function is convex on $X_{\{1,2\}}$ and $X_{\{2,3\}}$, then it must be convex on $X$'.
Further, it is shown that this method can easily be extended to more than three cells, thus leading to a parallelizable algorithm.
Unfortunately, it is also demonstrated that the discretized method does not converge to a global minimizer (for $d>1$).
When the discretization is refined, it is shown that the optimal solution is recovered in the limit. But this requires an increasing number of discretization points in the three sets $(X_1,X_2,X_3)$ and thus numerically obtaining a good approximation of the globally optimal solution remains challenging with this method.

\subsection{Outline and contribution}
\label{sec:IntroContribution}
\paragraph{Entropic domain decomposition.}
The key observation in this article is that for entropic optimal transport the domain decomposition algorithm converges to the unique globally optimal solution for arbitrary (measurable) bounded costs $c$, as soon as $\mu(X_2)>0$. In particular, this covers the case where $X$ and $Y$ are discrete and finite.
Therefore, entropic smoothing is not only a useful tool in solving the optimal transport sub-problems arising in domain decomposition, it also facilitates convergence of the whole algorithm.
While the convergence rate can be very slow in suitable worst-case examples, we show empirically that it is fast on problems with sufficient geometric structure,  thus leading to a practical and efficient parallel numerical method for large-scale problems.

\paragraph{Preliminaries and definition of the algorithm.} We start by establishing the basic mathematical setting and notation in Section \ref{sec:Notation}. Some basic facts about entropic optimal transport are recalled in Section \ref{sec:ReminderEntropicOT}.
In Section \ref{sec:Algorithm} the entropic domain decomposition algorithm is defined and a simple convergence proof is sketched.

\paragraph{Theoretical and empirical worst-case convergence rate.} Section \ref{sec:Convergence} is dedicated to proving that the iterates converge to the unique optimal solution linearly with respect to the Kullback--Leibler divergence.
We first give a proof for the three-cell example discussed above (Section \ref{sec:ConvergenceSimple}) and then generalize it to more general decompositions (Section \ref{sec:ConvergenceGeneral}).
Our results apply to arbitrary (measurable) bounded cost functions $c$, marginal measures $\mu$ and $\nu$, and we make minimal assumptions on the decomposition structure (it must be `connected' in a suitable sense, Definition \ref{def:PartitionGraph}, in particular convergence only requires $\mu(X_2)>0$ in the three-cell example).
As the entropic regularization parameter tends to zero, our rate bound tends to one exponentially.
This is reminiscent of the convergence rate for the standard Sinkhorn algorithm established in \cite{FranklinLorenz-Scaling-1989} with respect to Hilbert's projective metric.
While we do not expect that our bounds on the convergence rate are tight, we demonstrate in Section \ref{sec:NumericsWorstCase} with some numerical (near) worst-case examples that they capture the qualitative behaviour of the algorithm on difficult problems.

\paragraph{Relation to parallel sorting.} The slow convergence rates from Section \ref{sec:Convergence} do not suggest that the algorithm is efficient. But these slow rates rely on maliciously designed counter-examples. In practice usually much more geometric structure is available, similar to the setting originally considered by Benamou \cite{BenamouPolarDomainDecomposition1994}.
While a detailed analysis of the convergence rate in these cases is beyond the scope of this article, to provide some intuition, we discuss in Section \ref{sec:ConvergenceSorting} the relation of the domain decomposition algorithm for optimal transport with the odd-even transposition parallel sorting algorithm \cite[Exercise 37]{KnuthArtProgramming1998-3} in one dimension. This algorithm is known to converge in $O(N)$ iterations where $N$ is the number of cells that the domain is partitioned into.

\paragraph{Practical implementation and geometric large-scale examples.} In Section \ref{sec:Implementation} we provide a practical version of the domain decomposition algorithm, leading to an efficient numerical method for large scale problems with geometric structure.
We discuss how the memory footprint can be reduced, how one can handle approximate solutions to the sub-problems as obtained by the Sinkhorn algorithm, and how to combine the algorithm with the $\veps$-scaling heuristic and a multi-scale scheme, see \cite{SchmitzerScaling2019}.

\paragraph{Numerical experiments and comparison with single Sinkhorn algorithm.}
The efficiency of the implementation is then demonstrated numerically by computing the Wasserstein-2 distance between images in Section \ref{sec:NumericsLargeScale}.
We report accurate approximate primal and dual solutions with low entropic regularization artifacts after a logarithmic number (w.r.t.~the marginal size) of domain decomposition iterations. On a standard desktop computer with 6 cores a high-quality approximate optimal transport between two mega-pixel images is computed in approximately 4 minutes.

A comparison to a single Sinkhorn algorithm, when applied to the full problem, is given in Section \ref{sec:NumericsSingle}.
We find that the domain decomposition approach has several key advantages.
First, it allows more extensive parallelization.
The standard Sinkhorn algorithm is already parallel in the sense that each half-iteration consists essentially of a matrix-vector product that can be parallelized. But the results of this operation must be communicated to all workers after each half-iteration, thus only allowing local parallelization such as a single GPU.
In the domain decomposition variant, small sub-problems are solved independently and their results must only be communicated upon completion, thus allowing parallelization over multiple machines.

Second, domain decomposition allows to reduce the memory footprint.
The naive Sinkhorn algorithm requires storage of the full kernel matrix, the size of which grows quadratically with the marginal size.
This can be avoided by using efficient heuristics, such as Gaussian convolutions or pre-factored heat kernels \cite{Solomon-siggraph-2015} but these methods only work in particular settings and with sufficiently high regularization.
Alternatively, with adaptive kernel truncation \cite{SchmitzerScaling2019} the memory demand is approximately linear in the marginal size. But to retain numerical stability, the truncation parameter may not be chosen too aggressively, thus still making memory demand a practical constraint.
In the domain decomposition variant, only the kernel matrices for the sub-problems that are currently being re-solved are needed; for all other sub-problems only information about their marginals is kept.
In addition, via parallelization this may also be distributed over several computers.

In our numerical experiments we find that in comparable runtime (even without parallelization) the domain decomposition algorithm obtains more accurate primal and dual iterates by requiring only approximately 11\% of the memory of the single Sinkhorn solver.
Dealing only with small sub-problems at a time also makes it easier to handle numerical delicacies associated with small entropic regularization parameters. The domain decomposition method therefore more reliably solves large problems.

\section{Background}
\subsection{Setting and notation}
\label{sec:Notation}
\begin{itemize}
	\item $X$ and $Y$ are compact metric spaces. We assume compactness to avoid overly technical arguments while covering the numerically relevant setting.
	\item For a compact metric space $Z$ the set of finite signed Radon measures over $Z$ is denoted by $\meas(Z)$. The subsets of non-negative and probability measures are denoted by $\measp(Z)$ and $\prob(Z)$. The Radon norm of $\mu \in \meas(Z)$ is denoted by $\|\mu\|_{\meas(Z)}$ and one has $\|\mu\|_{\meas(Z)}=\mu(Z)$ for $\mu \in \measp(Z)$.
	We often simply write $\|\mu\|$ for the norm when the meaning is clear from context.
	\item For $\mu \in \measp(Z)$ we denote by $L^1(Z,\mu)$, $L^\infty(Z,\mu)$ the usual function spaces and add a subscript $+$ to denote the subsets of $\mu$-a.e.~non-negative functions. The corresponding norms are denoted by $\|\cdot\|_{L^1(Z,\mu)}$ and $\|\cdot\|_{L^\infty(Z,\mu)}$ but we often merely write $\|\cdot\|_1$ and $\|\cdot\|_\infty$ (or even $\|\cdot\|$ for the latter) when space and measure are clear from context.
	\item For $\mu \in \measp(Z)$ and a measurable function $u : Z \to \R$ we denote by $\la \mu, u \ra$ the integration of $u$ with respect to $\mu$, and define the normalized integral
	\begin{align}\label{eq:fint}
	\fint_Z u \,\diff \mu \assign \frac{1}{\mu(Z)} \int_Z u\,\diff \mu = \frac{\la \mu, u \ra}{\mu(Z)}.
	\end{align}
	\item For $\mu \in \measp(Z)$ and measurable $S \subset Z$ the restriction of $\mu$ to $S$ is denoted by $\mu \restr S$.
	Similarly, for a measurable function $u : Z \to \R$, set $u \restr S(x)=u(x)$ if $x \in S$, and $0$ otherwise.
	\item The maps $\proj_X : \measp(X \times Y) \to \measp(X)$ and $\proj_Y : \measp(X \times Y) \to \measp(Y)$ denote the projections of measures on $X \times Y$ to their marginals, i.e.
	\begin{align*}	
	(\proj_X \pi)(S_X) \assign \pi(S_X \times Y) \qquad \tn{and} \qquad (\proj_Y \pi)(S_Y) \assign \pi(X \times S_Y)
	\end{align*}
	for $\pi \in \measp(X \times Y)$, $S_X \subset X$, $S_Y \subset Y$ measurable.
	\item For (measurable) functions $a : X \to \RCupInf$, $b : Y \to \RCupInf$, the functions $a \oplus b, a \otimes b : X \times Y \to \RCupInf$ are given by
	\begin{align*}
		(a \oplus b)(x,y) & = a(x) + b(y), &
		(a \otimes b)(x,y) & = a(x) \cdot b(y).
	\end{align*}
	For two measures $\mu \in \measp(X)$, $\nu \in \measp(Y)$ their product measure on $X \times Y$ is denoted by $\mu \otimes \nu$.
\end{itemize}

\subsection{Entropic optimal transport}
\label{sec:ReminderEntropicOT}
We recall in this Section the entropic regularization approach to optimal transport and collect the main existence and characterization results which will be needed in this article.

\begin{definition}[Kullback--Leibler divergence]
Let $Z$ be a compact metric space. For $\mu \in \meas(Z)$, $\nu \in \measp(Z)$ the \emph{Kullback--Leibler divergence} (or relative entropy) of $\mu$ w.r.t.~$\nu$ is given by
\begin{align*}
\KL(\mu|\nu) \assign \begin{cases}
\int_Z \varphi\left(\RadNik{\mu}{\nu}\right)\,\diff \nu & \tn{if } \mu \ll \nu,\,\mu \geq 0, \\
+ \infty & \tn{else,}
\end{cases} 
\quad \tn{with} \quad
\varphi(s) \assign \begin{cases}
s\,\log(s)-s+1 & \tn{if } s>0, \\
1 & \tn{if } s=0, \\
+ \infty & \tn{else.}
\end{cases}
\end{align*}
The Fenchel--Legendre conjugate of $\varphi$ is given by $\varphi^\ast(z) = \exp(z)-1$.
\end{definition}

\begin{definition}[Entropy regularized optimal transport]
Let $\mu \in \prob(X)$, $\nu \in \prob(Y)$ and $\hat{\mu} \in \measp(X)$ and $\hat{\nu} \in \measp(Y)$ such that
\begin{align*}
	\|\hat{\mu}\|=\|\hat{\nu}\| & >0, &
	\hat{\mu} & \ll \mu, & \hat{\nu} & \ll \nu, &
	\RadNik{\hat{\mu}}{\mu} & \in L^\infty(X,\mu), &
	\RadNik{\hat{\nu}}{\nu} & \in L^\infty(Y,\nu).
\end{align*}
Similar to above, denote by
\begin{align*}
	\Pi(\hat{\mu},\hat{\nu}) & \assign
	\left\{ \pi \in \measp(X \times Y) \,\middle|\, \proj_X \pi=\hat{\mu},\proj_Y \pi = \hat{\nu}\right\}
\end{align*}
the set of \emph{transport plans} between $\hat{\mu}$ and $\hat{\nu}$.
The condition $\|\hat{\mu}\|=\|\hat{\nu}\|$ ensures that the set is non-empty.
For a lower-semicontinuous cost function $c \in L^\infty_+(X \times Y, \mu \otimes \nu)$ the Kantorovich optimal transport problem between $\hat{\mu}$ and $\hat{\nu}$ is then given by
\begin{align}
\label{eq:OTPrimal}
\inf \left\{ \int_{X \times Y} c(x,y)\,\diff \pi(x,y) \middle| \pi \in \Pi(\hat{\mu},\hat{\nu}) \right\}.
\end{align}
Existence of minimizers in this setting is provided, for instance, by \cite[Theorem 4.1]{Villani-OptimalTransport-09}.
For a regularization parameter $\veps>0$ define
\begin{equation}
\sKernel \assign \exp(-c/\veps) \in L^\infty_+(X \times Y,\mu \otimes \nu), \qquad \text{and} \qquad \kernel \assign \sKernel \cdot (\mu \otimes \nu).
\end{equation}
In this article we will frequently exploit that $k$ is bounded by $0 < \exp(-\|c\|_{\infty})/\veps) \leq k \leq 1$.
Then the entropic optimal transport problem, regularized with respect to $\mu \otimes \nu$ is given by
\begin{equation}
\inf \left\{ \veps \KL(\pi|\kernel) \,\middle|\, \pi \in \Pi(\hat{\mu},\hat{\nu})\right\}.
\label{eq:EntropicOTPrimal}
\end{equation}
\end{definition}
Problem \eqref{eq:EntropicOTPrimal} can be solved (approximately) with the Sinkhorn algorithm (Remark \ref{rem:Sinkhorn}). In addition, in our article, entropic smoothing will ensure convergence of the domain decomposition scheme.
We refer to \cite[Chapter 4]{PeyreCuturiCompOT} for an overview on entropic optimal transport and its numerical implications.
The ($\Gamma$-)convergence of \eqref{eq:EntropicOTPrimal} to \eqref{eq:OTPrimal} as $\veps \to 0$ has been shown in \cite{Cominetti-ExpBarrierConvergence-1992,LeonardSchroedingerMK2012,Carlier-EntropyJKO-2015} in various settings and with different strategies.

We now collect some results about entropic optimal transport required in this article. A proof is given in the appendix.
\begin{proposition}[Optimal entropic transport couplings]\hfill
	\label{prop:EntropicOTBasic}
	\begin{enumerate}[(i)]
		\item \eqref{eq:EntropicOTPrimal} has a unique minimizer $\pi^\ast \in \Pi(\hat{\mu},\hat{\nu})$.
			\label{item:EntropicOTBasic:Unique}
		\item \label{item:EntropicOTBasic:Scaling}
			There exist measurable $u^\ast : X \to \R_+$, $v^\ast : Y \to \R_+$ such that $\pi^\ast = u^\ast \otimes v^\ast \cdot \kernel$. $u^\ast$ and $v^\ast$ are unique $\mu$-a.e.~and $\nu$-a.e.~up to a positive re-scaling $(u^\ast,v^\ast) \to (\lambda \cdot u^\ast, \lambda^{-1} \cdot v^\ast)$ for $\lambda>0$.
		\item The couple ($u^\ast$, $v^\ast$) satisfies the integral equations
		\begin{equation}
		\label{eq:ExtremalityCondition}
		u^\ast(x) \int_Y \sKernel(x,y')v^\ast(y')\,\diff\nu(y') = \RadNik{\hat{\mu}}{\mu}(x) \quad \text{and} \quad v^\ast(y) \int_X \sKernel(x',y)u^\ast(x')\,\diff\mu(x') = \RadNik{\hat{\nu}}{\nu}(y)
		\end{equation}
		for $\mu$-a.e.~$x \in X$ and $\nu$-a.e.~$y \in Y$.
		\label{item:EntropicOTBasic:Marginals}
			\item $u^\ast \in L^\infty_+(X,\mu)$, $v^\ast \in L^\infty_+(Y,\nu)$, $\|u^\ast\|_{L^1(X,\mu)}>0$, $\|v^\ast\|_{L^1(Y,\nu)}>0$,
			\begin{align}
			\label{eq:uBound}
			\tfrac{1}{\|v^\ast\|_1} \RadNik{\hat{\mu}}{\mu}(x) & \leq u^\ast(x) \leq \tfrac{\exp(\|c\|_\infty/\veps)}{\|v^\ast\|_1} \RadNik{\hat{\mu}}{\mu}(x) & & \text{for } \mu\text{-a.e.~}x \in X, \\
			\label{eq:vBound}
			\tfrac{1}{\|u^\ast\|_1} \RadNik{\hat{\nu}}{\nu}(y) & \leq v^\ast(y) \leq \tfrac{\exp(\|c\|_\infty/\veps)}{\|u^\ast\|_1} \RadNik{\hat{\nu}}{\nu}(y) & & \text{for } \nu\text{-a.e.~}y \in Y
			\end{align}
			and $\log u^\ast \in L^1(X,\hat{\mu})$, $\log v^\ast \in L^1(Y,\hat{\nu})$.
			\label{item:EntropicOTBasic:Bounds}
	\end{enumerate}
\end{proposition}

\begin{definition}[Dual problem]
In addition to \eqref{eq:EntropicOTPrimal}, we will consider a corresponding dual problem which, for the purpose of this article, is best stated as
\begin{equation}
	\label{eq:EntropicOTDual}
	\sup \left\{ J(u,v) \middle| u \in L^\infty_+(X,\mu),\,v \in L^\infty_+(Y,\nu) \right\}
\end{equation}
with $J : L^\infty_+(X,\mu) \times L^\infty_+(Y,\nu) \mapsto [-\infty,\infty)$ given by
\begin{align}
	\label{eq:EntropicOTDualObjective}
	J: (u,v) & \mapsto \int_X \log u\,\diff \hat{\mu} + \int_Y \log v\,\diff \hat{\nu} - \int_{X \times Y} u \otimes v\,\diff \kernel + \|\kernel\|.
\end{align}
\end{definition}
This is a reparametrization of the usual dual problem where one uses $\veps \log u$ and $\veps \log v$ as variables instead (cf.~\cite{SchmitzerScaling2019}).
Again, we briefly collect a few helpful results. A proof is given in the appendix.
\begin{proposition}[Duality for regularized optimal transport] \label{prop:Duality}\hfill
\begin{enumerate}[(i)]
	\item The functional $J$, \eqref{eq:EntropicOTDualObjective}, is well defined.
	\label{item:Duality:Defined}
	\item \label{item:Duality:Gap}
		For $u \in L^\infty_+(X,\mu)$, $\nu \in L^\infty_+(Y,\nu)$, $\pi \in \Pi(\hat{\mu},\hat{\nu})$ one has $J(u,v) \leq \KL(\pi|\kernel)$.
	\item \label{item:Duality:OptimalityCondition} A transport plan $\pi^\ast \in \Pi(\hat{\mu},\hat{\nu})$ of the form $\pi^\ast = u^\ast \otimes v^\ast \cdot \kernel$ for $u^\ast \in L^\infty_+(X,\mu)$, $v^\ast \in L^\infty_+(Y,\nu)$ is optimal for \eqref{eq:EntropicOTPrimal}.
	In that case, $(u^\ast,v^\ast)$ are optimal for \eqref{eq:EntropicOTDual}. Maximizers for \eqref{eq:EntropicOTDual} exist.
\end{enumerate}
\end{proposition}

\begin{remark}[Sinkhorn algorithm]
\label{rem:Sinkhorn}
Problems \eqref{eq:EntropicOTPrimal} and \eqref{eq:EntropicOTDual} can be solved (approximately) with the Sinkhorn algorithm.
For some initial $\iterz{u} \in L_+^\infty(X,\mu)$ (not identical to zero), it is given for $\ell=0,1,2,\ldots$ by
\begin{align*}
	\iterl{v}(y) & \assign \frac{\RadNik{\hat{\nu}}{\nu}(y)}{\int_X \sKernel(x,y)\,\iterl{u}(x)\,\diff \mu(x)}, &
	\iterll{u}(x) & \assign \frac{\RadNik{\hat{\mu}}{\mu}(x)}{\int_Y \sKernel(x,y)\,\iterl{v}(x)\,\diff \nu(x)}.
\end{align*}
We refer to these two steps as $Y$- and $X$-iteration. It is quickly verified that a $Y$-iteration corresponds to a partial optimization of $J$, \eqref{eq:EntropicOTDualObjective}, over $v$ for fixed $u$. Conversely, an $X$-iteration is obtained by optimizing over $u$ for fixed $v$. Hence the sequence $(\iterl{u},\iterl{v})_\ell$ is a maximizing sequence of \eqref{eq:EntropicOTDualObjective}.
Complementarily, the sequence of measures $\iterl{\pi} \assign (\iterl{u} \otimes \iterl{v}) \cdot \kernel$ converges (under suitable assumptions) to the solution of \eqref{eq:EntropicOTPrimal}. In general one has $\proj_Y \iterl{\pi} = \nu$ but $\proj_X \iterl{\pi} \neq \mu$. Conversely, for $\iter{\pi}{\ell+1/2} \assign (\iterll{u} \otimes \iterl{v}) \cdot \kernel$ one has $\proj_X \iter{\pi}{\ell+1/2} = \mu$ but $\proj_Y \iter{\pi}{\ell+1/2} \neq \nu$.
Again, for a thorough overview we refer to \cite[Chapter 4]{PeyreCuturiCompOT}. A computationally efficient implementation is discussed in \cite{SchmitzerScaling2019}, which we use as a reference in this article.
\end{remark}

\section{Entropic domain decomposition algorithm}
\label{sec:Algorithm}
Throughout this article we will be concerned with solving
\begin{align}
\label{eq:Problem}
\min \left\{ \veps\,\KL(\pi|\kernel) \,\middle|\, \pi \in \Pi(\mu,\nu) \right\}
\end{align}
for $\mu \in \prob(X)$, $\nu \in \prob(Y)$, $c \in L^\infty_{+}(X \times Y,\mu \otimes \nu)$, $\veps>0$, $\sKernel \assign \exp(-c/\veps)$, $\kernel \assign \sKernel \cdot (\mu \otimes \nu)$. By Proposition \ref{prop:EntropicOTBasic}, \eqref{eq:Problem} has a unique minimizer $\pi^\ast$ that can be represented as $\pi^\ast = (u^\ast \otimes v^\ast) \cdot \kernel$ for corresponding scaling factors $u^\ast$, $v^\ast$.

Following Benamou \cite{BenamouPolarDomainDecomposition1994} (see Section \ref{sec:IntroMotivation}), the strategy of the algorithm is as follows:
Divide the space $X$ into two partitions $\partA$ and $\partB$ (which imply two partitions of the product space $X \times Y$) that need to `overlap' in a suitable sense.
Then, starting from an initial feasible coupling $\pi^0 \in \Pi(\mu,\nu)$, optimize the coupling on each of the cells of $\partA$ separately, while keeping the marginals on that cell fixed. This can be done independently and in parallel for each cell and the coupling attains a potentially better score while remaining feasible.
Then repeat this step on partition $\partB$, then again on $\partA$ and so on, continuing to alternate between the two partitions.

To construct $\partA$ and $\partB$ we first define a \emph{basic partition} of small cells. The two overlapping partitions are then created by suitable merging of the basic cells.
This induces a graph over the basic partition cells as vertices. The algorithm converges when this graph is connected, see Section \ref{sec:Convergence} for details.

\begin{definition}[Basic and composite partitions]
	A partition of $X$ into measurable sets $\{X_i\}_{i \in I}$, for some finite index set $I$, is called a basic partition of $(X,\mu)$ if the measures $\mu_i \assign \mu \restr X_i$ for $i \in I$ satisfy $\|\mu_i\|>0$.
	By construction one has $\sum_{i \in I} \mu_i = \mu$.
	In addition we will write $\kernel_i \assign \kernel \restr (X_i \times Y)$ for $i \in I$.
	Often we will refer to a basic partition merely by the index set $I$.

	For a basic partition $\{X_i\}_{i \in I}$ of $(X,\mu)$ a composite partition $\partGeneric$ is a partition of $I$.
	For $J \in \partGeneric$ we will use the following notation:
	\begin{align*}
	X_J & \assign \bigcup_{i \in J} X_i, &
	\mu_J & \assign \sum_{i \in J} \mu_i=\mu \restr X_J, &
	\kernel_J & \assign \sum_{i \in J} \kernel_i = \kernel \restr (X_J \times Y).
	\end{align*}
	Of course, the family $\{X_J\}_{J \in \partGeneric}$ is a measurable partition of $X$ and the families $\{X_i \times Y\}_{i \in I}$ and $\{X_J \times Y\}_{J \in \partGeneric}$ are measurable partitions of $X \times Y$.
\end{definition}

Throughout the article we will use one basic partition $I$ and two corresponding composite partitions $\partA$ and $\partB$.
A formal statement of the algorithm is given in Algorithm \ref{alg:DomDec}.

\begin{example}
\label{ex:ThreeCells}
In the three-cell example from Section \ref{sec:IntroMotivation}, the basic partition is given by $\{X_1,\allowbreak X_2,X_3\}$, $I=\{1,2,3\}$ and the two composite partitions are given by $\partA=\{\{1,2\},\{3\}\}$ and $\partB=\{\{1\},\{2,3\}\}$.
For $\ell$ odd, i.e.~during an $\partA$-iteration, the optimization on the `single cell' $\{3\}$ can be skipped since it is made redundant by the optimization on $\{2,3\}$ in the subsequent $\partB$-iteration: One has $P_Y \iterl{\pi} \restr (X_{3} \times Y))=P_Y \iterlm{\pi} \restr (X_{3} \times Y))$ and thus, when computing the partial marginal for cell $\{2,3\}$ in the next iteration one finds
\begin{multline*}\iterll{\nu_{\{2,3\}}}=P_Y(\iterl{\pi} \restr(X_{\{2,3\}} \times Y))=P_Y(\iterl{\pi} \restr(X_{2} \times Y))+P_Y(\iterl{\pi} \restr(X_{3} \times Y)) \\ =P_Y(\iterl{\pi} \restr(X_{2} \times Y))+P_Y(\iterlm{\pi} \restr(X_{3} \times Y)).
\end{multline*}
The same holds for the cell $\{1\}$ during the $\partB$-iteration.
We then obtain the method described in the introduction.
\end{example}

\begin{algorithmfloat}[ht]	
	\noindent
	\textbf{Input}: initial feasible coupling $\pi^{0} \in \Pi(\mu,\nu)$
	
	\noindent
	\textbf{Output}: a sequence $(\iterl{\pi})_{\ell}$ of feasible couplings in $\Pi(\mu,\nu)$
	\smallskip

	\begin{algorithmic}[1]
		\State $\iterz{\pi} \leftarrow \pi^{0}$
		\State $\ell \leftarrow 0$
		\Loop
		\State $\ell \leftarrow \ell+1$
		\State \algorithmicif\ ($\ell$ is odd)\ \algorithmicthen\ $\iterl{\partGeneric} \leftarrow \partA$\ \algorithmicelse\ $\iterl{\partGeneric} \leftarrow \partB$
		\Comment{select the partition}
		\ForAll{$J \in \iterl{\partGeneric}$}
		\Comment{iterate over each composite cell}
		\State $\iterl{\nu_J} \leftarrow P_Y(\iterlm{\pi} \restr(X_{J} \times Y))$
			\Comment{compute $Y$-marginal on cell}
			\label{line:YMarginal}
		\State $\iterl{\pi_{J}} \leftarrow \arg\min \left\{ \veps\,\KL(\pi|\kernel) \,\middle|\, \pi \in \Pi\left(\mu_J,\iterl{\nu_J}\right)\right\}$ \Comment{minimizer is unique}
			\label{line:CompositePi}
		\EndFor
		\State $\iterl{\pi} \leftarrow \sum_{J \in \iterl{\partGeneric}} \iterl{\pi_{J}}$
		\EndLoop
	\end{algorithmic}
	
	\noindent \textbf{Comment:} The theoretical, idealized algorithm runs indefinitely and thus has no stopping criterion. In practice, the stopping criterion may be a fixed number of iterations or a bound on the decrement of the primal score.
\caption{Entropic domain decomposition}
\label{alg:DomDec}
\end{algorithmfloat}

\begin{proposition}
	Algorithm \ref{alg:DomDec} is well-defined and for all $\ell \in \N_0$ one has $\iterl{\pi} \in \Pi(\mu,\nu)$.
\end{proposition}
\begin{proof}
	$\iterz{\pi} \in \Pi(\mu,\nu)$ by assumption.
	Assume now $\iterlm{\pi} \in \Pi(\mu,\nu)$ for some $\ell \geq 1$.
	Since
	$$\mu_J=\mu \restr X_J = (\proj_X \iterlm{\pi}) \restr X_J = \proj_X (\iterlm{\pi} \restr (X_J \times Y))$$
	we find that $\|\iterl{\nu}_J\|=\|\mu_J\|>0$ (for $\iterl{\nu}_J$ as introduced in line \ref*{line:YMarginal}). Also, since $0 \leq \mu_J \leq \mu$ and $0 \leq \iterl{\nu}_J\leq \nu$, the values of $\RadNik{\iterl{\nu}_J}{\nu}$ and $\RadNik{\iterl{\mu}_J}{\mu}$ are (up to negligible sets) contained in $[0,1]$. Therefore, by Proposition \ref{prop:EntropicOTBasic}, the regularized optimal transport problem in line \ref*{line:CompositePi} has a unique solution and thus $\iterl{\pi}_J$ is well-defined.
	We verify
	\begin{align*}
		\proj_X \iterl{\pi} & =\sum_{J \in \iterl{\partGeneric}} \proj_X \iterl{\pi}_J = \sum_{J \in \iterl{\partGeneric}} \mu_J = \mu, \\
		\proj_Y \iterl{\pi} & =\sum_{J \in \iterl{\partGeneric}} \proj_Y \iterl{\pi}_J = \sum_{J \in \iterl{\partGeneric}} \iterl{\nu}_J
		= \sum_{J \in \iterl{\partGeneric}} \proj_Y (\iterlm{\pi} \restr (X_J \times Y))
		= \proj_Y \iterlm{\pi} = \nu.
	\end{align*}
	Therefore, $\iterl{\pi} \in \Pi(\mu,\nu)$ and the claim follows by induction.
\end{proof}

\begin{definition}[Notation of iterates]
\label{def:Notation}
For $\ell\geq 1$ the following sequences will be used throughout the rest of this article:
\begin{itemize}
	\item Composite partitions $\iterl{\partGeneric} \assign \partA$ if $\ell$ is odd, $\partB$ else, as defined in Algorithm \ref{alg:DomDec}.
	\item Primal iterates $\iterl{\pi}$, as defined in Algorithm \ref{alg:DomDec}.
	\item Partial primal iterates and corresponding $Y$-marginals on the composite partition cells,
	\begin{align*}
		\iterl{\pi}_J & \assign \iterl{\pi} \restr (X_J \times Y), &
		\iterl{\nu}_J & \assign \proj_Y \iterl{\pi}_J
	\end{align*}
	for $J \in \iterl{\partGeneric}$, consistent with their definition in Algorithm \ref{alg:DomDec}.
	Let $\iterl{u}_J$, $\iterl{v}_J$ be corresponding scaling factors such that $\iterl{\pi}_J = \iterl{u}_J \otimes \iterl{v}_J \cdot \kernel$. Their existence is provided by Proposition \ref{prop:EntropicOTBasic}.
	\item Partial primal iterates and corresponding $Y$-marginals on the basic partition cells,
	\begin{align*}	
		\iterl{\pi}_i & \assign \iterl{\pi} \restr (X_i \times Y), &
		\iterl{\nu}_i & \assign \proj_Y \iterl{\pi}_i,
	\end{align*}
	for $i \in I$. It is easy to verify that for
	\begin{align}
		\label{eq:BasicScalings}
		\iterl{u}_i(x) & \assign \begin{cases} \iterl{u}_J(x) & \tn{if } x \in X_i, \\
			0 & \tn{else,}
			\end{cases}, &
		\iterl{v}_i(y) & \assign \iterl{v}_J(y)
	\end{align}
	where $J \in \iterl{\partGeneric}$ is uniquely determined by the condition $i \in J$,
	one finds $\iterl{\pi}_i = \iterl{u}_i \otimes \iterl{v}_i \cdot \kernel$ and therefore, by Propositions \ref{prop:EntropicOTBasic} and \ref{prop:Duality}, $\iterl{\pi}_i$ is the unique optimal entropic coupling between $\mu_i$ and $\iterl{\nu}_i$. This also follows from a restriction argument as in \cite[Theorem 4.6]{Villani-OptimalTransport-09}.
\end{itemize}
\end{definition}

Based on Proposition \ref{prop:EntropicOTBasic}, partial iterates satisfy locally some corresponding bounds, which we collect in the following Lemma, the proof of which is postponed to the Appendix.
\begin{lemma}[Boundedness of scaling factors and density bounds]
	\label{lem:ScalingLemma}
	Let $\ell > 1$, $i \in I$ and $J \in \iterl{\partGeneric}$ such that $i \in J$.
	\begin{enumerate}[(i)]
		\item
		\label{item:ScalingLemma:Bounds}
		There exist $\bar{C}, \underbar{C} > 0$, possibly depending on $\ell$ and $J$, such that
		\[
		\left\{
		\begin{aligned}
		\underbar{C} &\leq \iterl{u}_J(x) \leq \bar{C} \quad \text{for } \mu_J\text{-a.e.~} x \in X_J \\
		\underbar{C} \cdot \RadNik{\iterl{\nu_J}}{\nu}(y) &\leq \iterl{v}_J(y) \leq \bar{C} \quad \text{for } \nu\text{-a.e.~} y \in Y
		\end{aligned}
		\right.
		\]
		and, in particular, $\log \iterl{u}_J \in L^1(X, \mu_J)$ and $\log \iterl{v}_J \in L^1(Y, \iterl{\nu}_J)$.
		\item
		\label{item:ScalingLemma:Ratio}
		The ratios of the $\iterl{u}_J$-scaling factor are bounded,
		\begin{equation}\label{eq:ScalingRatioBound}
		\frac{\iterl{u_J}(x_1)}{\iterl{u_J}(x_2)} \leq \exp(\|c\|/\veps) \quad \text{for $\mu_J$-a.e.~$x_1, x_2 \in X_J$}.
		\end{equation}
		\item
		\label{item:ScalingLemma:Basic}
		Appropriate versions of \eqref{item:ScalingLemma:Bounds} and \eqref{item:ScalingLemma:Ratio} hold for $\iterl{u}_i$, $\iterl{v}_i$ via their definition, \eqref{eq:BasicScalings}.
		\item \label{item:DensityAbsBound:Abs}
		The product of the scaling factors $\iterl{u}_i$, $\iterl{v}_i$ is locally controlled by
		\begin{align}
		\label{eq:DensityAbsBound}
		\RadNik{\iterl{\nu}_i}{\nu}(y) \cdot \frac{\exp(-\|c\|/\veps)}{\|\mu_i\|} \leq \iterl{u}_i(x) \cdot \iterl{v}_i(y) & \leq \frac{\exp(2\|c\|/\veps)}{\|\mu_i\|}
		\end{align}
		for $\mu_i$-a.e.~$x \in X_i$ and $\nu$-a.e.~$y \in Y$.
		\item \label{item:DensityAbsBound:Rel}
		For any $j \in J$, one has
		\begin{align}\label{eq:DensityBound}
		\iterl{u}_{i}(x) \cdot \iterl{v}_{i}(y)
		\geq \exp(-3\|c\|/\veps)\,\|\mu_{j}\|
		\cdot
		\iterlm{u}_{j}(x') \cdot \iterlm{v}_{j}(y)
		\end{align}
		for $(\mu_i \otimes \mu_{j} \otimes \nu)$-a.e.~$(x,x',y) \in X_i \times X_j \times Y$, and
		\begin{equation}\label{eq:RadNikLowerBound1Step}
		\RadNik{\iterl{\nu}_i}{\nu}(y) \geq \exp(-2\|c\|/\veps)\,\|\mu_i\| \cdot \RadNik{\iterlm{\nu}_{j}}{\nu}(y) \quad \text{for $\nu$-a.e.~$y \in Y$.}
		\end{equation}
		\end{enumerate}
\end{lemma}

To provide some intuition, we now sketch a simple proof that Algorithm \ref{alg:DomDec} converges to the globally optimal solution. For simplicity, we restrict it to the three-cell problem and discrete (and finite) spaces $X$ and $Y$. While it could be extended to the general setting, we instead refer to Section \ref{sec:Convergence}.

\begin{proposition}
	\label{prop:SimpleConvergence}
	Let $X$ and $Y$ be finite sets. Let $\{X_1,X_2,X_3\}$, $I=\{1,2,3\}$ and $\partA=\{\{1,2\},\{3\}\}$ and $\partB=\{\{1\},\{2,3\}\}$ be a basic and two composite partitions of $X$. Then Algorithm \ref{alg:DomDec} converges to the globally optimal solution of problem \eqref{eq:Problem}.
\end{proposition}
\begin{proof}[Sketch]
When $X$ and $Y$ are discrete, the spaces of measures and measurable functions all become finite-dimensional real vector spaces.
For simplicity, we may assume that $\mu$ and $\nu$ assign strictly positive mass to every point in $X$ and $Y$. Then $\kernel$ assigns strictly positive mass to every point in $X \times Y$ and therefore $\pi \ll \kernel$ for any $\pi \in \measp(X \times Y)$. The function $\KL(\cdot|\kernel)$ is therefore finite and continuous on $\measp(X \times Y)$.
Denote by $S$ the `solving' map $\measp(X \times Y) \ni \pi \mapsto \arg\min \{ \veps\,\KL(\hat{\pi}|\kernel) | \hat{\pi} \in \Pi(\proj_X \pi,\proj_Y \pi)\}$ which is well-defined.

We show that $S$ is continuous: Let $(\pi_n)_n$ be a sequence of measures in $\measp(X \times Y)$ converging to $\pi_\infty$ (this implies that the entries of $\pi_n$ are uniformly bounded), then $S(\pi_n)=u_n \otimes v_n \cdot \kernel$ for suitable $u_n$, $v_n$ by Proposition \ref{prop:EntropicOTBasic}.
By re-scaling, with (\ref{eq:uBound},\ref{eq:vBound}) we may assume that $u_n$ and $v_n$ are uniformly bounded and by compactness in finite dimensions, we may extract some cluster points $u_\infty$ and $v_\infty$.
Since $u_n \otimes v_n \cdot \kernel \in \Pi(\proj_X \pi_n, \proj_Y \pi_n)$, one finds that $u_\infty \otimes v_\infty \cdot \kernel \in \Pi(\proj_X \pi_\infty,\proj_Y \pi_\infty)\}$, which must be equal to $S(\pi_\infty)$ by Proposition \ref{prop:Duality} \eqref{item:Duality:OptimalityCondition}. Therefore, $S(\pi_n)$ converges to $S(\pi_\infty)$ and $S$ is continuous.

Let now $F_A$ be the map that takes an iterate $\iterlm{\pi}$ to $\iterl{\pi}$ when $\ell$ is odd, and $F_B$ the map for $\ell$ even.
In the discrete setting, the restriction of a measure to a subset is a continuous map, and since $F_A$ and $F_B$ are built from measure restriction and the solving map $S$, they are continuous.

Consider now the sequence of odd-numbered iterates $(\iter{\pi}{2\ell+1})_\ell$. In finite dimensions we may extract a cluster point $\pi^\ast$. By continuity of $F_A$ and $F_B$, $F_A(F_B(\pi^\ast))$ must also be a cluster point of $(\iter{\pi}{2\ell+1})_\ell$. Since the sequence $(\KL(\iterl{\pi}|\kernel))_\ell$ is non-increasing (and $\KL(\cdot|\kernel)$ is continuous here), we must have $\KL(\pi^\ast|\kernel)=\KL(F_A(F_B(\pi^\ast))|\kernel)$.
But $F_A(F_B(\pi^\ast)) \neq \pi^\ast$ would imply that the $\KL$ divergence has been strictly decreased, so we must have that $F_A(F_B(\pi^\ast))=F_B(\pi^\ast)=\pi^\ast$, i.e.~$\pi^\ast$ is partially optimal on all the cells of the composite partitions $\partA$ and $\partB$. So there are scaling factors $(u_{\{1,2\}},v_{\{1,2\}})$, $(u_{\{3\}},v_{\{3\}})$ and $(u_{\{1\}},v_{\{1\}})$, $(u_{\{2,3\}},v_{\{2,3\}})$ such that we can write
\begin{align*}
	\pi^\ast
	= u_{\{1,2\}} \otimes v_{\{1,2\}} \cdot \kernel_{\{1,2\}} + u_{\{3\}} \otimes v_{\{3\}} \cdot \kernel_{\{3\}}
	= u_{\{1\}} \otimes v_{\{1\}} \cdot \kernel_{\{1\}} +  u_{\{2,3\}} \otimes v_{\{2,3\}} \cdot \kernel_{\{2,3\}}.
\end{align*}
We find that on $X_2 \times Y$ the functions $u_{\{1,2\}} \otimes v_{\{1,2\}}$ and $u_{\{2,3\}} \otimes v_{\{2,3\}}$ must coincide $\kernel$-a.e., and by re-scaling we may thus assume that $u_{\{1,2\}}=u_{\{2,3\}}$ on $X_2$ $\mu$-a.e., and $v_{\{1,2\}}=v_{\{2,3\}}$ on $Y$ $\nu$-a.e.~and therefore we can `glue' together $u_{\{1,2\}}$ and $u_{\{2,3\}}$ to get some $u$ and set $v=v_{\{1,2\}}=v_{\{2,3\}}$ such that $\pi^\ast=u \otimes v \cdot \kernel$, which must therefore be globally optimal by Proposition \ref{prop:Duality} \eqref{item:Duality:OptimalityCondition}.
Since this holds for any cluster point $\pi^\ast$, $(\iterl{\pi})_\ell$ converges to this point.
\end{proof}
\section{Linear convergence}
\label{sec:Convergence}
In this Section we prove linear convergence of Algorithm \ref{alg:DomDec} to the globally optimal solution of \eqref{eq:Problem} with respect to the Kullback--Leibler divergence.
We start with the proof for the three-cell setup in Section \ref{sec:ConvergenceSimple} and finally the proof for general decompositions in Section \ref{sec:ConvergenceGeneral}.
Numerical worst-case examples that demonstrate the qualitative accuracy of our convergence rate bounds are presented in Section \ref{sec:NumericsWorstCase}.

\begin{definition}[Primal suboptimality]
	For $\pi \in \Pi(\mu,\nu)$ the (scaled) suboptimality of $\pi$ for problem \eqref{eq:Problem} is denoted by
	\begin{align}
	\label{eq:Suboptimality}
	\Delta(\pi) & \assign \KL(\pi|\kernel)-\KL(\pi^\ast|\kernel)
	\end{align}
	where we recall that $\pi^\ast$ denotes the unique minimizer.
\end{definition}

\begin{remark}[Proof strategy]
	\label{rem:ProofStrategyThree}
	The basic strategy for the proofs is inspired by \cite{LuTs92} where linear convergence for block coordinate descent methods is shown. Indeed, the domain decomposition algorithm can be interpreted as a block coordinate descent for the entropic optimal transport problem, where in each iteration we add some additional artificial constraints that fix the $Y$-marginals of the current candidate coupling on each of the composite cells. The fundamental strategy of \cite{LuTs92} is to find a bound $\Delta(\iterl{\pi}) \leq C \left( \Delta(\iter{\pi}{\ell-k})-\Delta(\iter{\pi}{\ell})\right)$ for some $C \in \R_+$ where $k=1$ in the three-cell setting and $k>1$ in the general setting. This then quickly implies linear convergence of $\Delta(\iterl{\pi})$ to $0$.
	In this article, we exploit the identity $\Delta(\iter{\pi}{\ell-k})-\Delta(\iter{\pi}{\ell})=\KL(\iter{\pi}{\ell-k}|\iter{\pi}{\ell})$ (Lemma \ref{lem:SubOptKL}) and we use a primal-dual gap estimate (Proposition \ref{prop:Duality} \eqref{item:Duality:Gap}) to bound $\Delta(\iterl{\pi}) \leq \KL(\iterl{\pi}|\hat{\pi})$ for a suitable $\hat{\pi}$ of the form $\hat{u} \otimes \hat{v} \cdot \kernel$ (Lemma \ref{lem:PDGap}) and then use carefully chosen $\hat{u}$ and $\hat{v}$ to control the latter by the former.

	A major challenge is that Algorithm \ref{alg:DomDec} does not produce any `full' dual iterates. We may only use the dual variables for the composite cell problems in line \ref{line:CompositePi}, but they are not necessarily consistent across multiple cells. Inspired by the proof of Proposition \ref{prop:SimpleConvergence} we propose a way to approximately `glue' them together to obtain a global candidate.
\end{remark}

\subsection{Three cells}
\label{sec:ConvergenceSimple}
We first state the main result of this Section.
\begin{theorem}[Linear convergence for three-cell decomposition]
	\label{thm:ThreeConvergence}
	Let $(X_1,X_2,\allowbreak X_3)$, $I=\{1,2,3\}$, be a basic partition of $(X,\mu)$ into three cells and set $\partA=\{\{1,2\},\{3\}\}$, $\partB=\{\{1\},\{2,3\}\}$.
	Then, for iterates of Algorithm \ref{alg:DomDec} one has for $\ell > 1$,
	\begin{align}
	\label{eq:ThreeConvergenceKLDecrement}
	\Delta(\iterl{\pi}) \leq \left(1+\exp\left(-\tfrac{2\|c\|}{\veps}\right) \tfrac{\|\mu_2\|}{\|\mu_{\{1,3\}}\|}\right)^{-1}
	\cdot \Delta(\iterlm{\pi}).
	\end{align}
	In particular, the domain decomposition algorithm converges to the optimal solution.
\end{theorem}

The proof for Theorem \ref{thm:ThreeConvergence} requires some some auxiliary Lemmas.
Most of the Lemmas are already formulated for more general partitions to allow reusing them in the next Section.
The Lemmas strongly rely on algebraic properties of the $\KL$ divergence that are used in a similar way throughout the literature, see for example \cite[Lemma 2]{Greenkhorn2017}.

\begin{lemma}
	\label{lem:ScalingKL}
	Let $\pi=(u \otimes v) \cdot \kernel$ for $u \in L^\infty_+(X,\mu)$, $v \in L^\infty_+(Y,\nu)$. Then
	\begin{align}
	\KL(\pi|\kernel) & =
	\la \proj_X \pi, \log u \ra	 + \la \proj_Y \pi, \log v \ra
	- \|\pi\| + \|\kernel\|\,.
	\end{align}
	Similarly, given $\pi_i = (u_i \otimes v_i) \cdot \kernel_i$ for $u_i \in L^\infty_+(X,\mu)$, $v_i \in L^\infty_+(Y,\nu)$, $i \in I$, and $\pi = \sum_{i\in I} \pi_i$, then
	\begin{align}
	\KL(\pi|\kernel) & = \sum_{i\in I} \left[
	\la \proj_X \pi_i, \log u_i \ra	 + \la \proj_Y \pi_i, \log v_i \ra \right]
	- \|\pi\| + \|\kernel\|.
	\end{align}
\end{lemma}
We briefly recall from Section \ref{sec:Notation} at this point that $\la \rho, f \ra$ denotes integration of the measurable function $f$ against the measure $\rho$.
\begin{proof}
	We show the second statement. The first then follows by setting $(X_{i})_{i \in I},I$ to the trivial partition $(X_1\assign X)$, $I\assign\{1\}$.
	Arguing as in Proposition \ref{prop:EntropicOTBasic} \eqref{item:EntropicOTBasic:Bounds} we find $\log u_i \in L^1(X,\proj_X \pi_i)$, $\log v_i \in L^1(Y,\proj_Y \pi_i)$ and therefore, all integrals in the following are finite.
	The proof now follows quickly from direct computation:
	\begin{align*}
	\KL(\pi|\kernel) &
	= \sum_{i \in I} \KL(\pi_i|\kernel_i)
	= \sum_{i \in I} \int_{X \times Y} \varphi\left(\RadNik{\pi_i}{\kernel_i}\right)\diff\kernel_i \\
	& = \sum_{i \in I} \left[ \int_{X \times Y} \log\left(u_i(x)\,v_i(y)\right)\,\diff \pi_i(x,y) 
	- \pi_i(X \times Y)+\kernel_i(X \times Y) \right] \\
	& = \sum_{i \in I} \left[ \int_{X \times Y} \big[ \log(u_i(x))+ \log(v_i(y)) \big] \diff \pi_i(x,y) \right]
	- \|\pi\|+\|\kernel\| \\
	& = \sum_{i\in I} \left[\la \proj_X \pi_i, \log u_i \ra + \la \proj_Y \pi_i, \log v_i \ra \right]
	- \|\pi\| + \|\kernel\|
	\qedhere
	\end{align*}
\end{proof}

\begin{lemma}[Expressions for decrement and sub-optimality]
	\label{lem:SubOptKL}
	Let $\ell > 1$. Then,
	\[
	\Delta(\iterlm{\pi})-\Delta(\iterl{\pi}) = \KL(\iterlm{\pi}|\kernel)-\KL(\iterl{\pi}|\kernel) = \KL(\iterlm{\pi}|\iterl{\pi})
	\]
	and
	\[
	\Delta(\iterl{\pi}) = \KL(\iterl{\pi}|\kernel)-\KL(\pi^\ast|\kernel) = \KL(\iterl{\pi}|\pi^\ast).
	\]
\end{lemma}
\begin{proof}
	Fix $\ell > 1$. With the notation of Definition \ref{def:Notation} we can decompose the current and previous iterate as
	\begin{align*}
	\iterlm{\pi} & = \sum_{i \in I} \iterlm{u}_i \otimes \iterlm{v}_i \cdot \kernel_i, &
	\iterl{\pi} & = \sum_{i \in I} \iterl{u}_i \otimes \iterl{v}_i \cdot \kernel_i,
	\end{align*}
	with all partial scalings being essentially bounded with respect to $\mu$ or $\nu$.
	So, by using Lemma \ref{lem:ScalingKL}, we compute
	\begin{align}
	& \KL(\iterlm{\pi}|\kernel) - \KL(\iterl{\pi}|\kernel)
	= \sum_{J \in \iterl{\partGeneric}} \sum_{i \in J} \Big[\KL(\iterlm{\pi_i}|\kernel_i) - \KL(\iterl{\pi_i}|\kernel_i) \Big] \nonumber \\
	&= \sum_{J \in \iterl{\partGeneric}} \sum_{i \in J} \Big[ \la \proj_X \iterlm{\pi}_i, \log \iterlm{u}_i \ra
	+\la \proj_Y \iterlm{\pi}_i, \log \iterlm{v}_i \ra \nonumber \\
	&
	\qquad\qquad\qquad - \la \proj_X \iterl{\pi}_i, \log \iterl{u}_i \ra
	- \la \proj_Y \iterl{\pi}_i, \log \iterl{v}_i \ra
	-\|\iterlm{\pi}_i\| + \|\iterl{\pi}_i\| \Big] \nonumber \\
	&= \sum_{J \in \iterl{\partGeneric}} \sum_{i \in J} \Big[ \la \mu_i, \log \iterlm{u}_i \ra
	+\la \iterlm{\nu}_i, \log \iterlm{v}_i \ra \nonumber \\
	&
	\qquad\qquad\qquad - \la \mu_i, \log \iterl{u}_i \ra
	- \la \iterl{\nu}_i, \log \iterl{v}_i \ra
	-\|\iterlm{\pi}_i\| + \|\iterl{\pi}_i\| \Big].
	\label{eq:KLIncrementalProofA}
	\end{align}
	Note now that, for any $J \in \iterl{\partGeneric}$, the $Y$-marginal on each composite cell is kept fixed during the iteration (see Algorithm \ref{alg:DomDec}, lines \ref*{line:YMarginal}-\ref*{line:CompositePi}), i.e.,
	\begin{equation*}
	\sum_{i \in J} \iterl{\nu}_i = \proj_Y \iterl{\pi}_J = \iterl{\nu}_J = \proj_Y \iterlm{\pi}_J = \sum_{i \in J} \iterlm{\nu}_i
	\end{equation*}
	Hence, since $\iterl{v}_{i} = \iterl{v}_{J} = \iterl{v}_{j}$ for all $i, j \in J$ (see \eqref{eq:BasicScalings}), we have
	\begin{align*}
	\sum_{i \in J} \la \iterl{\nu}_i, \log \iterl{v}_i \ra =
	\sum_{i \in J} \la \iterlm{\nu}_i, \log \iterl{v}_i \ra.
	\end{align*}
	Thus, we can continue:
	\begin{align*}
	\eqref{eq:KLIncrementalProofA} & = 
	\sum_{J \in \iterl{\partGeneric}} \sum_{i \in J} \left[ \la {\mu}_i, \log\left(\tfrac{\iterlm{u}_i}{\iterl{u}_i}\right) \ra
	+\la \iterlm{\nu}_i, \log \left(\tfrac{\iterlm{v}_i}{\iterl{v}_i}\right) \ra
	-\|\iterlm{\pi}_i\| + \|\iterl{\pi}_i\| \right] \\
	& = \sum_{i \in I} \KL(\iterlm{\pi}_i|\iterl{\pi}_i) = \KL(\iterlm{\pi}|\iterl{\pi}).
	\end{align*}
	The proof of the second statement follows the same steps. It hinges on the fact that $\proj_X \iterlm{\pi}_i=\mu_i=\proj_X\pi^\ast_i$, all $v^\ast_{i}=v^\ast$ are identical, and $\sum_{i \in I} \proj_Y \iterlm{\pi}_i = \nu = \sum_{i \in I} \proj_Y \pi^\ast_i$ so that one has
	\begin{equation*}
	\sum_{i \in I} \la \proj_Y \pi^\ast_i, \log v^\ast_i \ra 
	=\sum_{i \in I} \la \proj_Y \pi^\ast_i, \log v^\ast \ra
	=\sum_{i \in I} \la \proj_Y \iterlm{\pi}_i, \log v^\ast \ra
	=\sum_{i \in I} \la \proj_Y \iterlm{\pi}_i, \log v^\ast_i \ra.
	\qedhere
	\end{equation*}
\end{proof}

\begin{lemma}[Primal-dual gap]
	\label{lem:PDGap}
	Let
	\begin{equation*}
		\tilde{\pi} = \sum_{i \in I} \tilde{\pi}_i, \quad
		\tilde{\pi}_i = (\tilde{u}_i \otimes \tilde{v}_i) \otimes \kernel_i, \quad
		\tilde{u}_i \in L^\infty_+(X,\mu), \quad
		\tilde{v}_i \in L^\infty_+(Y,\nu)
		\quad \text{for} \quad i \in I
	\end{equation*}
	and
	$\hat{\pi} = (\hat{u} \otimes \hat{v}) \cdot \kernel$ with $\hat{u} \in L^\infty_+(X,\mu)$, $\hat{v} \in L^\infty_+(Y,\nu)$.
	If $\tilde{\pi} \in \Pi(\mu, \nu)$, then
	\begin{equation*}
	\Delta(\tilde{\pi}) \leq \KL(\tilde{\pi}|\hat{\pi}).
	\end{equation*}
\end{lemma}
\begin{proof}
	If $\KL(\tilde{\pi}|\hat{\pi})=+\infty$ there is nothing to prove. Therefore, assume $\KL(\tilde{\pi}|\hat{\pi})<+\infty$.
	Abbreviating $\hat{\pi}_i\assign \hat{\pi} \restr (X_i \times Y)$ for $i \in I$, we find
	\begin{align*}
		\KL(\tilde{\pi}|\hat{\pi}) &= \sum_{i \in I} \KL(\tilde{\pi}_i|\hat{\pi}_i)
		= \sum_{i\in I} \int_{X \times Y} \varphi\left(\tfrac{\tilde{u}_i \otimes \tilde{v}_i}{\hat{u} \otimes \hat{v}}\right)\,\diff \hat{\pi}_i \\
		& = \sum_{i \in I} \left[
		\int_{X \times Y} \log\left(\tilde{u}_i \otimes \tilde{v}_i\right)\,\diff \tilde{\pi}_i
		- \int_{X \times Y} \log\left(\hat{u} \otimes \hat{v}\right)\,\diff \tilde{\pi}_i
		- \|\tilde{\pi}_i\| + \|\hat{\pi}_i\|\right] \\
		& = \KL(\tilde{\pi}|\kernel) - \int_X \log \hat{u}\,\diff \mu - \int_Y \log \hat{v}\,\diff \nu + \|\hat{\pi}\| - \|\kernel\| \\
		& = \KL(\tilde{\pi}|\kernel) - J(\hat{u},\hat{v}) \geq \Delta(\tilde{\pi}).
	\end{align*}
	Here the integral in the first line is finite since $\KL(\tilde{\pi}_i|\hat{\pi}_i)\leq \KL(\tilde{\pi}|\hat{\pi})<\infty$. The first integral in the second line is finite since $\log\left(\tilde{u}_i \otimes \tilde{v}_i\right) \in L^1(X \times Y,\tilde{\pi}_i)$, arguing as in Proposition \ref{prop:EntropicOTBasic} \eqref{item:EntropicOTBasic:Bounds}, and therefore so is the second. Since all scaling factors are essentially bounded, we can split the products within the logarithms into separate integrals in the third line where we also used $\tilde{\pi} \in \Pi(\mu,\nu)$.
	In the fourth line we use the definition of $J$, \eqref{eq:EntropicOTDualObjective}. The last inequality is due to Proposition \ref{prop:Duality}, since for the primal minimizer $\pi^\ast$ we find $\KL(\pi^\ast|\kernel) \geq J(\hat{u},\hat{v})$ and the primal-dual gap is thus a bound for the suboptimality \eqref{eq:Suboptimality}.
\end{proof}

\begin{proof}[Theorem \ref{thm:ThreeConvergence}]
	Assume first that $\ell>1$ is odd, and so the current composite partition is $\partA=\{\{1,2\},\{3\}\}$.
	On one hand, using Lemma \ref{lem:SubOptKL}, one finds
	\begin{align}
	\label{eq:ThreeConvergenceProofUpperBegin}
	\Delta(\iterlm{\pi})-\Delta(\iterl{\pi}) = \KL(\iterlm{\pi}|\iterl{\pi})
	\geq \KL(\iterlm{\pi}_2|\iterl{\pi}_2).
	\end{align}
	On the other hand, with Lemma \ref{lem:PDGap}, one can bound $\Delta(\iterl{\pi}) \leq \KL(\iterl{\pi}|\hat{\pi})$ for a suitable $\hat{\pi}$ of the form $\hat{\pi}=(\hat{u} \otimes \hat{v}) \cdot \kernel$. Using a piecewise definition on the partition cells for $\hat{u}$, we set here 
	\begin{equation}\label{eq:PiHat3Cells}
	\hat{u} \assign \iterl{u}_{\{1,2\}} + q \cdot \iterlm{u}_{\{3\}}
	\qquad \text{and} \qquad
	\hat{v} \assign \iterl{v}_{\{1,2\}}
	\end{equation}
	where the factor $q \in \R_{++}$ is to be determined later (observe that $\hat{u}$ and $\hat{v}$ are upper bounded, thanks to Lemma \ref{lem:ScalingLemma}). With this choice of $\hat{\pi}$ we find that $\iterl{\pi}_{\{1,2\}}=\hat{\pi}_{\{1,2\}}$ and thus
	\begin{align}	
	\label{eq:ThreeConvergenceProofLowerBegin}
	\Delta(\iterl{\pi}) \leq \KL(\iterl{\pi}|\hat{\pi}) = \KL(\iterl{\pi}_3|\hat{\pi}_3)\,.
	\end{align}
	To complete the proof we must now find a constant $C \in \R_+$ such that
	\begin{align*}
	\KL(\iterl{\pi}_3|\hat{\pi}_3) \leq C \cdot \KL(\iterlm{\pi}_2|\iterl{\pi}_2)\,.
	\end{align*}		
	To this end we write
	\begin{align}
	&\KL(\iterlm{\pi}_2|\iterl{\pi}_2) = \int_{X \times Y} \varphi\left(\RadNikD{\iterlm{\pi}_2}{\iterl{\pi}_2}\right)\,
	\diff \iterl{\pi}_2 \nonumber \\
	&= \int_{X \times Y} \varphi\left(\frac{\iterlm{u}_2(x)\,\iterlm{v}_2(y)}{\iterl{u}_2(x)\,\iterl{v}_2(y)}\right)\,
	\frac{\iterl{u}_2(x)}{\iterlm{u}_2(x)}\iterlm{u}_2(x)\,
	\iterl{v}_2(y)\,\sKernel(x,y)\,\diff \mu_2(x)\,\diff \nu(y).
	\label{eq:ThreeConvergenceProofKLStep}
	\end{align}
	Now apply \eqref{eq:ScalingRatioBound} at iterate $(\ell-1)$ and for $J = \{2,3\}$ to obtain that for $\mu$-a.e.~$x \in X_2, x_3 \in X_3$ and for $\nu$-a.e.~$y \in Y$ it holds
	\begin{align}
	\label{eq:ThreeConvergenceProofScalingBound}
	\iterlm{u}_2(x)\,\sKernel(x,y) & \geq \iterlm{u}_3(x_3)\,\sKernel(x_3,y)\,\cdot\exp\left(-2\|c\|/\veps\right).
	\end{align}
	Further, note that for $a \geq 0$ the map $\Phi : s \mapsto \big(\varphi\big(\tfrac{a}{s}\big) \cdot s\big)$ is convex on $\R_{++}$ and thus by Jensen's inequality one has for $\iterl{\nu}_2$-a.e.~$y \in Y$ (such that $\iterl{v}_2(y)>0$ almost surely) that
	\begin{align}
	\label{eq:ThreeConvergenceProofJensen}
	\int_{X} \varphi\left(\frac{\iterlm{u}_2(x)\,\iterlm{v}_2(y)}{\iterl{u}_2(x)\,\iterl{v}_2(y)}\right)\,
	\frac{\iterl{u}_2(x)}{\iterlm{u}_2(x)}\,\diff \mu_2(x)
	\geq
	\|\mu_2\| \cdot \varphi\left(\frac{\iterlm{v}_2(y)}{q\,\iterl{v}_2(y)}\right)\,q
	\end{align}
	where we now fix the value $q \assign \fint_X \frac{\iterl{u}_2}{\iterlm{u}_2}\,\diff \mu_2$ (which is finite because of the lower and upper boundedness of $\iterl{u}_2$ and $\iterlm{u}_2$, cf.~Lemma \ref{lem:ScalingLemma} (\ref{item:ScalingLemma:Bounds},\ref{item:ScalingLemma:Basic})). We recall the notation for the normalized integral from \eqref{eq:fint}.
	Plugging first \eqref{eq:ThreeConvergenceProofScalingBound} and then \eqref{eq:ThreeConvergenceProofJensen} into \eqref{eq:ThreeConvergenceProofKLStep} one obtains for $\mu$-almost all $x_3 \in X_3$
	\begin{equation}
	\KL(\iterlm{\pi}_2|\iterl{\pi}_2)
	\geq e^{-\frac{2\|c\|}{\veps}} \cdot \|\mu_2\|
	\cdot \int_{Y} \varphi\left(\frac{\iterlm{v}_2(y)}{q\,\iterl{v}_2(y)}\right)\,
	q\,\iterlm{u}_3(x_3)\,\iterl{v}_2(y)\,\sKernel(x_3,y)\,\diff \nu(y).
	\label{eq:ThreeConvergenceProofMiddle}
	\end{equation}
	Now we work backwards from \eqref{eq:ThreeConvergenceProofLowerBegin}. With the specified choices for $\hat{u}$ and $\hat{v}$ made in \eqref{eq:PiHat3Cells}, one finds for $\hat{\pi}$-a.e.~$(x,y) \in X_3 \times Y$
	\begin{align*}
	\RadNikD{\iterl{\pi}_3}{\hat{\pi}}(x,y) &
	= \frac{\iterl{u}_3(x) \cdot \iterl{v}_3(y)}{q \cdot \iterlm{u}_3(x) \cdot \iterl{v}_{\{1,2\}}(y)}
	= \frac{\iterlm{u}_3(x) \cdot \iterlm{v}_3(y)}{q \cdot \iterlm{u}_3(x) \cdot \iterl{v}_2(y)}
	= \frac{\iterlm{v}_2(y)}{q \cdot \iterl{v}_2(y)}
	\end{align*}
	where we have used that $\iterl{u}_3 \otimes \iterl{v}_3= \iterlm{u}_3 \otimes \iterlm{v}_3$ (since $\iterl{\pi}_3=\iterlm{\pi}_3$ due to $\{3\} \in \partA$) and $\iterlm{v}_3=\iterlm{v}_2$ (since $\{2,3\} \in \partB$).
	Consequently, with \eqref{eq:ThreeConvergenceProofMiddle}
	\begin{align*}
	\KL(\iterl{\pi}_3|\hat{\pi}_3) & = \int_{X \times Y} \varphi\left(\frac{\iterlm{v}_2(y)}{q \cdot \iterl{v}_2(y)}\right)
	q\,
	\iterlm{u}_3(x)\,\iterl{v}_2(y)\,\sKernel(x,y)\,\diff \mu_3(x)
	\diff \nu(y) \\
	& \leq e^{\frac{2\|c\|}{\veps}} \cdot \frac{1}{\|\mu_2\|} \int_{X}
	\diff \mu_3(x) \,\KL(\iterlm{\pi}_2|\iterl{\pi}_2)
	\end{align*}
	and with \eqref{eq:ThreeConvergenceProofUpperBegin} and \eqref{eq:ThreeConvergenceProofLowerBegin}
	\begin{align*}
	\Delta(\iterl{\pi}) & \leq \exp\left(\tfrac{2\|c\|}{\veps}\right) \cdot \frac{\|\mu_3\|}{\|\mu_2\|} \left(\Delta(\iterlm{\pi})-\Delta(\iterl{\pi})\right).
	\end{align*}
	For $\ell$ even one proceed analogously to obtain
	\begin{align*}
	\Delta(\iterl{\pi}) & \leq \exp\left(\tfrac{2\|c\|}{\veps}\right) \cdot \frac{\|\mu_1\|}{\|\mu_2\|} \left(\Delta(\iterlm{\pi})-\Delta(\iterl{\pi})\right).
	\end{align*}
	Hence, for any $\ell > 0$, using $\|\mu_{\{1,3\}}\| \geq \max\{\|\mu_1\|,\|\mu_3\|\}$, one obtains
	\begin{align*}
	\Delta(\iterl{\pi}) \leq \exp\left(\tfrac{2\|c\|}{\veps}\right) \tfrac{\|\mu_{\{1,3\}}\|}{\|\mu_2\|}
	\cdot \big(\Delta(\iterlm{\pi})-\Delta(\iterl{\pi})\big)
	\end{align*}
	which implies \eqref{eq:ThreeConvergenceKLDecrement}. Eventually, the sequence $(\iterl{\pi})_\ell$ is bounded and each element lies in $\Pi(\mu,\nu)$. As $\Pi(\mu,\nu)$ is weakly-$\ast$ closed, $(\iterl{\pi})_\ell$ must have a cluster point $\iter{\pi}{\infty} \in \Pi(\mu,\nu)$.
	By \eqref{eq:ThreeConvergenceKLDecrement} and since $\Delta$ is weakly-$\ast$ lower semi-continuous one has
	\begin{align*}
	\Delta(\iter{\pi}{\infty}) \leq \liminf_{\ell \to \infty} \Delta(\iterl{\pi})=0\,.
	\end{align*}
	Since $[\Delta(\pi)=\KL(\pi|\piOpt)=0]$ $\Leftrightarrow$ $[\pi=\piOpt]$ one has $\iter{\pi}{\infty} = \piOpt$.
	This also implies that \emph{every} cluster point of $(\iterl{\pi})_\ell$ must equal $\piOpt$ and that thus the sequence is converging.
\end{proof}

\subsection{General decompositions}
\label{sec:ConvergenceGeneral}
Now let $I$ be a basic partition of $(X,\mu)$ into $N$ basic cells and let $\partA$ and $\partB$ be two composite partitions of $I$.
In the special case of three cells discussed in the previous Section, convergence of the algorithm was driven by the overlap of the two partitions on the middle cell.
In the general case this overlap structure will be captured by the \emph{partition graph} and related auxiliary objects that we introduce now.

\begin{definition}[Partition graph, cell distance and shortest paths]\hfill
	\label{def:PartitionGraph}
	\begin{enumerate}[(i)]
	\item The partition graph is given by the vertex set $I$ (i.e.~each basic cell is represented by one vertex) and the edge set
	\begin{align*}
	E \assign \left\{ \vphantom{\sum} (i,j) \in I \,\middle|\, \exists\, J \in \partA \cup \partB \tn{ such that } \{i,j\} \subset J \right\},
	\end{align*}
	i.e.~there is an edge between two basic cells if they are part of the same composite cell in either of the two composite partitions $\partA$ or $\partB$.
	\label{item:PartitionGraph}
	\item We assume that no two basic cells are simultaneously contained in a composite cell of $\partA$ and $\partB$, i.e.~there exist no $i,j \in I$, $J_A \in \partA$, $J_B \in \partB$ such that $i,j \in J_A$ and $i,j \in J_B$ (otherwise $i$ and $j$ should be merged into one basic cell). This means that every edge in the partition graph is associated to precisely one of the two composite partitions.
	\item We assume that the partition graph is connected.
	\label{item:Connected}
	\item We denote by $\dist : I \times I \to \N_0$ the discrete metric on this graph induced by shortest paths, where each edge has length 1.
	\label{item:GraphMetric}
	\item Let $J_0 \in \partA$ be a selected composite cell. We define the \emph{cell distance}
	\begin{align*}
	D : I & \to \N_0,\quad  i \mapsto \dist(i,J_0) = \min\{ \dist(i,j) \,|\, j \in J_0\}
	\end{align*}
	and write $M \assign \max\{D(i) \,|\, i \in I\}$.
	\label{item:CellDistance}
	\item For each $i \in I$ we select one shortest path in the graph from $J_0$ to $i$, which we represent by a tuple $(n_{i,k})_{k=0}^{D(i)}$ of elements in $I$ with $n_{i,0} \in J_0$, $n_{i,D(i)}=i$, and $D(n_{i,k})=k$ for $k \in \{0,\ldots,D(i)\}$.
We refer to Fig.~\ref{fig:PartitionGraphD} for an illustration of such a construction. One can easily verify that, within any shortest path, the basic cells $n_{i,k}$ and $n_{i,k-1}$ belong to a common composite cell of partition $\partA$ when $k$ is even, and $\partB$ when $k$ is odd, or equivalently, of $\iter{\partGeneric}{\ell-k}$ for some odd $\ell> M$, see Fig.~\ref{fig:PartitionGraphIterations}.
	\label{item:ShortestPaths}
	\end{enumerate}
\end{definition}

Let us fix for the rest of this Section an odd iterate $\ell$ so that $\iterl{\partGeneric}=\partA$ (the same will then apply for $\ell$ even by swapping the roles of $\partA$ and $\partB$).

\begin{figure}[t]
	\centering
	\includegraphics[width=7cm]{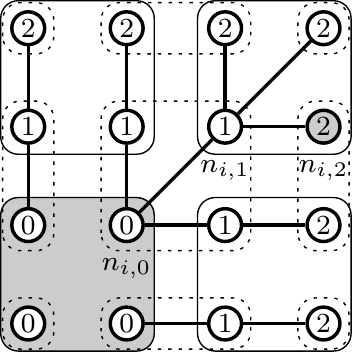}
	\caption{Illustration of partition graph and cell distance for a typical 2D setup. Vertices, the elements of $I$, are represented by circles, with cell distance $D(i)$ given as labels. Composite cells of $\partA$ are indicated by rectangles with solid lines, $J_0$ is highlighted with grey filling. Composite cells of $\partB$ are indicated by rectangles with dashed lines. Shortest paths from vertices to $J_0$ are indicated by black lines (these are often not unique).
		For the vertex highlighted in grey the indices $(n_{i,k})_{k=0}^{D(i)}$ are given below the corresponding nodes.}
	\label{fig:PartitionGraphD}
\end{figure}

\begin{figure}[t]
	\centering
	\includegraphics[height=6cm]{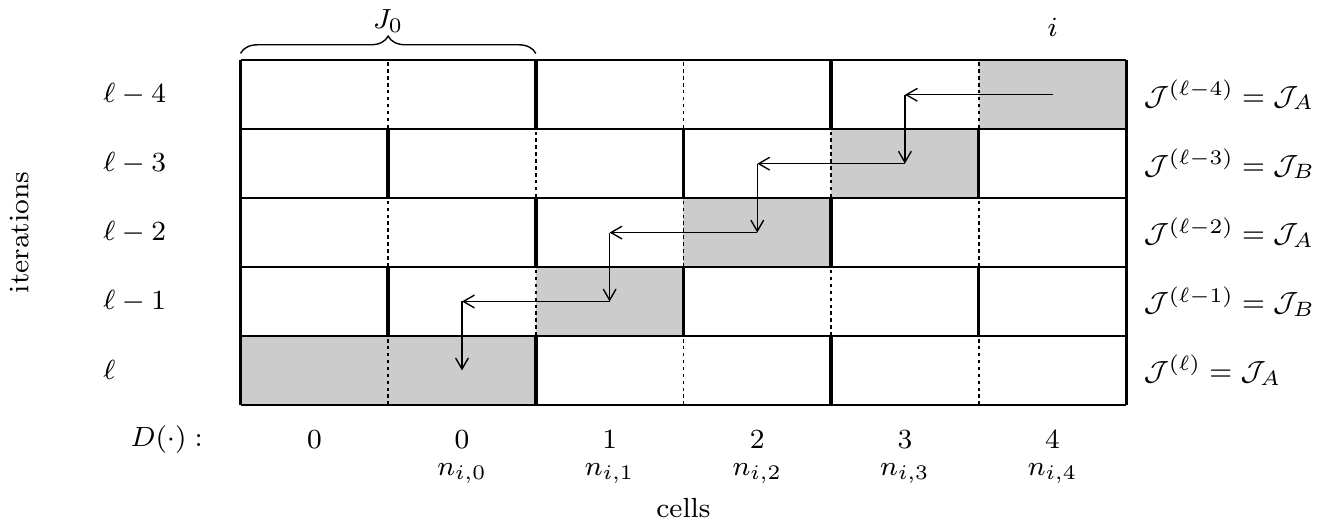}
	\caption{Illustration of proof strategies for Lemma \ref{lem:PrimalCandidate} and Theorem \ref{thm:NCellConvergence}. %
		The horizontal axis represents basic cells along the shortest path $(n_{i,k})_{k=0}^{D(i)}$ from  $J_0 \in \partA$ (left side) to some $i \in I$ (right side), vertical axis represents iterations for some odd $\ell>M$ (i.e.~$\iterl{\partGeneric}=\partA$). Solid vertical lines represent boundaries of composite cells at given iteration. %
		Grey shading of boxes illustrates the construction of $\tilde{\pi}$ in \eqref{eq:PiCandidate}. %
		Vertical arrows represents factors $q_{i,k}$, cf.~\eqref{eq:NCellProofUHat}. %
		Horizontal arrows represent edges in the partition graph and the equalities $\iter{v_{n_{i,k}}}{\ell-k}=\iter{v_{n_{i,k-1}}}{\ell-k}$ used in \eqref{eq:NCellProofDensityDecomposition}.}
	\label{fig:PartitionGraphIterations}
\end{figure}

\begin{theorem}[Linear convergence]
	\label{thm:NCellConvergence}
	Let $\ell > M$, with $M$ defined in Def.~\ref{def:PartitionGraph} \eqref{item:CellDistance}. Then, one finds
	\begin{align}
	\label{eq:NCellConvergenceKLDecrement}
	\Delta(\iterl{\pi}) \leq \frac{2MN\exp((6M+7)\,\|c\|/\veps)}{\muMin^{2M+1}+2MN\exp((6M+7)\,\|c\|/\veps)}
	\cdot \Delta(\iter{\pi}{\ell-M}).
	\end{align}
	where $\muMin \assign \min\{\|\mu_i\| | i \in I\}$. In particular, the domain decomposition algorithm converges to the optimal solution.
\end{theorem}

\begin{remark}[Proof strategy]
For the three-cell case the proof relied on the fact that $\iterl{\pi}_3$ remained unchanged in an $A$-iteration, since $\{3\} \in \partA$ was a composite cell containing only basic cell $3$. This is no longer true for general decompositions. To overcome this, we need to replace parts of the primal iterate $\iterl{\pi}$ by partial transport plans from previous iterations (Lemma \ref{lem:PrimalCandidate}).
When bounding the sub-optimality by the decrement (cf.~Remark \ref{rem:ProofStrategyThree}) we also need an approximate triangle inequality for the $\KL$-divergence (Lemma \ref{lem:PhiProductDecomposition}).
\end{remark}

\begin{lemma}[Construction of primal upper bound candidate]
	\label{lem:PrimalCandidate}
	For $\ell > M$, $\ell$ odd, the measure
	\begin{align}
	\label{eq:PiCandidate}
	\tilde{\pi} & = \sum_{i \in I} \iter{\pi_i}{\ell-D(i)}
	\end{align}
	satisfies $\tilde{\pi} \in \Pi(\mu,\nu)$ and $\KL(\tilde{\pi}|\pi^\ast) \geq \KL(\iterl{\pi}|\pi^\ast)$.
\end{lemma}

\begin{proof}
	In the proof we assume that $\ell$ is odd (i.e.~an $A$-iteration was just completed). The proof for $\ell$ even works completely analogous by swapping the roles of $J_A$ and $J_B$.
	
	We start with the claim that $\tilde{\pi} \in \Pi(\mu,\nu)$. Since $\tilde{\pi}$ is a sum of non-negative measures it is non-negative. By construction of the algorithm one always has $\proj_X \iter{\pi_i}{m}=\mu_i$ for all iterations $m \in \N_0$ and thus
	\begin{align*}
	\proj_X \tilde{\pi} = \sum_{i \in I} \proj_X  \iter{\pi_i}{\ell-D(i)} = \sum_{i \in I} \mu_i=\mu\,.
	\end{align*}
	For $k \in \{0,\ldots,M\}$ we introduce the auxiliary measures
	\begin{align*}
	\tilde{\pi}^k & = \sum_{i \in I} \iter{\pi_i}{\ell - \min\{D(i),k\}}\,.
	\end{align*}
	Note that $\tilde{\pi}^0=\iterl{\pi}$ and $\tilde{\pi}^M=\tilde{\pi}$.
	Further, one obtains for $k \in \{1,\ldots,M\}$
	\begin{align*}
	\tilde{\pi}_i^k - \tilde{\pi}_i^{k-1} & = \begin{cases}
	\iter{\pi_i}{\ell-k}-\iter{\pi_i}{\ell-k+1} & \tn{if } D(i) \geq k,\\
	0 & \tn{else.}
	\end{cases} \\
	\tilde{\pi}^k-\tilde{\pi}^{k-1} & = \sum_{\substack{i \in I:\\D(i)\geq k}}
	\left(\iter{\pi_i}{\ell-k}-\iter{\pi_i}{\ell-k+1}\right)
	\end{align*}
	Assume now that $k$ is odd. Recalling that $\ell$ is assumed to be odd, then the set $J^k \assign \{i \in I\,:\,D(i) \geq k\}$ is a union of $A$-cells and the iteration step from $\iter{\pi}{\ell-k}$ to $\iter{\pi}{\ell-k+1}$ is an iteration on $A$-cells (since $\ell-k+1$ is odd, recall also Definition \ref{def:PartitionGraph}), i.e.,~for any $J \in \partA$ one has
	\begin{align*}
	\proj_Y \iter{\pi_J}{\ell-k} = \proj_Y \iter{\pi_J}{\ell-k+1}.
	\end{align*}
	Since $J^k$ is a union of $A$-cells we find therefore
	\begin{align*}
	\proj_Y \sum_{i \in J^k} \iter{\pi_i}{\ell-k} =
	\proj_Y \sum_{i \in J^k} \iter{\pi_i}{\ell-k+1}\,.
	\end{align*}
	Using this, we find that $\proj_Y \tilde{\pi}^k = \proj_Y \tilde{\pi}^{k-1}$ for $k \in \{1,\ldots,M\}$ and $k$ odd, and we can argue in complete analogy for even $k$ via $B$-cells. Consequently,
	\begin{align*}
	\proj_Y \tilde{\pi} = \proj_Y \tilde{\pi}^M = \proj_Y \tilde{\pi}^{M-1} = \ldots = \proj_Y \tilde{\pi}^0 = \proj_Y \iterl{\pi} = \nu
	\end{align*}
	and thus $\tilde{\pi} \in \Pi(\mu,\nu)$.
	
	Now we establish the $\KL$ bound. Again, assume that $k$ is odd, so that $J^k$ is a union of $A$-cells and the iteration step from $\iter{\pi}{\ell-k}$ to $\iter{\pi}{\ell-k+1}$ is an iteration on $A$-cells. Then for any $J \in \partA$ one finds
	\begin{align*}
	\KL(\iter{\pi_J}{\ell-k}|\pi^\ast) \geq \KL(\iter{\pi_J}{\ell-k+1}|\pi^\ast)
	\end{align*}
	and since $J^k$ is a union of $A$-cells one gets that
	\begin{align*}
	\KL(\tilde{\pi}^k|\pi^\ast)-\KL(\tilde{\pi}^{k-1}|\pi^\ast)
	&= \sum_{i \in I} \left[ \KL(\tilde{\pi}_i^k|\pi_i^\ast)-\KL(\tilde{\pi}_i^{k-1}|\pi_i^\ast) \right] \\
	&= \sum_{i \in J^k} \left[ \KL(\iter{\pi_i}{\ell-k}|\pi_i^\ast)-\KL(\iter{\pi_i}{\ell-k+1}|\pi_i^\ast) \right]
	\geq 0.
	\end{align*}
	Again, an analogous argument holds when $k$ is even. Consequently one has
	\begin{align*}
	\KL(\tilde{\pi}|\pi^\ast) = \KL(\tilde{\pi}^M|\pi^\ast)
	\geq \KL(\tilde{\pi}^{M-1}|\pi^\ast) \geq \ldots
	\geq \KL(\tilde{\pi}^0|\pi^\ast) = \KL(\iterl{\pi}|\pi^\ast).
	& \qedhere
	\end{align*}
\end{proof}

\begin{lemma}[Global density bound]
	\label{lem:DensityMarginalBound}
	Let $\ell > M$. Then for $\nu$-a.e.~$y \in Y$ and all $i \in I$ one has
	\begin{equation}\label{eq:RadNikLowerBound}
	\RadNik{\iterl{\nu}_i}{\nu}(y) \geq \frac{\muMin^{M+1}}{N}\exp\left(-2\,(M+1)\,\|c\|/\veps\right).
	\end{equation}
\end{lemma}
\begin{proof}
	Let $J \in \iterl{\partGeneric}$ such that $i \in J$. We apply \eqref{eq:RadNikLowerBound1Step} for each $j \in J$ to obtain
	\begin{align*}
	\RadNik{\iterl{\nu}_i}{\nu}(y) \geq \exp(-2\|c\|/\veps)\,\muMin \cdot \max_{j \in J} \RadNik{\iterlm{\nu}_{j}}{\nu}(y)\,.
	\end{align*}
	for $\nu$-a.e.~$y \in Y$. Applying this bound recursively, the $\max$ runs over increasingly many cells and after $M+1$ applications of the bound we find:
	\begin{align*}
	\RadNik{\iterl{\nu}_i}{\nu}(y) \geq \exp(-2\,(M+1)\,\|c\|/\veps)\,\muMin^{M+1} \cdot \max_{j \in I} \RadNik{\iter{\nu}{\ell-M-1}_{j}}{\nu}(y)\,.
	\end{align*}
	Observe now that $\sum_{j \in I} \RadNik{\iter{\nu}{\ell-M-1}_{j}}{\nu}(y)=1$ and thus the maximum must be bounded from below by $1/N$.
\end{proof}

\begin{proof}[Theorem \ref{thm:NCellConvergence}]
	Let us set
	\begin{align}
	\label{eq:NCellDualCandidate}
	\hat{\pi} \assign (\hat{u} \otimes \hat{v}) \cdot \kernel \quad \text{with} \quad
	\hat{u}_i \assign \left(\prod_{k=0}^{D(i)-1} q_{i,k}\right) \cdot \iter{u_i}{\ell-D(i)} \quad \text{and} \quad \hat{v} & \assign \iterl{v_{J_0}}
	\end{align}
	where the factors $q_{i,k}$ are defined as
	\begin{align}
	\label{eq:NCellProofUHat}
	q_{i,k} \assign \fint_X \frac{\iter{u}{\ell-k}_{n_{i,k}}}{\iter{u}{\ell-k-1}_{n_{i,k}}} \,\diff \mu_{n_{i,k}}
	\end{align}
	and are finite thanks to the lower and upper boundedness of $u$-scalings, cf.~Lemma \ref{lem:ScalingLemma}, (\ref{item:ScalingLemma:Bounds},\ref{item:ScalingLemma:Basic}).
	Thus, with Lemmas \ref{lem:SubOptKL}, \ref{lem:PDGap} and $\tilde{\pi}$ from Lemma \ref{lem:PrimalCandidate}, one obtains
	\begin{align}
	\Delta(\iterl{\pi}) = \KL(\iterl{\pi}|\pi^\ast) & \leq \KL(\tilde{\pi}|\pi^\ast) \leq \KL(\tilde{\pi}|\hat{\pi})
	= \sum_{i \in I} \KL(\tilde{\pi}_i|\hat{\pi}_i)
	\label{eq:NCellProofSubObtBasic}
	\end{align}
	where one has
	\begin{align*}
	\KL(\tilde{\pi}_i|\hat{\pi}_i) & = \int_{X_i \times Y} \varphi\left(\RadNik{\tilde{\pi}_i}{\hat{\pi}_i}\right)\,\diff \hat{\pi}_i\,.
	\end{align*}
	For $i \in J_0$ one finds $\hat{\pi}_i$ almost everywhere that
	\begin{align*}
	\RadNikD{\tilde{\pi}_i}{\hat{\pi}_i} = \frac{\tilde{u}_i \otimes \tilde{v}_i}{\hat{u}_i \otimes \hat{v}}
	= \frac{\iter{u_i}{\ell} \otimes \iter{v_i}{\ell}}
	{\iter{u_i}{\ell} \otimes \iterl{v}_{J_0}}=1
	\qquad \tn{ and thus } \qquad
	\KL(\tilde{\pi}_i|\hat{\pi}_i)=0
	\end{align*}
	where we have used that $\hat{u}_i=\iterl{u}_i$ and $\iterl{v}_i=\iterl{v}_{J_0}$ when $i \in J_0$ (which implies $D(i)=0$).
	
	To control $\KL(\tilde{\pi}_i|\hat{\pi}_i)$ for $i \in I$ with $D(i)>0$ we will use the following bound for $k \in \{0,\ldots,D(i)-1\}$:
	\begin{align}
	& \KL(\iter{\pi}{\ell-k-1}_{n_{i,k}}|\iter{\pi}{\ell-k}_{n_{i,k}}) \nonumber \\
	= {} & \int_{X \times Y} \varphi\left(
	\tfrac{\iter{u}{\ell-k-1}_{n_{i,k}}(x) \cdot \iter{v}{\ell-k-1}_{n_{i,k}}(y)}
	{\iter{u}{\ell-k}_{n_{i,k}}(x) \cdot \iter{v}{\ell-k}_{n_{i,k}}(y)}
	\right)
	\iter{u}{\ell-k}_{n_{i,k}}(x) \cdot \iter{v}{\ell-k}_{n_{i,k}}(y) \, \sKernel(x,y) \, \diff \mu_{n_{i,k}}(x) \diff \nu(y) \nonumber \\
	\geq {} & e^{-\frac{2\|c\|}{\veps}} \cdot \int_{X \times Y} \varphi\left(
	\tfrac{\iter{u}{\ell-k-1}_{n_{i,k}}(x) \cdot \iter{v}{\ell-k-1}_{n_{i,k}}(y)}
	{\iter{u}{\ell-k}_{n_{i,k}}(x) \cdot \iter{v}{\ell-k}_{n_{i,k}}(y)}
	\right)
	\tfrac{\iter{u}{\ell-k}_{n_{i,k}}(x)}{\iter{u}{\ell-k-1}_{n_{i,k}}(x)}\iter{u}{\ell-k-1}_{n_{i,k}}(\hat{x}) \cdot \iter{v}{\ell-k}_{n_{i,k}}(y) \, \diff \mu_{n_{i,k}}(x) \diff \nu(y) \nonumber \\
	= {} & e^{-\frac{2\|c\|}{\veps}} \cdot \int_{Y} \left[ \int_X \varphi\left(
	\tfrac{\iter{u}{\ell-k-1}_{n_{i,k}}(x) \cdot \iter{v}{\ell-k-1}_{n_{i,k}}(y)}
	{\iter{u}{\ell-k}_{n_{i,k}}(x) \cdot \iter{v}{\ell-k}_{n_{i,k}}(y)}
	\right)
	\tfrac{\iter{u}{\ell-k}_{n_{i,k}}(x)}{\iter{u}{\ell-k-1}_{n_{i,k}}(x)} \, \diff \mu_{n_{i,k}}(x)\right]\iter{u}{\ell-k-1}_{n_{i,k}}(\hat{x}) \cdot \iter{v}{\ell-k}_{n_{i,k}}(y) \,\diff \nu(y) \nonumber \\
	\geq {} & e^{-\frac{2\|c\|}{\veps}}\,\|\mu_{n_{i,k}}\|\int_Y \varphi\left(
	\tfrac{\iter{v}{\ell-k-1}_{n_{i,k}}(y)}{q_{i,k}\, \iter{v}{\ell-k}_{n_{i,k}}(y)}
	\right) q_{i,k}\,\iter{u}{\ell-k-1}_{n_{i,k}}(\hat{x})\, \iter{v}{\ell-k}_{n_{i,k}}(y) \, \diff \nu(y)
	\label{eq:NCellProofKLUpper}
	\end{align}
	for $\mu_{n_{i,k}}$-a.e.~$\hat{x}\in X_{n_{i,k}}$, where in the first inequality we have used \eqref{eq:ScalingRatioBound} and the boundedness of $c$, and in the second inequality we have used Jensen's inequality as in \eqref{eq:ThreeConvergenceProofJensen} in the proof of Theorem \ref{thm:ThreeConvergence}.
	
	Now note that one has $\iter{v_{n_{i,k}}}{\ell-k}=\iter{v_{n_{i,k-1}}}{\ell-k}$ for $k \in \{1,\ldots,D(i)\}$, since $\{n_{i,k},n_{i,k-1}\} \in J$ for some $J \in \iter{\partGeneric}{\ell-k}$, cf.~Def.~\ref{def:PartitionGraph} \eqref{item:ShortestPaths} and Fig.~\ref{fig:PartitionGraphIterations}.
	Therefore, for $i$ with $D(i)>0$ and recalling $n_{i,D(i)} = i$, one obtains $\hat{\pi}_i$-almost everywhere that
	\begin{align}
	\RadNikD{\tilde{\pi}_i}{\hat{\pi}_i} & = \frac{\tilde{u}_i \otimes \tilde{v}_i}{\hat{u}_i \otimes \hat{v}_i}
	= \frac{\iter{u_i}{\ell-D(i)} \otimes \iter{v_i}{\ell-D(i)}}{
		\left(\prod_{k=0}^{D(i)-1} q_{i,k}\right)
		\cdot \iter{u_i}{\ell-D(i)}
		\otimes \iterl{v}_{J_0}}
	= \frac{\iter{v_{n_{i,D(i)}}}{\ell-D(i)}}{
		\left(\prod_{k=0}^{D(i)-1} q_{i,k}\right)
		\cdot \iterl{v}_{n_{i,0}}} \nonumber \\
	& = \left(\prod_{k=0}^{D(i)-1} q_{i,k}\right)^{-1} \cdot \frac{\iter{v_{n_{i,D(i)-1}}}{\ell-D(i)}}
	{\iterl{v}_{n_{i,0}}}
	\cdot \underbrace{\left( \prod_{k=1}^{D(i)-1} \frac{\iter{v_{n_{i,k-1}}}{\ell-k}}{\iter{v_{n_{i,k}}}{\ell-k}} \right)}_{=1}
	= \prod_{k=0}^{D(i)-1} \frac{\iter{v_{n_{i,k}}}{\ell-k-1}}
	{q_{i,k}\,\iter{v_{n_{i,k}}}{\ell-k}}\,.
	\label{eq:NCellProofDensityDecomposition}
	\end{align}
	So continuing the suboptimality bound from \eqref{eq:NCellProofSubObtBasic} we obtain
	\begin{align*}
	\Delta(\iterl{\pi}) & \leq \sum_{\substack{i \in I:\\D(i)>0}} \int_{X \times Y} \varphi\left(
	\prod_{k=0}^{D(i)-1} \frac{\iter{v_{n_{i,k}}}{\ell-k-1}}{q_{i,k}\,\iter{v_{n_{i,k}}}{\ell-k}}
	\right)\,\diff \hat{\pi}_i.
	\end{align*}
	In the following we now transform the `$\varphi$ of product' into a `sum of $\varphi$' via Lemma \ref{lem:PhiProductDecomposition}. 
	For each term we then use several auxiliary results to bound $\hat{\pi}_i$ from below by the appropriate $\iter{\pi}{\ell-k}_i$.
	Using the definition of $q_{i,k}$ in \eqref{eq:NCellProofUHat} and \eqref{eq:DensityBound} we find that
	\begin{align*}
	\frac{\iter{v_{n_{i,k}}}{\ell-k-1}}{q_{i,k}\,\iter{v_{n_{i,k}}}{\ell-k}} \leq  \frac{\exp(3\|c\|/\veps)}{\muMin}
	\assignRe L \geq 1
	\end{align*}
	and so by using Lemma \ref{lem:PhiProductDecomposition} one obtains for
	\begin{align*}
	C_1 & \assign 2\,D(i) \cdot L^{D(i)-1} \geq D(i) \cdot \max\{2,L^{D(i)-1}\}
	\end{align*}
	that
	\begin{align}
	\Delta(\iterl{\pi}) &\leq C_1 \cdot \sum_{\substack{i \in I:\\D(i)>0}} \sum_{k=0}^{D(i)-1}
	\int_{X \times Y} \varphi\left(
	\tfrac{\iter{v_{n_{i,k}}}{\ell-k-1}}{q_{i,k}\,\iter{v_{n_{i,k}}}{\ell-k}}
	\right) \,\diff \hat{\pi}_i.
	\label{eq:NCellProofDeltaAikSum}
	\end{align}
	Further, since $i = n_{i,D(i)}$ and $\{n_{i,D(i)}, n_{i,D(i)-1}\} \subset J$ for some $J \in \iter{\partGeneric}{\ell-D(i)}$ (cf.~Def.~\ref{def:PartitionGraph} \eqref{item:ShortestPaths}), by using \eqref{eq:ScalingRatioBound} one finds that for $\mu_i$-a.e.~$x \in X_i$
	\begin{align*}
	q_{i,D(i)-1}\cdot \iter{u}{\ell-D(i)}_i(x) &= \fint_X \frac{\iter{u_{n_{i,D(i)}}}{\ell-D(i)}(x) \cdot \iter{u}{\ell-(D(i)-1)}_{n_{i,D(i)-1}}(\hat{x})}{\iter{u}{\ell-D(i)}_{n_{i,D(i)-1}}(\hat{x})} \,\diff \mu_{n_{i,D(i)-1}}(\hat{x}) \\
	&\leq \exp(\|c\|/\veps) \fint_X \iter{u}{\ell-(D(i)-1)}_{n_{i,D(i)-1}}(\hat{x}) \,\diff \mu_{n_{i,D(i)-1}}(\hat{x})
	\end{align*}
	and therefore, applying iteratively this bound along the chain connecting $X_{i}$ to $X_{n_{i,0}}$, we obtain for $\mu_i$-a.e.~$x \in X_i$ and $\mu_{n_{i,0}}$-a.e.~$x' \in X_{n_{i,0}}$
	\begin{align}
	\hat{u}_i(x) & 
	= \left[\prod_{k=0}^{D(i)-1} q_{i,k}\right]\cdot \iter{u_i}{\ell-D(i)}(x)
	\leq e^{\frac{\|c\|}{\veps}} \cdot \fint_X \left[\prod_{k=0}^{D(i)-2} q_{i,k} \right] \iter{u}{\ell-D(i)+1}_{n_{i,D(i)-1}}(\hat{x}) \,\diff \mu_{n_{i,D(i)-1}}(\hat{x})
	\nonumber \\
	& \leq \dots \leq \exp(D(i)\,\|c\|/\veps) \cdot \fint_{X} \iterl{u}_{n_{i,0}}(\hat{x})\,\diff \mu_{n_{i,0}}(\hat{x}) \leq \exp((D(i)+1)\,\|c\|/\veps) \cdot \iterl{u}_{n_{i,0}}(x').
	\label{eq:NCellProofUVHatBoundA}
	\end{align}
	Combining successively the upper bound in \eqref{eq:DensityAbsBound}, \eqref{eq:RadNikLowerBound} and the lower bound in \eqref{eq:DensityAbsBound} one deduces for $({\mu}_{n_{i,0}} \otimes {\mu}_{n_{i,k}})$-a.e.~$(x, x') \in X_{n_{i,0}} \times X_{n_{i,k}}$ and $\nu$-a.e.~$y \in Y$ that
	\begin{align}
	\iterl{u}_{n_{i,0}}(x) \cdot \iterl{v}_{n_{i,0}}(y) &\leq \tfrac{\exp\big(2 \|c\|/\veps\big)}{\muMin} \leq \tfrac{N\exp\big((2M+4) \|c\|/\veps\big)}{\muMin^{M+2}}	\RadNik{\iter{\nu}{\ell-k}_{n_{i,k}}}{\nu}(y)
	\nonumber \\
	& \leq \tfrac{N\exp\big((2M+5) \|c\|/\veps\big)}{\muMin^{M+2}}\,\cdot\,\|\mu_{n_{i,k}}\| \cdot
	\iter{u}{\ell-k}_{n_{i,k}}(x') \cdot \iter{v}{\ell-k}_{n_{i,k}}(y)
	\label{eq:NCellProofUVHatBoundB}
	\end{align}
	and we set
	\begin{equation*}
	C_2 \assign \tfrac{N\exp\big((2M+5) \|c\|/\veps\big)}{\muMin^{M+2}}.
	\end{equation*}
	Eventually, applying again \eqref{eq:ScalingRatioBound}, we can easily see that
	\begin{equation}\label{eq:step3}
	\iter{u}{\ell-k}_{n_{i,k}}(x) \leq e^{\frac{2\,\|c\|}{\veps}} \cdot q_{i,k} \cdot \iter{u}{\ell-k-1}_{n_{i,k}}(x') \qquad \text{for $\mu_{n_{i,k}}$-a.e.~$x, x' \in X_{n_{i,k}}$}
	\end{equation}
	
	We now have all the ingredients to go back to \eqref{eq:NCellProofDeltaAikSum}. For every fixed $i \in I$ with $\hat{\pi}_i = \hat{u}_i \otimes \hat{v}_i \cdot \kernel_i$ and $0 \leq k \leq D(i)-1$, and for $\mu_{n_{i,0}}$-a.e.~$x' \in X_{n_{i,0}}$ and $\mu_{n_{i,k}}$-a.e.~$x'' \in X_{n_{i,k}}$, we compute
	\begin{align*}
	& \int_{X \times Y} \varphi\left(
	\tfrac{\iter{v_{n_{i,k}}}{\ell-k-1}}{q_{i,k}\,\iter{v_{n_{i,k}}}{\ell-k}}
	\right) \,\diff \hat{\pi}_i = \int_{X \times Y} \varphi\left(
	\tfrac{\iter{v_{n_{i,k}}}{\ell-k-1}(y)}{q_{i,k}\,\iter{v_{n_{i,k}}}{\ell-k}(y)}
	\right) \hat{u}_i(x)\hat{v}_i(y)\underbrace{\sKernel(x,y)}_{\leq1} \,\diff {\mu}_i(x) \diff\nu(y) \\
	& {}\stackrel{\eqref{eq:NCellProofUVHatBoundA}}{\leq} \|\mu_i\|e^{\frac{(M+1)\,\|c\|}{\veps}}\int_{Y} \varphi\left(
	\tfrac{\iter{v_{n_{i,k}}}{\ell-k-1}(y)}{q_{i,k}\,\iter{v_{n_{i,k}}}{\ell-k}(y)}
	\right) \cdot \iterl{u}_{n_{i,0}}(x') \cdot \iterl{v}_{n_{i,0}}(y) \, \diff\nu(y) \\
	&\stackrel{\eqref{eq:NCellProofUVHatBoundB}}{\leq} C_2 \|\mu_i\|\cdot \|\mu_{n_{i,k}}\|e^{\frac{(M+1)\,\|c\|}{\veps}}\int_{Y} \varphi\left(
	\tfrac{\iter{v_{n_{i,k}}}{\ell-k-1}(y)}{q_{i,k}\,\iter{v_{n_{i,k}}}{\ell-k}(y)}
	\right) \iter{u}{\ell-k}_{n_{i,k}}(x'') \cdot \iter{v}{\ell-k}_{n_{i,k}}(y) \, \diff\nu(y) \\
	&\stackrel{\eqref{eq:step3}}{\leq} C_2 \|\mu_i\|\cdot \|\mu_{n_{i,k}}\|e^{\frac{(M+3)\,\|c\|}{\veps}}\int_{Y} \varphi\left(
	\tfrac{\iter{v_{n_{i,k}}}{\ell-k-1}(y)}{q_{i,k}\,\iter{v_{n_{i,k}}}{\ell-k}(y)}
	\right) q_{i,k} \cdot \iter{u}{\ell-k-1}_{n_{i,k}}(x'') \cdot \iter{v}{\ell-k}_{n_{i,k}}(y) \, \diff\nu(y) \\
	&\stackrel{\eqref{eq:NCellProofKLUpper}}{\leq} C_2 \|\mu_i\|e^{\frac{(M+5)\,\|c\|}{\veps}} \KL(\iter{\pi}{\ell-k-1}_{n_{i,k}}|\iter{\pi}{\ell-k}_{n_{i,k}})
	\end{align*}
	This eventually yields
	\begin{align*}
	\Delta(\iterl{\pi}) & \leq C_1 \cdot C_2 \cdot e^{\frac{(M+5)\,\|c\|}{\veps}} \sum_{\substack{i \in I:\\D(i)>0}} \sum_{k=0}^{D(i)-1}
	\|\mu_i\| \KL\left(\iter{\pi}{\ell-k-1}_{n_{i,k}}|\iter{\pi}{\ell-k}_{n_{i,k}}\right) \nonumber \\
	& \leq C_1 \cdot C_2 \cdot e^{\frac{(M+5)\,\|c\|}{\veps}} \sum_{\substack{j \in I:\\D(j)<M}} \sum_{i \in I}
	\|\mu_i\| \KL\left(\iter{\pi}{\ell-D(j)-1}_{j}|\iter{\pi}{\ell-D(j)}_{j}\right) \nonumber \\
	& \leq C_1 \cdot C_2 \cdot e^{\frac{(M+5)\,\|c\|}{\veps}} \sum_{\substack{j \in I}} \sum_{k = 0}^{M-1} \KL\left(\iter{\pi}{\ell-k-1}_{j}|\iter{\pi}{\ell-k}_{j}\right) \nonumber \\
	& = C_1 \cdot C_2 \cdot e^{\frac{(M+5)\,\|c\|}{\veps}} \cdot
	\left(
	\Delta\big(\iter{\pi}{\ell-M}\big)-\Delta\big(\iter{\pi}{\ell}\big)
	\right)
	\end{align*}
	where we used in the second line that $D(j)=k$ if $j=n_{i,k}$ for $j \in I$ (since $n_{i,k}$ is at the $k$-th step of the shortest path to cell $i$).
	This yields the sought-after bound:
	\begin{equation*}
	\Delta(\iterl{\pi}) \leq \frac{2MNe^{\frac{(6M+7)\,\|c\|}{\veps}}}{\muMin^{2M+1}} \cdot \left(
	\Delta\big(\iter{\pi}{\ell-M}\big)-\Delta\big(\iter{\pi}{\ell}\big)
	\right).
	\end{equation*}
	Convergence of the iterates to the optimal solution follows as in Theorem \ref{thm:ThreeConvergence}.
\end{proof}

\subsection{Numerical worst-case examples}
\label{sec:NumericsWorstCase}
In the preceding Sections we derived bounds on the convergence rate based on various worst-case estimates. But convergence might actually be much faster. In this Section we provide numerical examples that demonstrate that, while the pre-factors in our bounds might not be tight, in the three-cell-case the qualitative dependency of the convergence rate on the parameters $\veps$ and the mass in the middle cell $\mu(X_2)$ are accurate; and that convergence slows down in the general case as the maximum distance $M$ from the root cell $J_0$ in the partition graph increases.

\begin{figure}[t]
\centering
\begin{tikzpicture}[x=1cm,y=1cm]
	{\small
	\begin{scope}[x={(0.7,0)},y={(0,0.7)}]

	\foreach \x in {1,2} {
		\draw[black, line width=0.5pt] (0,\x) -- ++ (3,0);
		\draw[black, line width=0.5pt] (\x,0) -- ++ (0,3);
	}
	\draw[black, line width=1pt] (0,0) rectangle (3,3);
	\foreach \x/\y/\c in {0/0/0,1/0/10,2/0/1,%
0/1/10,1/1/0,2/1/10,%
0/2/1,1/2/10,2/2/0} {
		\node at (\x+0.5,\y+0.5) []{\c};
	}
	\foreach \x in {1,2,3} {
	\node at (\x-0.5,0) [anchor=north]{$x_\x$};
	\node at (3,\x-0.5) [anchor=west]{$y_\x$};
	}
	\node at (0,1.5) [anchor=east]{$c=$};
	
	\begin{scope}[shift={(6,0)}]
	\foreach \x in {1,2} {
		\draw[black, line width=0.5pt] (0,\x) -- ++ (3,0);
		\draw[black, line width=0.5pt] (\x,0) -- ++ (0,3);
	}
	\draw[black, line width=1pt] (0,0) rectangle (3,3);
	\foreach \x/\y/\c in {%
0/0/0,1/0/0,2/0/p,%
0/1/0,1/1/q,2/1/0,%
0/2/p,1/2/0,2/2/0} {
		\node at (\x+0.5,\y+0.5) []{$\c$};
	}
	\node at (0,1.5) [anchor=east]{$\iterz{\pi}=$};
	\end{scope}
	
	\begin{scope}[shift={(12,0)}]
		\foreach \i/\x/\y in {2/0/1.5,1/4/0.5,3/4/2.5} {
			\coordinate (i\i) at (\x,\y);
		}
		\foreach \i/\j/\l/\a in {2/1/10/below,2/3/10/above,1/3/1/right} {
			\draw[black,line width=0.75pt] (i\i) -- (i\j) node[midway,\a]{\l};
		}
		\foreach \i/\a in {2/east,1/north,3/south} {
			\draw[black,line width=1pt,fill=black!50!white] (i\i) circle[radius=2pt];
			\node at (i\i) [anchor=\a,inner sep=5pt]{\i};
		}
	\end{scope}
	
	\draw[black,line width=0.5pt,dash pattern=on 2pt off 2pt] (-2,-1) -- ++ (20,0);
	\begin{scope}[shift={(2,-6.5)}]
	\foreach \x in {1,2,3,4} {
		\draw[black, line width=0.5pt] (0,\x) -- ++ (5,0);
		\draw[black, line width=0.5pt] (\x,0) -- ++ (0,5);
	}
	\draw[black, line width=1pt] (0,0) rectangle (5,5);
	\foreach \x in {0,1,2,3,4} {
	\foreach \y in {0,1,2,3,4} {
		\node at (\x+0.5,\y+0.5) []{\ifthenelse{\x=\y}{0}{%
		\ifthenelse{\x\y=04}{1}{%
		\ifthenelse{\x\y=40}{1}{10}}}};
	}
	}
	\node at (0,2.5) [anchor=east]{$c=$};		
	\end{scope}
	\begin{scope}[shift={(10,-6.5)}]
	\foreach \x in {1,2,3,4} {
		\draw[black, line width=0.5pt] (0,\x) -- ++ (5,0);
		\draw[black, line width=0.5pt] (\x,0) -- ++ (0,5);
	}
	\draw[black, line width=1pt] (0,0) rectangle (5,5);
	\foreach \x in {0,1,2,3,4} {
	\foreach \y in {0,1,2,3,4} {
		\node at (\x+0.5,\y+0.5) []{\ifthenelse{\x=\y}{%
		\ifthenelse{\x=0}{0}{%
		\ifthenelse{\x=4}{0}{1}}}{%
		\ifthenelse{\x\y=04}{1}{%
		\ifthenelse{\x\y=40}{1}{0}}}};
	}
	}
	\node at (0,2.5) [anchor=east]{$\iterz{\pi}=\tfrac{1}{5} \cdot$};		
	\end{scope}

	\end{scope}
	}
\end{tikzpicture}
\caption{Setup for worst-case examples. Top row: three cells, Example \ref{ex:WorstCaseThree}, cost $c$, initial coupling $\iterz{\pi}$ and embedding of $X=Y=\{1,2,3\}$ into $\R^2$ such that $c$ becomes the distance function. The domain decomposition algorithm must exchange the masses of the first and third column of $\iterz{\pi}$.\newline
Bottom row: cost $c$ and initial coupling $\iterz{\pi}$ for chain setup, Example \ref{ex:WorstCaseChain}, for $N=5$.}
\label{fig:WorstCaseSetup}
\end{figure}

\begin{figure}[t]
\centering
\includegraphics[]{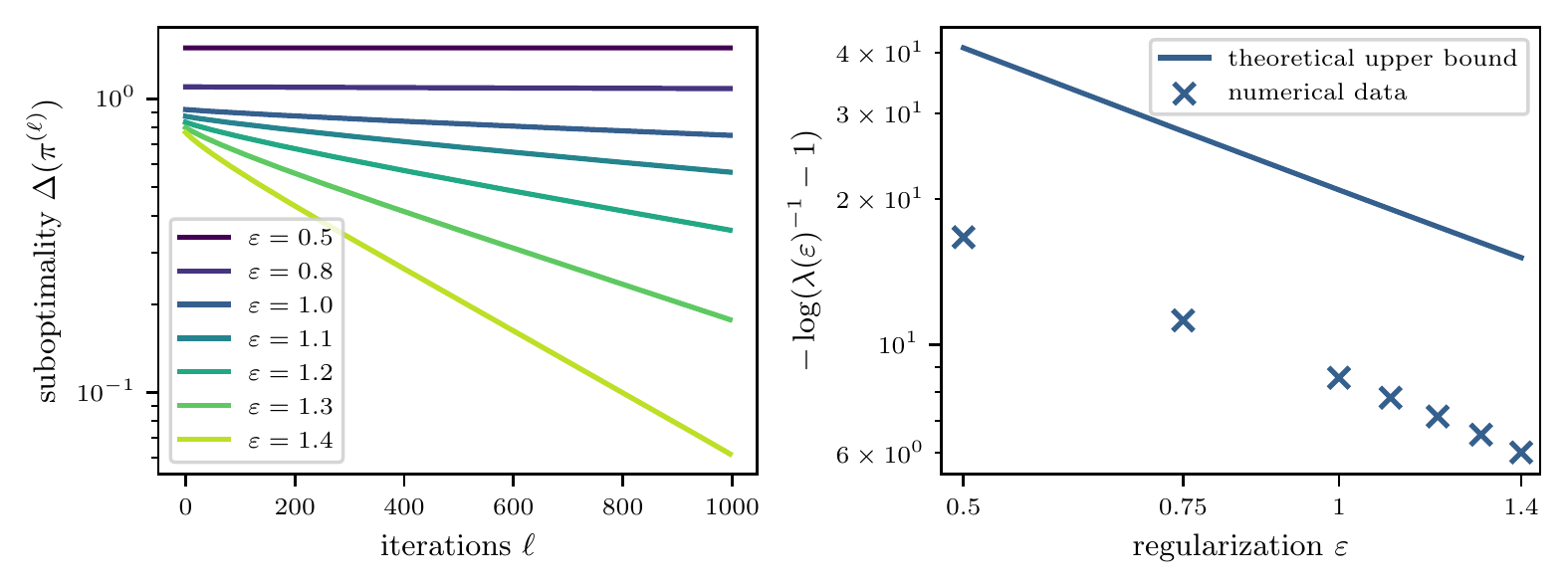}
\caption{Results for $\veps$-dependency in three-cell worst-case example, see Example \ref{ex:WorstCaseThree}. Left: Sub-optimality $\Delta(\iterl{\pi})$ over iterations for various values of $\veps$. The values converge to zero linearly. %
	Right: We extract the linear contraction factor $\lambda(\veps)$ from the left plot and compare it to the theoretical bound from Theorem \ref{thm:ThreeConvergence}. This is visualized best in the form $-\log(\lambda(\veps)^{-1}-1)$ for which the bound yields $\tfrac{2\|c\|}{\veps}-\log(q/(1-q))$, see \eqref{eq:LambdaBound}. In particular, the $1/\veps$-behaviour is accurate and the pre-factor appears to be off only by a factor of $\approx 2$.}
\label{fig:WorstCaseThreeEps}
\end{figure}

\begin{example}[Three cells, dependency on $\veps$]
\label{ex:WorstCaseThree}
Let $X=Y=\{1,2,3\}$ with cost $c$ as given in Figure \ref{fig:WorstCaseSetup}. This configuration can be obtained with embeddings of $X$ and $Y$ into $\R^2$. Set $\mu=\nu=p \cdot \delta_{1} + q \cdot \delta_2 + p \cdot \delta_3$ for $q \in (0,1)$, $p=\tfrac{1-q}{2}$. We choose $q=0.3$.
Basic and composite partitions are chosen as in the previous three-cell examples (Example \ref{ex:ThreeCells}): $X_i=\{i\}$ for $i \in I=\{1,2,3\}$, $\partA=\{\{1,2\},\{3\}\}$, $\partB=\{\{1\},\{2,3\}\}$.

The unique optimal coupling in the unregularized case is given by $\pi^\ast=p \cdot \delta_{(1,1)} + q \cdot \delta_{(2,2)} + p \cdot \delta_{(3,3)}$. As initial coupling for the algorithm we pick $\iterz{\pi} \assign p \cdot \delta_{(1,3)} + q \cdot \delta_{(2,2)} + p \cdot \delta_{(3,1)}$. So the mass in cell $X_1$ must be moved into cell $X_3$ and vice versa.

In the unregularized case $\iterz{\pi}$ would be a fixed-point of the iterations, since $\iterz{\pi}$ is partially optimal on $X_{\{1,2\}}$ and $X_{\{2,3\}}$. This provides a simple example that the unregularized domain decomposition algorithm requires additional assumptions for convergence to the global solution.

In the regularized setting, some mass will also be put onto the points $(1,2)$, $(2,1)$, $(2,3)$, and $(3,2)$ despite the high cost, thus leading to the possibility to move mass between $X_1$ to $X_3$ via $X_2$ and thus eventually solving the problem. But this mass tends to zero exponentially as $\veps\to 0$ thus resulting in an exponential number of iterations.

Numerical results are summarized in Figure \ref{fig:WorstCaseThreeEps}.
To estimate $\Delta(\iterl{\pi})$ we obtain a high precision approximate solution $\pi^\ast$ by applying a single Sinkhorn algorithm to the problem. Since this can directly exchange mass between cells 1 and 3, for the single solver this problem is not particularly difficult (the same applies to the other examples in this Section).
Let $\lambda(\veps)$ be the empirical contraction factor for the sub-optimality. Theorem \ref{thm:ThreeConvergence} yields the bound
\begin{align}
	\lambda(\veps) \leq \left(1+\exp(-2\|c\|/\veps) \cdot \tfrac{q}{1-q} \right)^{-1}
	\,\, \Leftrightarrow \,\,
	-\log(\lambda(\veps)^{-1}-1) \leq \tfrac{2\|c\|}{\veps} - \log(q/(1-q)).
	\label{eq:LambdaBound}
\end{align}
As indicated in Figure \ref{fig:WorstCaseThreeEps}, the $1/\veps$-dependency accurately describes the convergence behaviour.
\end{example}

\begin{figure}[t]
\centering
\includegraphics[]{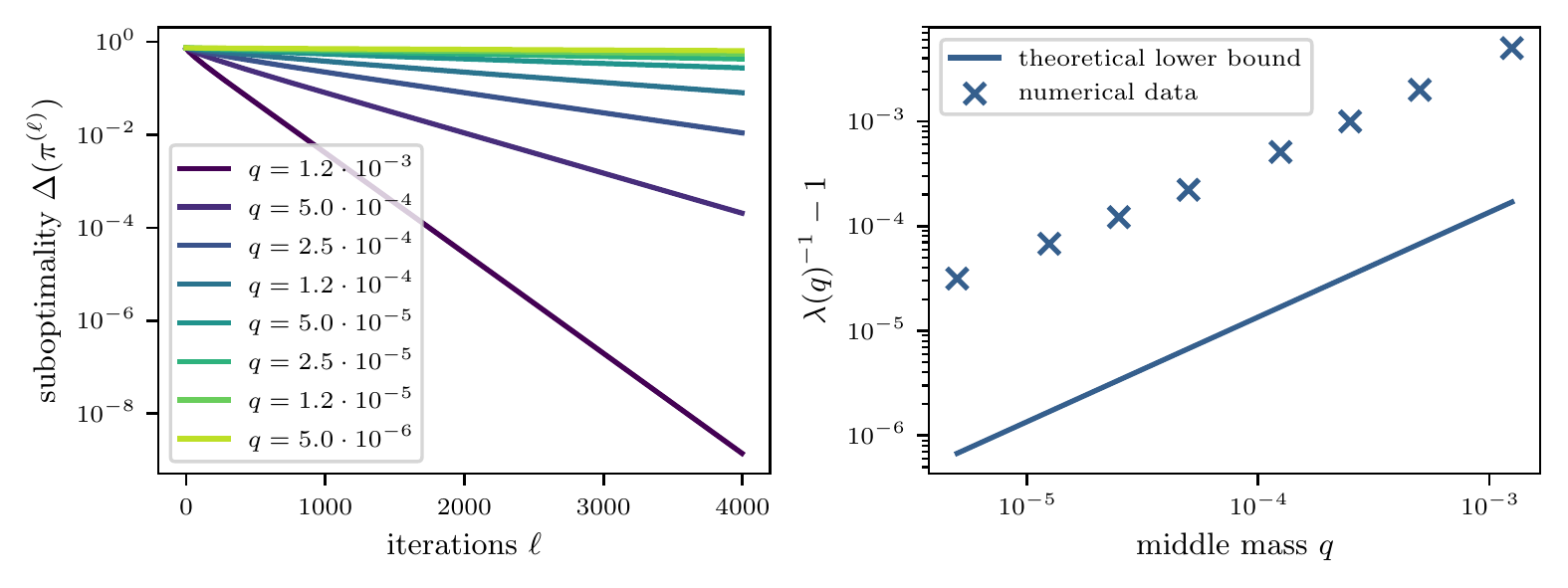}
\caption{Results for $q$-dependency in three-cell worst-case example, see Example \ref{ex:WorstCaseThreeQ}. Left: Sub-optimality $\Delta(\iterl{\pi})$ over iterations for various values of $q$. The values converge to zero linearly. %
	Right: We extract the linear contraction factor $\lambda(q)$ from the left plot and compare it to the theoretical bound from Theorem \ref{thm:ThreeConvergence}. This is visualized best in the form $\lambda(q)^{-1}-1$ for which the bound yields $\exp(-2\|c\|/\veps) \cdot \frac{q}{1-q}$, see \eqref{eq:LambdaBoundQ}. In particular, the proportionality to $q$ (for small values) is accurate. The pre-factor appears to be off by a factor of $\approx 50$.}
\label{fig:WorstCaseThreeQ}
\end{figure}

\begin{example}[Three cells, dependency on $q$]
	\label{ex:WorstCaseThreeQ}
	We revisit the previous Example \ref{ex:WorstCaseThree}, but now we fix the regularization parameter $\veps=10$ and vary the mass $q$ in the middle cell.
	Intuitively, far from the optimal solution, the mass exchanged between the cells in each iteration should be proportional to $q$ and thus for small $q$ the number of required iterations should be approximately proportional to $q^{-1}$. This would imply that for small $q$ the contraction ratio $\lambda(q)$ scales as $\exp(-C\,q) \approx 1-C\,q$ for some constant $C$.
	For small $q$ this matches the behaviour of the bound of Theorem \ref{thm:ThreeConvergence}, which implies, cf.~\eqref{eq:LambdaBound},
	\begin{align}
		\label{eq:LambdaBoundQ}
		\lambda(q)^{-1}-1 \geq \exp(-2\|c\|/\veps) \cdot \frac{q}{1-q}.
	\end{align}
	Numerical results are summarized in Figure \ref{fig:WorstCaseThreeQ}. In our setting the proportionality between $\lambda(q)^{-1}-1$ and $q$ is confirmed numerically.
	Note that we chose $\veps=10$ relatively high. For smaller values we observed that the mass exchanged between the cells in each iterations was no longer proportional to $q$ and thus the interplay between the parameters $\veps$ and $q$ became more complicated.
\end{example}

\begin{figure}[t]
\centering
\includegraphics[]{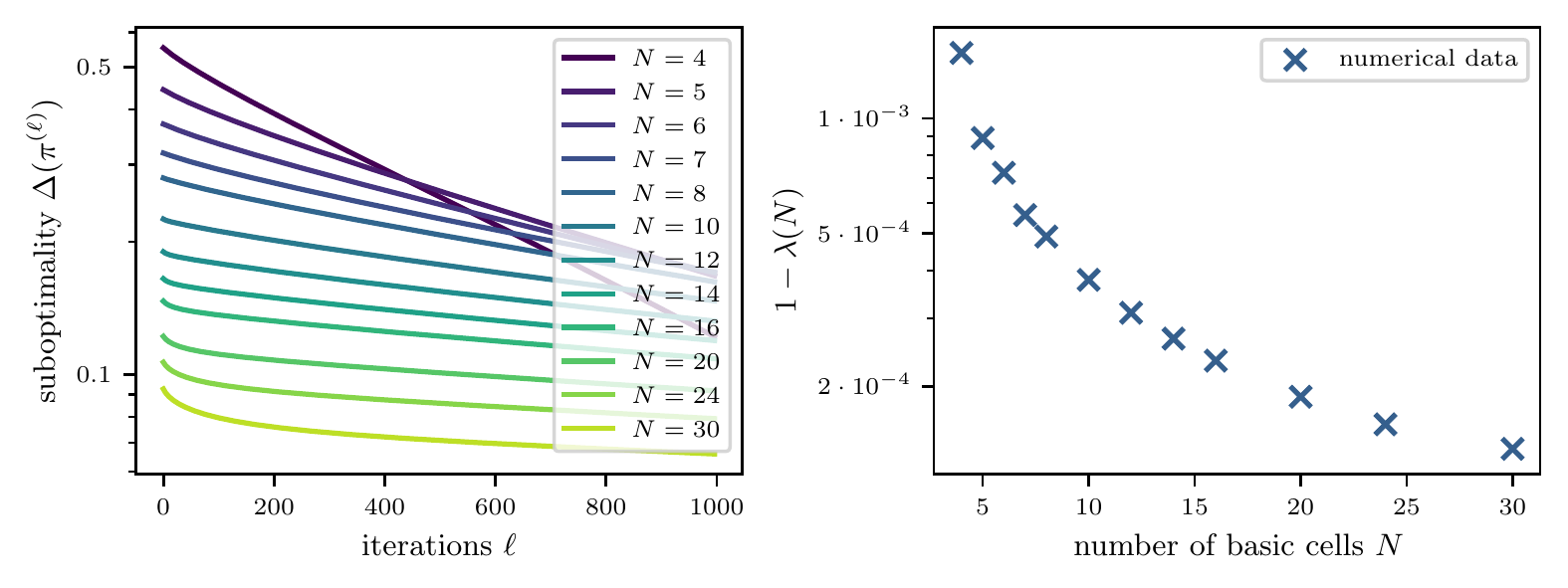}
\caption{Results for $N$-dependency in chain graph worst-case example, see Example \ref{ex:WorstCaseChain}. Left: Sub-optimality $\Delta(\iterl{\pi})$ over iterations for various values of $N$. The values converge to zero linearly. %
	Right: Linear contraction factors $\lambda(N)$ extracted from the left plot. Since they approach 1 very quickly, they are visualized best in the form $1-\lambda(N)$ in log-scale.}
\label{fig:WorstCaseN}
\end{figure}

\begin{example}[Chain graph, dependency on graph size]
\label{ex:WorstCaseChain}
The dependency of the contraction ratio bound of Theorem \ref{thm:NCellConvergence} on the graph structure is more complex. The graph structure determines the numbers $M$, $N$ and implicitly $\muMin$ (since mass must be distributed over increasingly many cells). Due to the variety of worst-case estimates that entered into the proof of \eqref{eq:NCellConvergenceKLDecrement} we do not expect it to be particularly tight. Nevertheless, it is easy to confirm numerically that convergence slows down when the graph size increases.

We extend Example \ref{ex:WorstCaseThree}. For $N \geq 3$ let $X\assign Y\assign I\assign\{1,\ldots,N\}$, $X_i\assign\{i\}$ for $i \in I$. Let $\mu\assign \nu \assign \tfrac{1}{N} \sum_{i \in I} \delta_i$. Set
$$c(i,j) \assign \begin{cases}
	0 & \tn{if } i=j, \\
	1 & \tn{if } (i,j) \in \{(1,N),(N,1)\}, \\
	10 & \tn{else.}
	\end{cases}
$$
The unique optimal coupling is given by $\pi^\ast=\tfrac{1}{N} \sum_{i=1}^N \delta_{(i,i)}$, we choose $\iterz{\pi} = \tfrac{1}{N} \big(\sum_{i=2}^{N-1} \delta_{(i,i)} + \delta_{(1,N)} + \delta_{(N,1)}\big)$. $c$ and $\iterz{\pi}$ are illustrated in Figure \ref{fig:WorstCaseSetup}.
For $N$ even, we set
$$\partA\assign\{\{1,2\},\{3,4\},\ldots,\{N-1,N\}\},\quad
\partB\assign\{\{1\},\{2,3\},\ldots,\{N-2,N-1\},\{N\}\},$$
for $N$ odd we adapt this in the obvious fashion. Finally, let $\veps=1.4$ (but the qualitative results do not depend on $\veps$).
Similar to Example \ref{ex:WorstCaseThree}, mass from the first and last cell must be exchanged, but this time via an increasing number of intermediate cells.
Numerical results for this setting are shown in Figure \ref{fig:WorstCaseN}. We find that $\lambda(N)$ approaches 1 as $N$ increases.
\end{example}

\section{Relation to parallel sorting}
\label{sec:ConvergenceSorting}

The linear rates obtained in Theorems \ref{thm:ThreeConvergence} and \ref{thm:NCellConvergence} become very slow for small $\veps$. The number of required iterations to achieve prescribed sub-optimality scales like $\exp(C/\veps)$ as $\veps$ approaches zero.
This rate is comparable to the result by Franklin and Lorenz \cite{FranklinLorenz-Scaling-1989} on the convergence of the discrete Sinkhorn algorithm.
They show that the intermediate marginal $\iterl{\mu} = P_X \iterl{\pi}$ in the Sinkhorn algorithm satisfies $d(\iterl{\mu},\mu) \leq q^\ell \cdot d(\iterz{\mu},\mu)$ where $d$ denotes Hilbert's projective metric and $q \approx 1-4\exp(-\|c\|/\veps)$ for $\veps \ll \|c\|$.

While we do not claim that the bounds in Theorems \ref{thm:ThreeConvergence} and \ref{thm:NCellConvergence} are tight, they at least qualitatively capture the right scaling behaviour with respect to $\veps$, the mass contained in the basic cells, and the (approximate) graph diameter $M$, as demonstrated with some worst-case examples in Section \ref{sec:NumericsWorstCase}.

We emphasize however that in Section \ref{sec:Convergence} we make virtually no assumptions on the cost function (beyond bounded) and on the structure of the basic and composite partitions (beyond basic cells carrying non-zero mass and a connected partition graph).
This freedom is essential for the construction of the examples in Section \ref{sec:NumericsWorstCase}.
Conversely, in practice we are usually not dealing with such restive problems, but often need to solve transport problems on low-dimensional domains, with highly structured cost functions, and it is possible to choose highly structured basic and composite partitions.
Based on the intuition from Figure \ref{fig:IntroIllustration} we expect that the algorithm will perform much better in such settings, which is confirmed by numerical experiments in Section \ref{sec:NumericsLargeScale}.

Unfortunately, a convergence analysis that leverages the additional geometric structure to obtain better rates is much more challenging than the worst-case bounds from Section \ref{sec:Convergence} and thus beyond the scope of this article.

To provide at least some insight we discuss in this Section an example of discrete, one-dimensional, unregularized optimal transport where the domain decomposition algorithm becomes equivalent to the parallel odd-even transposition sorting algorithm.
Let
\begin{align*}
	X & =\{x_1,\ldots,x_{N}\}, &
	Y & =\{y_1,\ldots,y_{N}\},
\end{align*}
for some $N \in \N$.
Let $\mu$ and $\nu$ be the normalized counting measures on $X$ and $Y$, which can be written as
\begin{align*}
	\mu &= \frac{1}{N} \sum_{i=1}^{N} \delta_{x_i}, &
	\nu &= \frac{1}{N} \sum_{i=1}^{N} \delta_{y_i}.
\end{align*}
and let $c : X \times Y \to \R_+$ satisfy the strict Monge property \cite{ReviewMongeMatrix-96}, i.e.~
\begin{equation}
\label{eq:MongeCost}
c(x_i,y_j)+c(x_r,y_s) < c(x_i,y_s) + c(x_r,y_s) \qquad \tn{for} \qquad 1 \leq i < r \leq N,\,1 \leq j < s \leq N.
\end{equation}
A simple example for such a setting would be when $X$ and $Y$ are (sorted) sets of points in $\R$, $x_i < x_j$, $y_i < y_j$ for $1 \leq i < j \leq N$, and $c(x_i,y_j)=h(x_i-y_j)$ for a strictly convex function $h$. This covers the Wasserstein distances for $p>1$.

For this setting we are now interested in the unregularized optimal transport problem
\begin{equation}
	\min \left\{ \int_{X \times Y} c\,\diff \pi \middle| \pi \in \Pi(\mu,\nu) \right\}.
\end{equation}
It is easy to see that this is equivalent to solving the linear assignment problem between the sets $X$ and $Y$ with respect to the cost $c$ and that the unique optimal solution is given by $\pi=\tfrac{1}{N} \sum_{i=1}^N \delta_{(x_i,y_i)}$ (uniqueness is implied by strict inequality in \eqref{eq:MongeCost}).

Let now a basic partition of $X$ be given by the partition into singletons, i.e.~$X_i=\{x_i\}$ for $i \in I = \{1,\ldots,N\}$ and the two composite partitions $\partA$ and $\partB$ will be set to
\begin{align}
	\partA & = (\{1,2\},\{3,4\},\ldots,\{N-1,N\}), &
	\partB & = (\{1\},\{2,3\},\ldots,\{N-2,N-1\},\{N\})
\end{align}
where for simplicity of exposition we assume that $N$ is even.
Further, let $\iterz{\pi}=\tfrac{1}{N} \sum_{i=1}^N \delta_{(x_i,y_{\sigma(i)})}$ where $\sigma : \{1,\ldots,N\} \to \{1,\ldots,N\}$ is some permutation.
We can now conceive running an unregularized version of Algorithm \ref{alg:DomDec} where we replace the regularized transport solver in line \ref{line:CompositePi} by the unregularized counterpart.
It can quickly be seen that
(due to the strict Monge property) the solution of the (now unregularized) sub-problems in line \ref{line:CompositePi} is always unique
and at each iteration $\iterl{\pi}$ corresponds to a permutation $\iterl{\sigma}$.
Each composite cell (with the exception of the cells $\{1\}$ and $\{N\}$ in $\partB$) consists of two consecutive elements of $X$ and in line \ref{line:CompositePi} it is tested whether the two corresponding assigned points in $Y$ should be kept in the current order or flipped.
Therefore, in this special case the domain decomposition algorithm reduces to the odd-even transposition parallel sorting algorithm \cite[Exercise 37]{KnuthArtProgramming1998-3} which is known to converge in $O(N)$ iterations. It is not hard to see that the analysis for odd-even transposition sort could be adapted to the continuous one-dimensional strictly-Monge case to establish convergence in $O(1/\muMin)$ steps.

Unfortunately, the techniques do not generalize to multiple dimensions and an accurate analysis of the convergence speed of the domain decomposition algorithm in that setting is therefore still an open problem.
Nevertheless, we are optimistic that the expected performance for `structured problems' in higher dimensions will be much better than the worst-case bounds from Section \ref{sec:Convergence}, which is confirmed numerically in Section \ref{sec:NumericsLargeScale}.

\section{Practical large-scale implementation}
\label{sec:Implementation}

\newcommand{\cell}{\tn{cell}}
\newcommand{\myitem}{\textbullet\,}
\newcommand{\ForAllLine}[1]{\algorithmicforall\ #1 \algorithmicdo:}
\begin{algorithmfloat}[hbt]
	\noindent
	\textbf{Input}:\\
	\myitem initial feasible basic $Y$-marginals $(\nu_i^{0})_{i \in I}$ (i.e.~satisfying $\|\nu_i^0\|=\|\mu_i\|$ for $i \in I$ and $\sum_{i \in I} \nu^0_i=\nu$)\\
	\myitem scaling factor $u^0 \in L^\infty_+(X,\mu)$\\
	\noindent
	\textbf{Output}:\\
	\myitem final feasible basic $Y$-marginals $(\nu_i)_{i \in I}$\\
	\myitem final partial scaling factors $(u_J)_{J \in \partA}$, $(u_J)_{J \in \partB}$ on both composite partitions
	\smallskip
	\begin{algorithmic}[1]
		\State \ForAllLine{$i \in I$} $\nu_i \leftarrow \nu_i^0$
		\State \ForAllLine{$J \in \partA$} $u_{A,J} \leftarrow u \restr X_J$
		\State \ForAllLine{$J \in \partB$} $u_{B,J} \leftarrow u \restr X_J$
		\State $\ell \leftarrow 0$
		\Repeat
		\State $\ell \leftarrow \ell+1$
		\State \algorithmicif\ ($\ell$ is odd)\ \algorithmicthen\ $C \leftarrow A$\ \algorithmicelse\ $C \leftarrow B$
		\Comment{select the partition}
		\ForAll{$J \in \partC$} \label{line:CompositeFor}
		\Comment{solve over each composite cell}
		\State $\nu_{\cell} \leftarrow \sum_{i \in J} \nu_i$
			\Comment{compute $Y$-marginal on cell}
		\State $\pi_\cell,u_{C,J} \leftarrow \Call{Sinkhorn}{\mu_J,\nu_\cell,u_{C,J}}$ \Comment{approx.~sol.~w.~Sinkhorn algorithm}
		\label{line:Sinkhorn}
		\State \ForAllLine{$i \in J$} $\nu_i \leftarrow \proj_Y (\pi_\cell \restr (X_i \times Y))$ \Comment{extract new basic partial marginals} \label{line:BasicPartialMarginals}
		\State $\Call{BalanceMeasures}{(\nu_i)_{i \in J})}$
		\State \ForAllLine{$i \in J$} $\Call{Truncate}{\nu_i}$
		\EndFor \label{line:EndCompositeFor}
		\Until stopping criterion
		\State \Return $(\nu_i)_{i \in I}$, $(u_{A,J})_{J \in \partA}$, $(u_{B,J})_{J \in \partB}$
	\end{algorithmic}	

	\noindent \textbf{Comment:} The functions \textsc{Sinkhorn} and \textsc{BalanceMeasures} are described in Section \ref{sec:ComputBalanceMeasures}. The function \textsc{Truncate} refers to truncation up to a small threshold to obtain sparse vectors, as described in Section \ref{sec:ComputOverview}, item \eqref{item:Truncation}. \\
	The variable $\pi_\cell$ is a local variable that is not stored beyond the for-loop over $\partC$.
\caption{Practical implementation of entropic domain decomposition}
\label{alg:DomDecComput}
\end{algorithmfloat}

\subsection{Computational adaptations}
\label{sec:ComputOverview}
Our strongest motivation for studying the entropic domain decomposition algorithm is to develop an efficient numerical method.
The statement given in Algorithm \ref{alg:DomDec} is mathematically abstract and served as preparation for theoretical analysis.
In this Section we describe a more concrete computational implementation, formalized in Algorithm \ref{alg:DomDecComput}. We summarize the main modifications:
\begin{enumerate}[(i)]
	\item Instead of storing the full coupling $\iterl{\pi}$, we only store the $Y$-marginal $\iterl{\nu}_i \assign \proj_Y \iterl{\pi}_i$ of each basic cell $i \in I$.
	This suffices to compute the $Y$-marginals of $\iterl{\pi}$ on each composite cell, as done in Algorithm \ref{alg:DomDec}, line \ref*{line:YMarginal}.
	The partial couplings $\iterl{\pi}_i$ can quickly be recovered with the scaling factors $u$, see \eqref{item:ComputAlgScaling} below.
	\item To further reduce the memory footprint, we truncate $\iterl{\nu}_i$ after each iteration, i.e.~we only keep a sparse representation of the entries above a small threshold, which we set to $10^{-15}$.
	In theory this invalidates the global density bound \eqref{eq:RadNikLowerBound} in Lemma \ref{lem:DensityMarginalBound} and thus jeopardizes convergence of the method.
	In practice our threshold is low enough such that $\iterl{\nu}_i$ on neighbouring basic cells have strong overlap and thus, based on the intuition from Section \ref{sec:ConvergenceSorting}, on geometric problems we still expect convergence to an approximate solution. This is can be confirmed numerically via the primal-dual gap (Section \ref{sec:NumericsQuantitative} and Table \ref{tab:Summary}).
	\label{item:Truncation}
	\item We replace `abstract exact solution' of the entropic transport problems on the composite cells by the Sinkhorn algorithm (compare Algorithm \ref{alg:DomDec}, line \ref*{line:CompositePi} and Algorithm \ref{alg:DomDecComput}, line \ref*{line:Sinkhorn}).
	This means, we need to be able to handle approximate solutions for the cell problems. We describe how this is done in Section \ref{sec:ComputBalanceMeasures}.
	\item We explicitly keep track of the partial scaling factors $u$ on the composite cells of $\partA$ and $\partB$.
	This allows accurate initialization of the Sinkhorn algorithm on each composite cell, in particular after the first few iterations.
	Further, it enables us to quickly reconstruct the partial couplings on composite and basic cells.
	Finally, we can use them to estimate the sub-optimality via the primal-dual gap, see Section \ref{sec:ComputEvaluatePDGap}.
	\label{item:ComputAlgScaling}
	\item We combine Algorithm \ref{alg:DomDecComput} with a multi-scale scheme and the $\veps$-scaling heuristic described in \cite{SchmitzerScaling2019}. This is critical for computational efficiency, as it allows to obtain a good (and sparse) initial coupling $\pi^0$ (or its partial $Y$-marginals on the basic cells) and drastically reduces the number of required iterations.
	Some more details are given in Section \ref{sec:ComputMultiScale}.
\end{enumerate}

\newcommand{\err}{\tn{Err}}
\subsection{Sinkhorn solver and handling approximate cell solutions}
\label{sec:ComputBalanceMeasures}
The Sinkhorn solver at line \ref*{line:Sinkhorn} in Algorithm \ref{alg:DomDecComput} only returns an approximate solution. This requires some additional post-processing.
Note that for efficiency, we initialize the solver with the scaling factor $u_{C,J}$ from the previous iteration, i.e.~we start the solver with a $Y$-iteration (Remark \ref{rem:Sinkhorn}).

As stopping criterion of the Sinkhorn solver we use the $L^1$-marginal error of the $X$-marginal after the $Y$-iteration.
We require that this error is smaller than $\|\mu_J\| \cdot \err$ where $\err$ is some global error threshold.
In this way, the summed $L^1$-error of all $X$-marginals over all composite cells is bounded by $\err$ after each iteration.

We terminate the algorithm after a $Y$-iteration, i.e.~the $Y$-marginals are satisfied exactly, $\proj_Y \pi_\cell=\nu_\cell$, and after the for-loop in line \ref*{line:EndCompositeFor} one has $\sum_{i \in I} \nu_i = \nu$, which is crucial for the validity of the domain decomposition scheme.
Since $\mu_J = \mu \restr X_J$ is constant throughout the algorithm it is easier to handle a deviation between $\proj_X \pi_\cell$ and $\mu_J$.

However, when $\pi_\cell$ is not contained in $\Pi(\mu_J,\nu_\cell)$, after computing the basic cell marginals in line \ref*{line:BasicPartialMarginals} it may occur that $\|\nu_i\| \neq \|\mu_i\|$. This means, that while $\sum_{i \in J} \nu_i = \nu_\cell$, the masses of the $\nu_i$ may not exactly match the masses in their basic cells. This would cause problems in subsequent iterations. Therefore, the function $\textsc{BalanceMeasures}$ determines for each basic cell $i \in J$ the deviation $\|\nu_i\| - \|\mu_i\|$ and then moves mass between the different $(\nu_i)_{i \in J}$ to compensate the deviations, while preserving their non-negativity and their sum. This can be done efficiently directly in the sparse data structure and without adding new support points.

\subsection{Estimating the primal-dual gap}
\label{sec:ComputEvaluatePDGap}
Explicitly evaluating the primal-dual gap between \eqref{eq:EntropicOTPrimal} and \eqref{eq:EntropicOTDual} (see Proposition \ref{prop:Duality} \eqref{item:Duality:Gap}) is a useful tool to verify that a high-quality approximate solution was found.
As primal candidate we can use the current primal iterate. We can sum up the contribution of each $\pi_\cell$ to \eqref{eq:EntropicOTPrimal} in the for-loop starting at line \ref*{line:CompositeFor}.
The partial replacement with older iterates, as in Lemma \ref{lem:PrimalCandidate}, was merely required for the proof strategy, for numerical evaluation it is not necessary.

As a dual candidate we could use $\hat{u}$ and $\hat{v}$ as constructed in (\ref{eq:NCellDualCandidate},\ref{eq:NCellProofUHat}). It is not hard to see that upon perfect convergence of the domain decomposition algorithm, the functions $\iterl{u}_{J_A}$ and $\iterlm{u}_{J_B}$ on two overlapping composite cells $J_A \in \partA$, $J_B \in \partB$ would only differ by a constant factor on their overlap $J_A \cap J_B \neq \emptyset$.
Consequently, this `glued' together $(\hat{u},\hat{v})$ would be a dual maximizer (cf.~proof strategy of Proposition \ref{prop:SimpleConvergence}). In practice, before perfect convergence, we find however that this construction yields unnecessarily bad dual estimates. We now sketch a slightly more sophisticated construction (along with an explanation why the above construction does not work well numerically).

Inspired by the partition graph, Definition \ref{def:PartitionGraph}, and the construction (\ref{eq:NCellDualCandidate},\ref{eq:NCellProofUHat}) we now define another weighted graph.
The vertex set is given by the composite cells $\partA$. Between two vertices $J_1, J_2 \in \partA$ there will be an edge if there exists some $J_B \in \partB$ such that $J_1 \cap J_B \neq \emptyset$ and $J_2 \cap J_B \neq \emptyset$.
The corresponding edge-weight will be set to
\begin{align*}
	q_{(J_1,J_2)} \assign \fint_{J_1 \cap J_B} \frac{u_{A,J_1}}{u_{B,J_B}}\,\diff \mu
		\cdot \fint_{J_2 \cap J_B} \frac{u_{B,J_B}}{u_{A,J_2}}\,\diff \mu.
\end{align*}
Note that the weight depends on the orientation of the edge and that there may be multiple edges between the same two vertices.
Similar to (\ref{eq:NCellDualCandidate},\ref{eq:NCellProofUHat}) we could now fix a root cell $J_0 \in \partA$ and for any $J \in \partA$ we fix some path $(J_0,J_1,\ldots,J_n=J)$ and set
\begin{align*}
	\hat{u}_{J} \assign \prod_{k=0}^{n-1} q_{(J_k,J_{k+1})} \cdot u_{A,J}.
\end{align*}
Again, after perfect convergence, the $\hat{u}$ obtained by gluing the $\hat{u}_J$ would be a dual maximal $X$-scaling.
The value of $\hat{u}_J$ would not depend on the choice of the path from $J_0$ to $J$ and when going around a cycle in the graph, the product of the corresponding weight factors will be equal to $1$.
This means that $\log q_{(J_1,J_2)}$ is a discrete gradient on the graph, i.e.~summing $\log q_{(J_1,J_2)}$ around a cycle yields zero.

But before perfect convergence, the weights are inconsistent in the sense that $\log q_{(J_1,J_2)}$ is not a gradient and by picking some arbitrary path from $J_0$ to $J$ we distribute this inconsistency in an arbitrary way over $\hat{u}_J$. In practice this leads to poor candidates $\hat{u}$, in particular far from the root cell.

We weaken this effect by doing a discrete Helmholtz decomposition of $\log q_{(J_1,J_2)}$. More concretely, we look for a minimizer of
\begin{align*}
	\min \left\{ \sum_{\substack{J_1,J_2 \in \partA:\\(J_1,J_2) \tn{ is graph edge}}} \left( (V(J_1)-V(J_2)) - \log q_{(J_1,J_2)}  \right)^2
	\middle|
	V : \partA \to \R
	\right\}.
\end{align*}
The gradient of the minimizing potential $V$ will be an optimal approximation of the edge weights in a mean-square sense.
This problem can be solved quickly via a linear system, also when there are multiple edges between two vertices.
We then set $\hat{u}_J \assign \exp(V_J) \cdot u_{A,J}$. We observe that this distributes the inconsistencies in the weights more evenly over the graph and yields better dual estimates.

Since we truncate the partial $Y$-marginals (see Section \ref{sec:ComputOverview} \eqref{item:Truncation}), unlike in the theoretical analysis the factors $v_J$ are not defined on all of $Y$. But once the factors $V$ are available, the $v_J$ can be `glued' together similar to $\hat{u}$.

In practice this method yields small primal-dual gaps (cf.~Section \ref{sec:NumericsQuantitative}), confirming that the entropic domain decomposition works well.
\subsection[Multi-scale scheme and epsilon-scaling]{Multi-scale scheme and $\veps$-scaling}
\label{sec:ComputMultiScale}
In \cite{SchmitzerScaling2019} it is described in detail how a combination of a multi-scale scheme and $\veps$-scaling increases the computational efficiency of the Sinkhorn algorithm and how to implement them numerically.
Similarly, the two techniques are crucial for computational efficiency of the domain decomposition scheme.
For instance, to initialize Algorithms \ref{alg:DomDec} and \ref{alg:DomDecComput} one requires a feasible coupling $\pi^0 \in \Pi(\mu,\nu)$ (or its partial $Y$-marginals on basic cells). Without further prior knowledge, the only canonical candidate would be the product measure $\mu \otimes \nu$. But this is almost certainly far from the optimal coupling, hence requiring many iterations, and it is dense, i.e.~requiring a lot of memory in the initial iterations. Instead, as usual, we start by solving the problem on a coarse scale and then refine the (approximate) solutions to subsequent finer scales for initialization.

\newcommand{\pa}{\tn{pa}}
\newcommand{\Pa}{\tn{Pa}}
We now provide a few more details on how this is done.
Assume $X$ and $Y$ are discrete, finite sets. Let $\hat{X}$ and $\hat{Y}$ be coarse approximations of $X$ and $Y$, let $\pa: X \sqcup Y \to \hat{X} \sqcup \hat{Y}$ be a function that assigns to every element of $X$ and $Y$ their \emph{parent} element in $\hat{X}$ and $\hat{Y}$. Coarse approximations of $\mu$ and $\nu$ are then given by $\hat{\mu} \assign \pa_\sharp \mu$ and $\hat{\nu} \assign \pa_\sharp \nu$ where $\pa_\sharp$ denotes the push-forward of measures.
Let $\{\hat{X}_i\}_{i \in \hat{I}}$ be a basic partition of $(\hat{X},\hat{\mu})$ and let $\{X_i\}_{i \in I}$ be a basic partition of $(X,\mu)$. We choose them in such a way that the fine basic partition can be interpreted as a refinement of the coarse basic partition, i.e.~for each $i \in I$ there is some unique $j \in \hat{I}$ such that $\pa(x) \in \hat{X}_{j}$ for all $x \in X_i$.
We denote the function $I \to \hat{I}$, taking $i$ to $j$ by $\Pa$.

Let now $(\hat{\nu}_j)_{j \in \hat{I}}$ be the tuple of partial marginals as returned by Algorithm \ref{alg:DomDecComput} on the coarse scale. We use as initialization on the subsequent finer scale the measures defined via
\begin{align}
	\label{eq:FineBasicMarginals}
	\nu_i^0(y) \assign \nu(y) \cdot \frac{\hat{\nu}_{\Pa(i)}(\pa(y))}{\hat{\nu}(\pa(y))} \cdot \frac{\mu(X_i)}{\hat{\mu}(\hat{X}_{\Pa(i)})}
\end{align}
where by a slight abuse of notation, we write $\nu(y)$ to mean $\nu(\{y\})$ for singletons.
\begin{proposition}
	\label{prop:FineBasicMarginals}
	\eqref{eq:FineBasicMarginals} provides a valid initialization for Algorithm \ref{alg:DomDecComput}.
\end{proposition}
\begin{proof}
	We need to show that $\sum_{i \in I} \nu^0_i=\nu$ and $\|\nu^0_i\|=\|\mu_i\|$ for all $i \in I$.
	For the former, for $y \in Y$ we find
	\begin{align*}
		\sum_{i \in I} \nu_i^0(y) & = \sum_{j \in \hat{I}}
			\sum_{\substack{i \in I:\\\Pa(i)=j}} \nu_i^0(y)
		= \nu(y) \sum_{j \in \hat{I}} \frac{\hat{\nu}_j(\pa(y))}{\hat{\nu}(\pa(y))}
			\sum_{\substack{i \in I:\\\Pa(i)=j}} \frac{\mu(X_i)}{\hat{\mu}(\hat{X}_j)}=\nu(y)
	\end{align*}
	where in the third expression we first see that the third factor equals 1, since $\pa^{-1}(\hat{X}_j)=\cup_{i \in I:\Pa(i)=j} X_i$, and subsequently the second factor equals 1, since $\sum_{j \in \hat{I}} \hat{\nu}_j=\hat{\nu}$, which is implied by the domain decomposition algorithm on the coarse scale.
	
	Similarly, for the latter, for $i \in I$,
	\begin{align*}
		\|\nu_i^0\| & =\sum_{\hat{y} \in \hat{Y}} \sum_{\substack{y \in Y:\\\pa(y)=\hat{y}}} \nu_i^0(y)
		=  \frac{\mu(X_i)}{\hat{\mu}(\hat{X}_{\Pa(i)})}
			\sum_{\hat{y} \in \hat{Y}} \frac{\hat{\nu}_{\Pa(i)}(\hat{y})}{\hat{\nu}(\hat{y})}
			\underbrace{\sum_{\substack{y \in Y:\\\pa(y)=\hat{y}}} \nu(y)}_{=\hat{\nu}(\hat{y})} \\
		& = \frac{\mu(X_i)}{\hat{\mu}(\hat{X}_{\Pa(i)})}
			\sum_{\hat{y} \in \hat{Y}} \hat{\nu}_{\Pa(i)}(\hat{y})
			= \frac{\mu(X_i)}{\hat{\mu}(\hat{X}_{\Pa(i)})}
			\hat{\nu}_{\Pa(i)}(\hat{Y})=\mu(X_i)
	\end{align*}
	where we have used the consistency condition $\|\hat{\nu}_j\|=\|\hat{\mu}_j\|=\hat{\mu}(\hat{X}_j)$ for $j \in \hat{I}$ in the coarse domain decomposition algorithm.
\end{proof}

	Of course this scheme can be repeated over multiple successive levels of approximations.
	
	In addition, we refine the scaling factors $u$ from the coarse to the fine level to obtain good initializations. For this, we first `glue' together the final $u$ at the coarse level, as described in Section \ref{sec:ComputEvaluatePDGap}.
	Then we do a linear interpolation onto the fine points (which we perform in $\log$-space).
	
	Finally, we remark that when changing the parameter $\veps$ during $\veps$-scaling from $\veps_{\tn{old}}$ to $\veps_{\tn{new}}$, the scalings $(u,v)$ should not be left constant, but instead $(\veps \cdot \log(u), \veps \cdot \log(v))$ should be kept constant, so we set $u_{\tn{new}} \assign \exp(\tfrac{\veps_{\tn{old}}}{\veps_{\tn{new}}} \cdot \log u_{\tn{old}})$, and similarly for $v_\tn{new}$, cf.~\cite[Section 3.2]{SchmitzerScaling2019}.
\section{Numerical geometric large scale examples}
\label{sec:NumericsLargeScale}
We now demonstrate the efficiency of Algorithm \ref{alg:DomDecComput} on large-scale geometric problems and show that it compares favorably to the single Sinkhorn algorithm.
\subsection{Preliminaries}
\label{sec:NumericsPreliminaries}
\paragraph{Hardware and implementation.}
All numerical experiments were performed on a standard desktop computer with an Intel Core i7-8700 CPU with 6 physical cores and 32 GB of RAM, running a standard Ubuntu-based distribution. We developed an experimental implementation for Python 3.7.
Parallelization was implemented with MPI, via the module mpi4py.\footnote{\url{https://mpi4py.readthedocs.io/}}
We used a simple master/worker setup with one thread (master) keeping track of the overall problem and submitting the sub-tasks to the other threads (workers).
Parallelization was used for solving of the composite cell problems, the \textsc{BalanceMeasures} function and data refinement from coarse to fine layers.

\paragraph{Stopping criterion and measure truncation.}
We set the stopping criterion for the internal Sinkhorn solver to $\err=10^{-4}$ (see Section \ref{sec:ComputBalanceMeasures}). This bounds the $L^1$-marginal error of the primal iterates.
Partial marginals $(\nu_i)_{i \in I}$ are truncated at $10^{-15}$ and stored as sparse vectors.

\paragraph{Test data.}
We focus on solving the Wasserstein-2 optimal transport problem, i.e.~we set $c(x,y)=\|x-y\|^2$ but the scheme can be applied to arbitrary costs and we expect efficient performance for any cost of the form $c(x,y)=h(x-y)$ for strictly convex $h$, such that (in the continuous, unregularized setting) the unique optimal plan is concentrated on the graph of a Monge map, see \cite{McCannGangboOTGeometry1996}.

\begin{figure}[t]
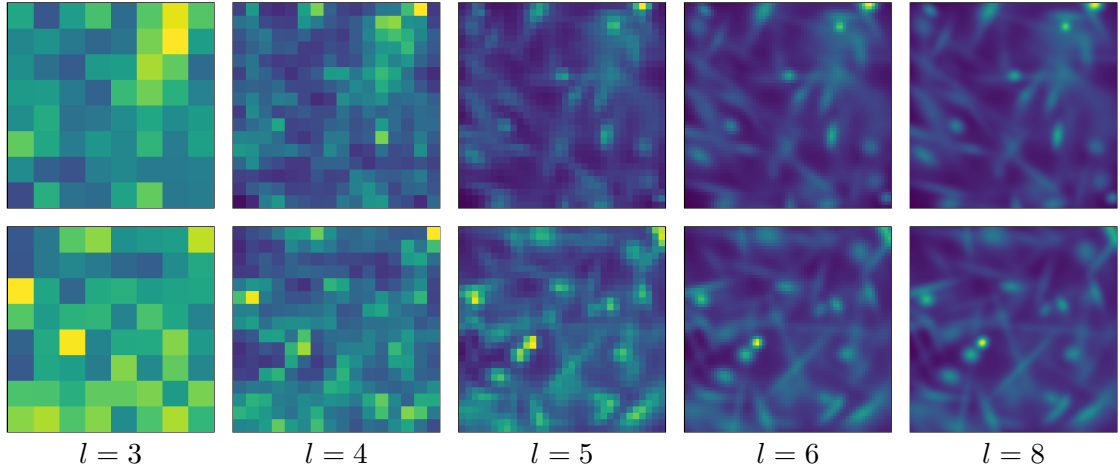

	\centering
	{\def\imgw{2.7cm}
	\begin{tikzpicture}[x=\imgw,y=\imgw,img/.style={inner sep=0pt,draw=black,line width=1pt,anchor=north west}]
	\foreach \x/\y/\l in {0/0/3,1/0/4,2/0/5,3/0/6,4/0/8} {
		\node[img] at (1.1*\x,-1.1*\y){\includegraphics[width=\imgw]{fig/atomic_cells/muX_l\l.png}};
	}
	\begin{scope}[shift={(0,-1.1)}]
	\foreach \x/\y/\l in {0/0/3,1/0/4,2/0/5,3/0/6,4/0/8} {
		\node[img] at (1.1*\x,-1.1*\y)[label=below:{$l=\l$}]{\includegraphics[width=\imgw]{fig/atomic_cells/muY_l\l.png}};
	}
	\end{scope}
	\end{tikzpicture}
	}
	\caption{Two example images for size $256 \times 256$, at different layers, $l=8$ being the original. Yellow represents high, blue low mass density. The combination of strong local mass concentrations, smooth regions and areas with very little mass leads to challenging test problems with non-trivial transport maps.}
	\label{fig:ExampleImages}
\end{figure}
As test data we use 2D images with dimensions $2^n \times 2^n$ for $n=6$ to $n=11$, i.e.~from $64 \times 64$ up to $2048 \times 2048$, the number of pixels per image ranging from 4.1E3 to 4.2E6.
Similar to \cite{SchmitzerShortCuts2015,SchmitzerScaling2019} the images were randomly generated, consisting of mixtures of Gaussians with random variances and magnitudes.
This represents challenging problem data, leading to non-trivial optimal couplings, involving strong compression, expansion and anisotropies, as visualized in Figures \ref{fig:ExampleImages} and \ref{fig:BasicCellsLayers}.
For each problem size we generated 10 test images, i.e.~45 pairwise non-trivial transport problems, to get some estimate of the average performance.

\paragraph{Multi-scale and $\veps$-scaling.}
Each image is then represented by a hierarchy of increasingly finer images of size $2^l \times 2^l$ where $l$ ranges from $3$ to $n$.
We refer to $l$ as the \emph{layer}. This induces a simple quad-tree structure over the images where each pixel at layer $l$ is the parent of four pixels at layer $l+1$, see Section \ref{sec:ComputMultiScale} and \cite{SchmitzerShortCuts2015,SchmitzerScaling2019} for more details.

We embed the pixels into $\R^2$ as a Cartesian grid of edge length $\Delta x_n \assign 1$ between two neighbouring pixels at layer $l=n$ and with edge length $\Delta x_l \assign 2^{n-l}$ on coarser layers.
At each layer $l$ we start solving with $\veps=2 \cdot \Delta x_l^2$, for which we apply four iterations (two on $\partA$ and two on $\partB$). Then we decrease $\veps$ by a factor of two and apply two more iterations (one on $\partA$ and one on $\partB$). We repeat this (so that $\veps=0.5 \cdot \Delta x_l^2$) and then refine the image.
This leaves the scale of $c(\cdot,\cdot)/\veps$ approximately invariant throughout the layers (cf.~\cite{SchmitzerScaling2019}).
At the finest layer, we perform two additional iterations at $\veps=0.25$, which implies that the influence of entropic regularization in our final solution is rather weak (cf.~Figure \ref{fig:BasicCellsEps}).

This means, that for two images of size $2^n \times 2^n$ we only perform $(n-2) \cdot 8 + 2$ iterations of the domain decomposition algorithm, even though there are $2^{2n-4}$ basic cells in the $X$-image (see below) and the graph diameter (Definition \ref{def:PartitionGraph}) is on the order $O(2^{n-2})$. We observe that this logarithmic number of iterations is fully sufficient (cf.~Table \ref{tab:Summary}).
This is only possible due to the geometric problem structure (Section \ref{sec:ConvergenceSorting}) and the multi-scale scheme with $\veps$ scaling.

\paragraph{Basic and composite partitions.}
As basic partition we divide each image (at a given layer $l$) into blocks of $s \times s$ pixels where the cell size $s$ is a suitable divisor of $2^l$.
The composite partition $\partA$ is then generated by summarizing the basic cells into $2 \times 2$ blocks, the partition $\partB$ is generated in the same way but with an offset of one basic cell in each direction, i.e.~with single basic cells in the corners and $2 \times 1$ blocks along the image boundaries.
This is visualized in Figure \ref{fig:PartitionGraphD}.
Each basic cell at layer $l$ is therefore split into four basic cells at layer $l+1$, thus satisfying the assumption from Section \ref{sec:ComputMultiScale}.

The cell size $s$ is an important parameter.
For small $s$ the composite cell problems are small (and thus easier), but one must solve a larger number of composite problems, implying more communication overhead in parallelization and requiring more domain decomposition iterations.
For large $s$ the number of composite problems is smaller, thus inducing less communication overhead and requiring less iterations, but solving the composite problems becomes more challenging.
We found that $s=4$ yielded the best results for our code.
Studying the influence of this parameter in more detail and reducing the computational overhead in the implementation are a relevant direction for future work.

\newcommand{\col}{\tn{col}}
\subsection{Visualization of primal iterates}
\label{sec:Visualization}
To get an impression of the algorithm's behaviour we visualize the evolution of the primal iterates $\iterl{\pi}$.
This is difficult since in our examples the iterates are measures on $\R^4$. But the partition structure provides some assistance.
We assign colors to the basic cells in an alternating pattern (e.g.~a checkerboard tiling). Denote by $\col_i$ the color assigned to basic cell $i \in I$ ($\col_i$ could just be a RGB vector).
We can then approximately visualize a coupling $\pi \in \Pi(\mu,\nu)$ with partial marginals $\nu_i \assign \proj_Y \pi_i$ via the color image on $Y$ generated by $\sum_{i \in I} \col_i \cdot \RadNik{\nu_i}{\nu}$. When $\pi$ is a deterministic coupling, e.g.~introduced by a Monge map $T : X \to Y$, then region $T(X_i)$ will be colored with $\col_i$, i.e.~the image would look like a Cartesian grid that was morphed by the map $T$.
When $\pi$ is non-deterministic, different $\nu_i$ can overlap and in these regions the resulting color will be a linear interpolation of the colors $\col_i$, weighted by the $\nu_i$.

\begin{figure}[t]
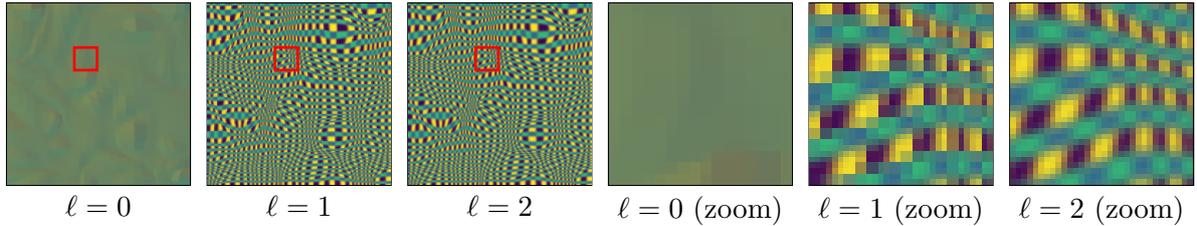

	\centering
	{\def\imgw{2.4cm}
	\begin{tikzpicture}[x=\imgw,y=\imgw,img/.style={inner sep=0pt,draw=black,line width=1pt,anchor=north west}]
	\foreach \x/\y/\l in {0/0/0,1/0/1,2/0/2} {
		\node[img] at (1.1*\x,-1.1*\y)[label=below:{$\ell=\l$}]{\includegraphics[width=\imgw]{fig/atomic_cells/iter_\l_col.png}};
		\begin{scope}[shift={(1.1*\x,-1.1*\y)},x={(1./256,0)},y={(0,-1./256)}]
			\draw[red,line width=1pt] (96,64) rectangle ++(32,32);
		\end{scope}
	}
	\foreach \x/\y/\l in {3/0/0,4/0/1,5/0/2} {
		\node[img] at (1.1*\x,-1.1*\y)[label=below:{$\ell=\l$ (zoom)}]{\includegraphics[width=\imgw]{fig/atomic_cells/iter_zoom_\l_col.png}};
	}
	\end{tikzpicture}
	}
	\caption{Primal iterates for the images from Figure \ref{fig:ExampleImages}, visualized as described in Section \ref{sec:Visualization}, on layer $l=8$, $\veps=2$, for $\ell=0,1,2$. The initial coupling is blurry due to the refinement construction via \eqref{eq:FineBasicMarginals}. For iterates after $\ell=2$ it is hard to spot further changes.}
	\label{fig:BasicCellsIter}
\end{figure}

\begin{figure}[t]
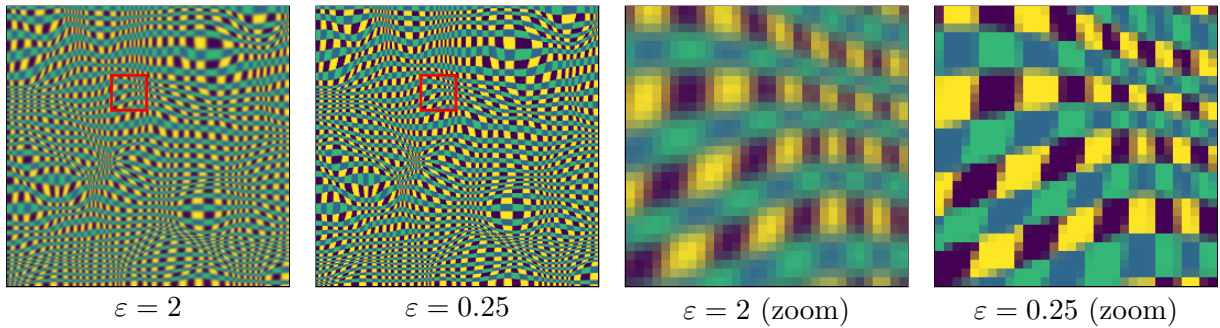

	\centering
	{\def\imgw{3.7cm}
	\begin{tikzpicture}[x=\imgw,y=\imgw,img/.style={inner sep=0pt,draw=black,line width=1pt,anchor=north west}]
	\foreach \x/\y/\l/\e in {0/0/0/2,1/0/3/0.25} {
		\node[img] at (1.1*\x,-1.1*\y)[label=below:{$\veps=\e$}]{\includegraphics[width=\imgw]{fig/atomic_cells/eps_nEps\l_col.png}};
		\begin{scope}[shift={(1.1*\x,-1.1*\y)},x={(1./256,0)},y={(0,-1./256)}]
			\draw[red,line width=1pt] (96,64) rectangle ++(32,32);
		\end{scope}
	}
	\foreach \x/\y/\l/\e in {2/0/0/2,3/0/3/0.25} {
		\node[img] at (1.1*\x,-1.1*\y)[label=below:{$\veps=\e$ (zoom)}]{\includegraphics[width=\imgw]{fig/atomic_cells/eps_zoom_nEps\l_col.png}};
	}
	\end{tikzpicture}
	}
	\caption{Primal iterates for the images from Figure \ref{fig:ExampleImages}, visualized as described in Section \ref{sec:Visualization}, on layer $l=8$, for $\veps=2$ ($\ell=4$ on that layer) and $\veps=0.25$ ($\ell=10$), see Section \ref{sec:NumericsPreliminaries} for a description of $\veps$-scaling and iteration schedule.
	For $\veps=2$ some blur between neighbouring cells is visible, for $\veps=0.25$ it is virtually gone (up to some inevitable overlap due to the discretization).}
	\label{fig:BasicCellsEps}
\end{figure}

\begin{figure}[t]
	\centering
	{\def\imgw{4.8cm}
	\begin{tikzpicture}[x=\imgw,y=\imgw,img/.style={inner sep=0pt,draw=black,line width=1pt}]
	\foreach \x/\y/\l in {0/0/3,1/0/4,2/0/5,0/1/6,1/1/7,2/1/8} {
		\node[img] at (1.1*\x,-1.15*\y)[label=below:{$l=\l$}]{\includegraphics[width=\imgw]{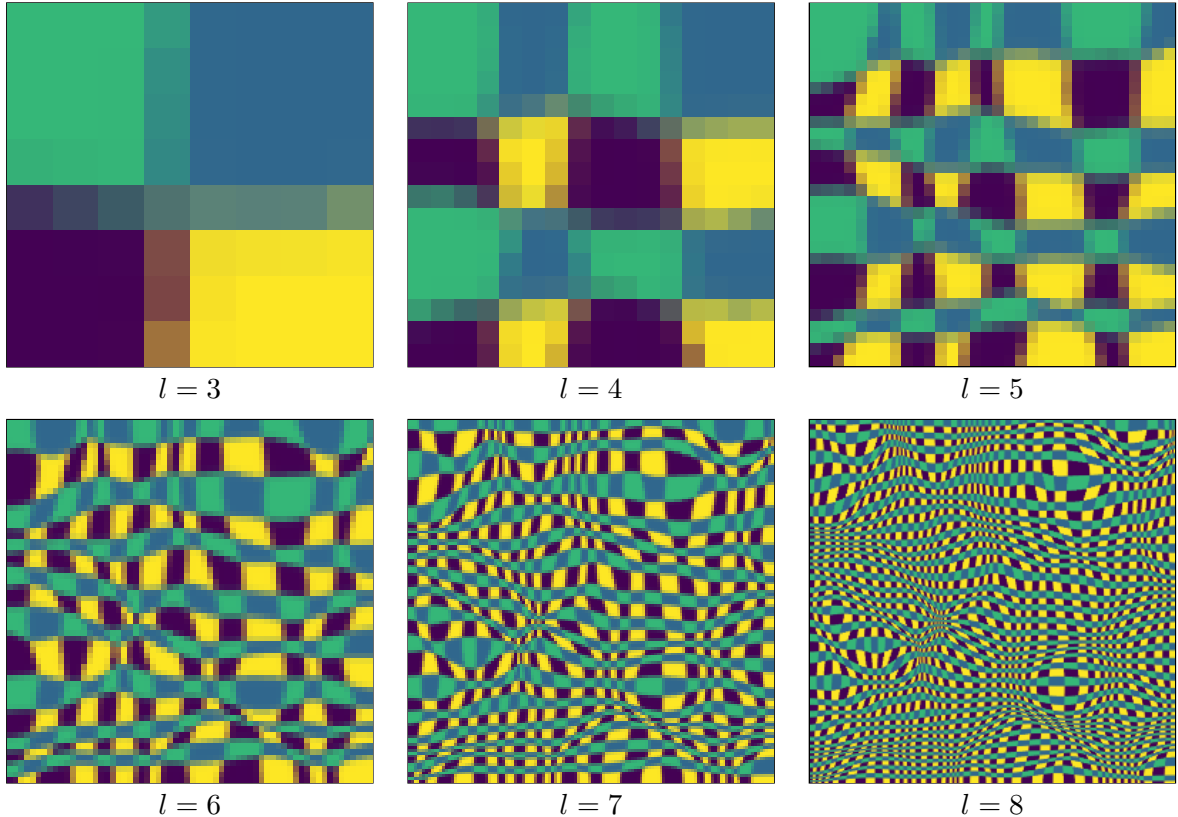}};
	}
	\end{tikzpicture}
	}
	\caption{Primal iterates for the images from Figure \ref{fig:ExampleImages}, visualized as described in Section \ref{sec:Visualization}, on layers $l=3$ to $l=8$, for $\veps=0.5 \cdot \Delta x_l$, where $\Delta x_l=2^{l-8}$ is the distance between adjacent grid points on a given layer, and $\ell=8$ on each layer, except on the finest layer, where $\veps=0.25$  and $\ell=10$.
	See Section \ref{sec:NumericsPreliminaries} for a description of $\veps$-scaling and iteration schedule.
	During each refinement a basic cell is split into four smaller basic cells and additional details in the optimal coupling become discernible.}
	\label{fig:BasicCellsLayers}
\end{figure}

Consider the optimal transport problem between the two images from Figure \ref{fig:ExampleImages} with size $256 \times 256$.
Various aspects of this problem are illustrated in Figures \ref{fig:BasicCellsIter}, \ref{fig:BasicCellsEps} and \ref{fig:BasicCellsLayers}.

Figure \ref{fig:BasicCellsIter} shows the effect of the first two iterations on the finest layer $l=8$.
The initial image is essentially uniform, since the initial marginals $\nu_i^0$, as generated with \eqref{eq:FineBasicMarginals} on basic cells $i \in I$ that were generated by sub-dividing the same coarse basic cell $j \in \hat{I}$ are, up to a constant factor, identical.
After the first iteration (on $\partA$) a grid pattern reemerges, after the second iteration (on $\partB$), it is hard to spot further changes (cf.~Figure \ref{fig:BasicCellsEps}, where for $\veps=2$ the iterate after four iterations is shown).

The effect of $\veps$-scaling within a given layer is visualized in Figure \ref{fig:BasicCellsEps}. For $\veps=2$ the checkerboard pattern is blurred between neighbouring cells. For $\veps=0.25$, neighbouring cells overlap by at most one pixel, which would also be true in the discrete unregularized setting (recall that our pixels do not all carry the same mass, so the discrete unregularized optimal coupling is usually not deterministic).
This demonstrates that the algorithm can be run until very low regularization parameters.

Finally, Figure \ref{fig:BasicCellsLayers} demonstrates the evolution of the multi-scale scheme.
It shows the final coupling on each layer for the smallest $\veps$ used on that layer, just before refinement (except for the finest layer). One can trace how increasingly finer structures of the optimal coupling emerge as the resolution improves.

\subsection{Quantitative evaluation}
\label{sec:NumericsQuantitative}
\begin{table}[t]
\centering
\scriptsize
\begin{tabular}{|@{\,}l@{\,}|*{6}{@{\,}r@{\,\,}}|}
\hline
image size & $64 \times 64$ & $128 \times 128$ & $256 \times 256$ & $512 \times 512$ & $1024 \times 1024$ & $2048 \times 2048$\\
\hline
\multicolumn{7}{|@{\,}l@{\,}|}{\textbf{domain decomposition} (stopping criterion: $\textnormal{$L^1$-error} \leq 10^{-4}$)} \\
\hline
\multicolumn{7}{|@{\,}l@{\,}|}{total runtime in seconds} \\
\hline
time (1 worker) & $(3.7\!\pm\!0.0)\tn{E{0}}$ & $(1.5\!\pm\!0.0)\tn{E{1}}$ & $(5.9\!\pm\!0.0)\tn{E{1}}$ & $(2.4\!\pm\!0.0)\tn{E{2}}$ & $(1.1\!\pm\!0.0)\tn{E{3}}$ & $(4.8\!\pm\!0.0)\tn{E{3}}$ \\
time (2 workers) & $(2.1\!\pm\!0.0)\tn{E{0}}$ & $(7.9\!\pm\!0.1)\tn{E{0}}$ & $(3.1\!\pm\!0.0)\tn{E{1}}$ & $(1.2\!\pm\!0.0)\tn{E{2}}$ & $(5.7\!\pm\!0.1)\tn{E{2}}$ & $(2.5\!\pm\!0.0)\tn{E{3}}$ \\
time (3 workers) & $(1.7\!\pm\!0.0)\tn{E{0}}$ & $(5.7\!\pm\!0.1)\tn{E{0}}$ & $(2.1\!\pm\!0.0)\tn{E{1}}$ & $(8.5\!\pm\!0.1)\tn{E{1}}$ & $(3.8\!\pm\!0.0)\tn{E{2}}$ & $(1.7\!\pm\!0.0)\tn{E{3}}$ \\
time (4 workers) & $(1.5\!\pm\!0.0)\tn{E{0}}$ & $(4.6\!\pm\!0.1)\tn{E{0}}$ & $(1.7\!\pm\!0.0)\tn{E{1}}$ & $(6.5\!\pm\!0.0)\tn{E{1}}$ & $(2.9\!\pm\!0.0)\tn{E{2}}$ & $(1.3\!\pm\!0.0)\tn{E{3}}$ \\
time (5 workers) & $(1.5\!\pm\!0.0)\tn{E{0}}$ & $(4.0\!\pm\!0.1)\tn{E{0}}$ & $(1.4\!\pm\!0.0)\tn{E{1}}$ & $(5.4\!\pm\!0.0)\tn{E{1}}$ & $(2.4\!\pm\!0.0)\tn{E{2}}$ & $(1.1\!\pm\!0.0)\tn{E{3}}$ \\
time final layer (5 w.) & $(8.7\!\pm\!0.4)\tn{E{-1}}$ & $(2.6\!\pm\!0.0)\tn{E{0}}$ & $(1.0\!\pm\!0.0)\tn{E{1}}$ & $(4.1\!\pm\!0.0)\tn{E{1}}$ & $(2.0\!\pm\!0.0)\tn{E{2}}$ & $(8.5\!\pm\!0.0)\tn{E{2}}$ \\
\hline
\multicolumn{7}{|@{\,}l@{\,}|}{times for specific functions in seconds (5 workers)} \\
\hline
\textsc{Sinkhorn} & $(1.3\!\pm\!0.0)\tn{E{0}}$ & $(3.6\!\pm\!0.1)\tn{E{0}}$ & $(1.2\!\pm\!0.0)\tn{E{1}}$ & $(4.9\!\pm\!0.0)\tn{E{1}}$ & $(2.3\!\pm\!0.0)\tn{E{2}}$ & $(1.0\!\pm\!0.0)\tn{E{3}}$ \\\textsc{BalanceMeasures} & $(1.3\!\pm\!0.0)\tn{E{-1}}$ & $(2.6\!\pm\!0.1)\tn{E{-1}}$ & $(7.3\!\pm\!0.3)\tn{E{-1}}$ & $(2.2\!\pm\!0.1)\tn{E{0}}$ & $(7.0\!\pm\!0.1)\tn{E{0}}$ & $(2.6\!\pm\!0.0)\tn{E{1}}$ \\\textsc{Truncate} & $(1.4\!\pm\!0.1)\tn{E{-2}}$ & $(5.7\!\pm\!0.5)\tn{E{-2}}$ & $(2.1\!\pm\!0.1)\tn{E{-1}}$ & $(8.1\!\pm\!0.2)\tn{E{-1}}$ & $(3.2\!\pm\!0.1)\tn{E{0}}$ & $(1.3\!\pm\!0.0)\tn{E{1}}$ \\refinement & $(3.3\!\pm\!0.2)\tn{E{-2}}$ & $(1.0\!\pm\!0.1)\tn{E{-1}}$ & $(3.2\!\pm\!0.1)\tn{E{-1}}$ & $(1.2\!\pm\!0.0)\tn{E{0}}$ & $(5.0\!\pm\!0.0)\tn{E{0}}$ & $(2.1\!\pm\!0.0)\tn{E{1}}$ \\
\hline
\multicolumn{7}{|@{\,}l@{\,}|}{total number of entries in all $(\nu_i)_{i \in I}$} \\
\hline
maximum & $(8.4\!\pm\!0.2)\tn{E{4}}$ & $(3.5\!\pm\!0.0)\tn{E{5}}$ & $(1.4\!\pm\!0.0)\tn{E{6}}$ & $(5.5\!\pm\!0.0)\tn{E{6}}$ & $(2.1\!\pm\!0.0)\tn{E{7}}$ & $(8.3\!\pm\!0.0)\tn{E{7}}$ \\final & $(1.7\!\pm\!0.0)\tn{E{4}}$ & $(6.7\!\pm\!0.1)\tn{E{4}}$ & $(2.6\!\pm\!0.0)\tn{E{5}}$ & $(1.0\!\pm\!0.0)\tn{E{6}}$ & $(3.9\!\pm\!0.0)\tn{E{6}}$ & $(1.6\!\pm\!0.0)\tn{E{7}}$ \\
\hline
\multicolumn{7}{|@{\,}l@{\,}|}{solution quality} \\
\hline
relative PD gap & $(1.9\!\pm\!1.6)\tn{E{-4}}$ & $(1.0\!\pm\!0.6)\tn{E{-4}}$ & $(1.0\!\pm\!0.7)\tn{E{-4}}$ & $(4.3\!\pm\!3.2)\tn{E{-5}}$ & $(1.0\!\pm\!1.3)\tn{E{-4}}$ & $(6.2\!\pm\!3.0)\tn{E{-5}}$ \\$L^1$ marginal error $X$ & $(5.1\!\pm\!0.3)\tn{E{-5}}$ & $(5.0\!\pm\!0.2)\tn{E{-5}}$ & $(4.9\!\pm\!0.1)\tn{E{-5}}$ & $(4.7\!\pm\!0.1)\tn{E{-5}}$ & $(4.8\!\pm\!0.1)\tn{E{-5}}$ & $(4.4\!\pm\!0.1)\tn{E{-5}}$ \\$L^1$ marginal error $Y$ & $(7.8\!\pm\!25.8)\tn{E{-7}}$ & $(1.2\!\pm\!1.3)\tn{E{-8}}$ & $(3.7\!\pm\!0.2)\tn{E{-8}}$ & $(1.5\!\pm\!0.1)\tn{E{-7}}$ & $(4.4\!\pm\!0.1)\tn{E{-7}}$ & $(1.1\!\pm\!0.0)\tn{E{-6}}$ \\
\hline
\multicolumn{7}{|@{\,}l@{\,}|}{\textbf{single sparse multi-scale Sinkhorn} (stopping criterion: $\textnormal{$L^\infty$-error} \leq 10^{-7}$)} \\
\hline
time & $(3.6\!\pm\!0.6)\tn{E{0}}$ & $(2.7\!\pm\!0.7)\tn{E{1}}$ & $(1.3\!\pm\!0.6)\tn{E{2}}$ & $(2.8\!\pm\!1.2)\tn{E{2}}$ & -  & -  \\\# kernel entries, max & $(6.1\!\pm\!0.4)\tn{E{5}}$ & $(2.8\!\pm\!0.1)\tn{E{6}}$ & $(1.2\!\pm\!0.0)\tn{E{7}}$ & $(5.1\!\pm\!0.0)\tn{E{7}}$ & -  & -  \\\# kernel entries, final & $(9.3\!\pm\!0.3)\tn{E{4}}$ & $(4.0\!\pm\!0.1)\tn{E{5}}$ & $(1.7\!\pm\!0.0)\tn{E{6}}$ & $(6.9\!\pm\!0.1)\tn{E{6}}$ & -  & -  \\$L^1$ marginal error $X$ & $(1.1\!\pm\!0.3)\tn{E{-4}}$ & $(3.1\!\pm\!0.9)\tn{E{-4}}$ & $(7.5\!\pm\!2.1)\tn{E{-4}}$ & $(2.6\!\pm\!0.7)\tn{E{-3}}$ & -  & -  \\relative dual score & $(-1.8\!\pm\!2.3)\tn{E{-4}}$ & $(-1.4\!\pm\!0.5)\tn{E{-3}}$ & $(-4.9\!\pm\!1.7)\tn{E{-3}}$ & $(-6.9\!\pm\!2.9)\tn{E{-3}}$ & -  & -  \\
\hline
\end{tabular}

\caption{Summary of average performance (with standard deviation) of the domain decomposition algorithm and a single sparse Sinkhorn algorithm \cite{SchmitzerScaling2019} for comparison. For each image size results are averaged over 45 problem instances (Section \ref{sec:NumericsPreliminaries}). %
The relative PD gap and the relative dual score are defined in \eqref{eq:RelPDGap} and \eqref{eq:RelDualScore}.}
\label{tab:Summary}
\end{table}

\begin{figure}[t]
	\centering
	\includegraphics[]{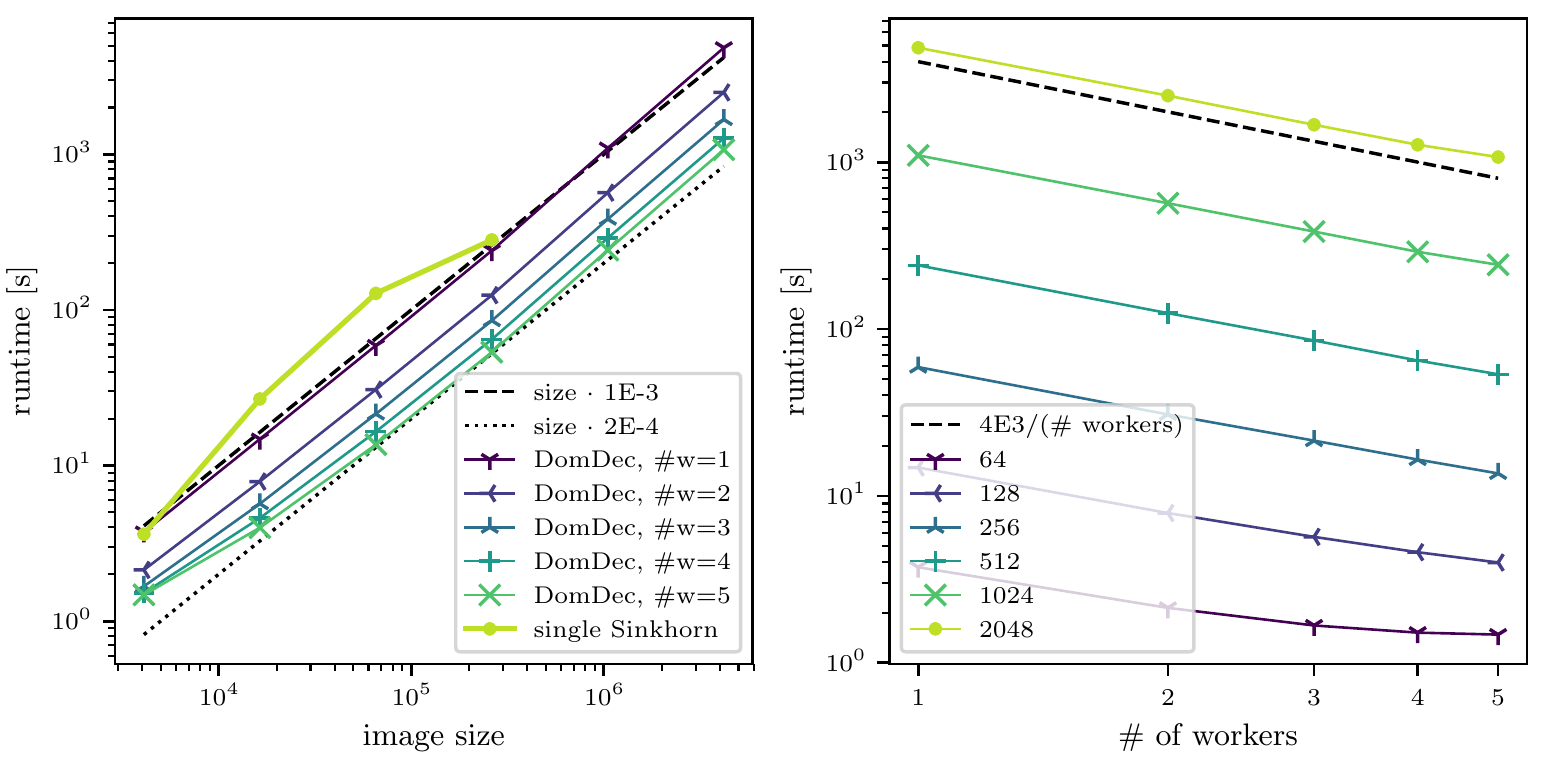}
	\caption{Average runtimes of domain decomposition algorithm. %
	The runtime increases approximately linearly (or slightly super-linearly) with the size of the marginal images, allowing the solution of large problems. The runtime of a single Sinkhorn algorithm (with different stopping criterion) is shown for comparison (see Section \ref{sec:NumericsSingle} for details).
	Runtime decreases approximately in inverse proportionality with the number of worker threads (up to five workers), indicating efficient parallelization (except for small images where overhead is more significant and thus the effect of parallelization saturates earlier).}
	\label{fig:Runtimes}
\end{figure}

\begin{figure}[t]
	\centering
	\includegraphics[]{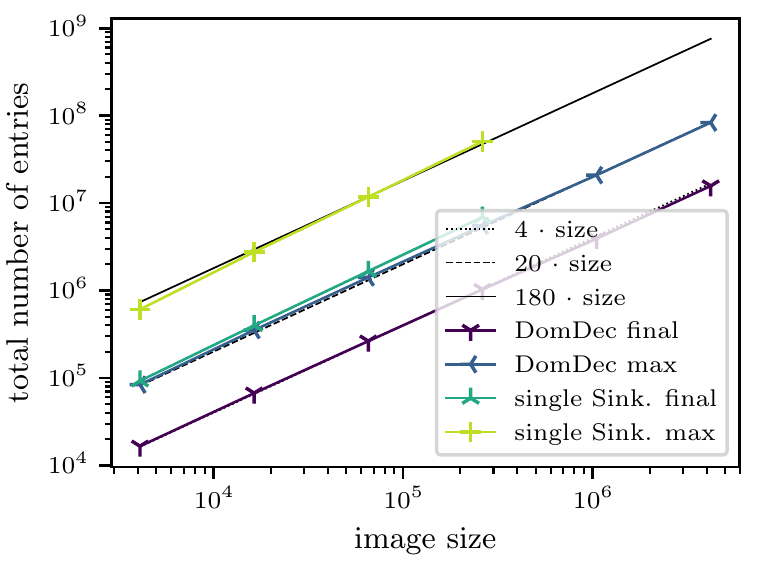}
	\caption{Number of non-zero entries to store the partial marginal list $(\nu_i)_{i \in I}$ at the finest layer. Maximal number (for largest value of $\veps$) and final number are shown. %
	For comparison the number of non-zero entries in the truncated (stabilized) kernel matrix of the single Sinkhorn algorithm is shown (again, maximal and final number on finest layer). %
	All numbers scale linearly in the image size, indicating that the truncation schemes efficiently reduce memory demand. For domain decomposition the final number of entries per pixel is approximately four, i.e.~the coupling is rather concentrated and the effect of entropic blurring is weak. Maximal and final entry numbers are significantly lower with domain decomposition, the maximal number (which represents the memory bottleneck) is almost reduced by a factor 10.}
	\label{fig:Sparsity}
\end{figure}

Having convinced ourselves by means of visualizations that the entropic domain decomposition algorithm seems to work as intended, we now turn to a more quantitative analysis of its performance.
A comparison to the performance of a single sparse multi-scale Sinkhorn algorithm \cite{SchmitzerScaling2019} is provided in Section \ref{sec:NumericsSingle}.

\paragraph{Runtime.} We measure the runtime of the algorithm on different problem sizes, with the number of worker threads ranging from 1 to 5. For simplicity, for a single worker thread, the implementation still uses two threads (one master, one worker), but the performance cannot be distinguished from a true single-thread implementation, since the master thread is essentially idle in this case.
In addition we are interested in the time required for solving the final (i.e.~finest) layer, and for the times required by the main sub-routines, i.e.~the Sinkhorn solver, the measure balancing, truncation and refinement.
These are listed in Table \ref{tab:Summary} and the total runtimes are visualized in Figure \ref{fig:Runtimes}.

We find that the runtime scales approximately linearly (or at most slightly super-linear) with the size of the marginals which is crucial for solving large problems.
The time required for solving the final layer comprises approximately 80\% of the total time, indicating that even with a good initialization, solving the finest layer is a challenging problem.
The time spent on the \textsc{Sinkhorn} sub-solver for the composite cell problems accounts for approximately 96\% of the runtime for large problems, which means that other functions cause relatively little overhead.
On small problems these ratios are a bit lower, since the overhead of the multi-scale scheme is more significant.

For large problems the runtime is approximately proportional to the inverse number of workers (up to five workers), indicating that parallelization is computationally efficient. Again, for small problems, computational overhead becomes more significant, such that the benefit of parallelization saturates earlier.
For large problems we estimate that MPI communication accounts for approximately 3\% of time in a domain decomposition iteration.

\paragraph{Sparsity.}
In addition to runtime the memory footprint of the algorithm is an important criterion for practicability.
It is dominated by storing the list $(\nu_i)_{i \in I}$ (after truncation at $10^{-15}$, see Section \ref{sec:ComputOverview}), since only a few small sub-problems are solved simultaneously at any given time.
We are therefore interested in the number of non-zero entries in this list throughout the algorithm. 
This number is usually maximal on the finest layer (most pixels) and for the highest value of $\veps$ on that layer (with strongest entropic blur), which therefore constitutes the memory bottleneck of the algorithm.

This maximal number and the final number of entries are listed in Table \ref{tab:Summary} and visualized in Figure \ref{fig:Sparsity}.
We find that both scale approximately linear with marginal size. Again, this is a prerequisite for solving large problems.
Maximally the algorithm requires approximately 20 entries per pixel in the first image, and approximately 4 per pixel upon convergence. The latter indicates that entropic blur is indeed strongly reduced in the final iterate.

\paragraph{Solution quality.}
The data reported in Table \ref{tab:Summary} also confirms that the algorithm provides high-quality approximate solutions of the problem.
The relative primal-dual (PD) gap is defined as (see \eqref{eq:EntropicOTDualObjective} for the definition of $J$)
\begin{align}
	\label{eq:RelPDGap}
	\frac{\KL(\pi|\kernel)-J(u,v)}{\KL(\pi|\kernel)-\|\kernel\|},
\end{align}
i.e.~we divide the primal-dual gap by the primal score (where, for simplicity we subtract the constant term $\|\kernel\|$). The dual variables $(u,v)$ are obtained from the composite duals by gluing, as discussed in Section \ref{sec:ComputEvaluatePDGap}.
We find that this is on the order of $10^{-4}$, i.e.~a high relative accuracy.

The $L^1$-error of the $X$-marginal of the primal iterate is consistently even smaller than the prescribed bound $\err=10^{-4}$ (cf.~Section \ref{sec:ComputBalanceMeasures}). This happens since we do not check the stopping criterion for the \textsc{Sinkhorn} sub-solver after every iteration.
The $L^1$-error of the $Y$-marginal is not (numerically) zero, since a small error accumulates over time during the several additions and subtractions of the partial marginals and in the \textsc{BalanceMeasures} function. But since it is still substantially below the $X$-error we do not consider this to be a problem. Additional stabilization-routines could be added to the algorithm, should it become necessary.

\subsection{Comparison to single Sinkhorn algorithm}
\label{sec:NumericsSingle}

As a reference we compare the performance of the domain decomposition algorithm to that of a single Sinkhorn algorithm (with adaptive kernel truncation, coarse-to-fine scheme and $\veps$-scaling as in \cite{SchmitzerScaling2019}). The results are also summarized in Table \ref{tab:Summary}, and visualized in Figures \ref{fig:Runtimes} and \ref{fig:Sparsity}.
The single Sinkhorn algorithm was limited to sizes up to $512 \times 512$ due to numerical stability and memory issues (see below for details).

Both algorithms have several parameters that influence the trade-off between accuracy and runtime.
One important parameter is the stopping criterion for the sub-solver Sinkhorn algorithm in the former, and for the global single Sinkhorn algorithm in the latter.
Another one is the truncation threshold for partial marginals (cf.~Section \ref{sec:ComputOverview}) in the former, and for the (stabilized) kernel matrix in the latter \cite[Section 3.3]{SchmitzerScaling2019}.

\paragraph{Stopping criterion.} The stopping criterion for the domain decomposition solver was picked such that the global $L^1$-error of the marginals is below $\err=10^{-4}$ (cf.~Sections \ref{sec:ComputBalanceMeasures} and \ref{sec:NumericsPreliminaries}).
We observe that if we pick the same criterion for the single Sinkhorn solver, it takes much longer than the domain decomposition method (even without parallelization) and the runtime scales significantly super-linearly with marginal size.

A more detailed analysis reveals that the single Sinkhorn algorithm spends a major part of time on the largest $\veps$ on the finest layer, i.e.~immediately after the last refinement.
We conjecture that this is due to the fact that the single Sinkhorn algorithm represents the primal iterates only implicitly through the dual iterates $(u,v)$ via $(u \otimes v) \cdot \kernel$ (cf.~Section \ref{sec:ReminderEntropicOT}).
During layer-refinement only the dual variables are refined explicitly. This can introduce quite non-local marginal errors in the primal iterate (i.e.~mass must be moved far to correct them) that take many iterations to correct.
We tried various heuristics to reduce this effect (e.g.~more sophisticated interpolation of the dual variables during refinement) but unfortunately without success.

In contrast, the domain decomposition method keeps explicitly track of primal and dual (partial) iterates, both of which are refined separately during layer transitions. In particular the total mass of the refined primal iterate within each refined basic cell is correct after refinement (cf.~\eqref{eq:FineBasicMarginals} and Proposition \ref{prop:FineBasicMarginals}) and thus fewer global mass movements are required for correcting these errors.

Therefore, as in \cite{SchmitzerScaling2019}, numerically we use the $L^\infty$ error as stopping criterion for the single Sinkhorn algorithm. We set the stopping accuracy to $10^{-7}$, such that (without parallelization) both algorithms have approximately the same runtime.

\paragraph{Sparsity.}
To perform its iterations the single Sinkhorn algorithm must keep track of a truncated (stabilized) kernel matrix, i.e.~one column of non-zero entries for each pixel of the first marginal.
Following \cite[Section 3.3, Equation (3.8)]{SchmitzerScaling2019} we keep entries of the (stabilized) kernel matrix where $(u \otimes v) \cdot \sKernel \geq \theta$ with the choice $\theta=10^{-10}$.

The domain decomposition method (in the form of Algorithm \ref{alg:DomDecComput}) only needs to store a column of non-zero entries for each basic cell (of size $s \times s$) in the first marginal.
Thus, to represent the primal iterate with comparable accuracy, the number of non-zero entries required in the domain decomposition method is reduced by a factor $s^{-2}$.

In this sense, for the parameter choices in our experiments, the domain decomposition algorithm represents the primal iterates slightly more accurately.
At the same time, the number of entries to store the final iterate is reduced to approximately 15\% compared to the single Sinkhorn algorithm. The maximum number of entries is even reduced to approximately 11\%.
On problem size $2048 \times 2048$ the single Sinkhorn algorithm ran out of memory on our test machine.

\paragraph{Solution quality.}
As discussed above, to obtain similar runtimes for both algorithms, a less strict stopping criterion had to be chosen for the single Sinkhorn algorithm.
Therefore, it produces higher $L^1$-marginal errors for the $X$-marginal that increase with problem size. The $Y$-marginal error is (numerically) zero, since the algorithm was terminated after a $Y$-iteration (cf.~Remark \ref{rem:Sinkhorn}).
The relative dual score reported in Table \ref{tab:Summary} is defined by
\begin{align}
\label{eq:RelDualScore}
\frac{J(u_{\tn{single}},v_{\tn{single}})-J(u_{\tn{domdec}},v_{\tn{domdec}})}{J(u_{\tn{domdec}},v_{\tn{domdec}})-\|\kernel\|},
\end{align}
cf.~\eqref{eq:RelPDGap}, where $(u_{\tn{domdec}},v_{\tn{domdec}})$ refers to the (glued) final dual variables of the domain decomposition method (Section \ref{sec:ComputEvaluatePDGap}) and $(u_{\tn{single}},v_{\tn{single}})$ to the final dual iterates of the single Sinkhorn algorithm. The fact that this number is negative indicates that the domain decomposition provides better dual candidates than the single Sinkhorn.

\paragraph{Numerical stability.}
For small regularization the standard Sinkhorn algorithm becomes numerically unstable due to the finite precision of discretized numbers on computers. This can be remedied to a large extent by the absorption technique \cite{SchmitzerScaling2019}.
But refinement between layers remains a delicate point since the interpolated dual variables from the coarser layer may not provide a perfect initial guess at the finer layer.
This can lead to numerically zero or empty columns and rows in the (stabilized) kernel matrix and thus will lead to invalid entries in the scalings $(u,v)$.
While it is possible to fix these cases, on large problems this is relatively cumbersome and computationally inefficient. We therefore decided to abort the algorithm in these cases instead.
In our experiments this concerned 3 of the 45 examples of size $512 \times 512$ and the majority of the $1024 \times 1024$ examples (which we have therefore removed from the comparison).

These issues may also arise in the domain decomposition algorithm but there they are localized to small problems on composite cells which are easy to fix (e.g.~by temporarily increasing $\veps$). We have implemented a corresponding safe-guard. In all our examples it was invoked on two $1024 \times 1024$ instances.

\paragraph{Summary.}
Clearly the domain decomposition method is superior to the single Sinkhorn algorithm in various aspects.
In comparison, the domain decomposition method (without parallelization) produces more accurate primal and dual iterates in comparable time, with approximately 11\% of the memory.
It runs more reliably on large problems, since numerical hickups of the Sinkhorn algorithm can be fixed locally.
On top of that the domain decomposition method allows for real non-local parallelization, whereas the Sinkhorn algorithm can only be parallelized very locally over the matrix-vector multiplications.

\section{Conclusion and outlook}
\label{sec:Conclusion}

In this article we studied Benamou's domain decomposition algorithm for optimal transport in the entropy regularized setting.
The key observation is that the regularized version converges to the unique globally optimal solution under very mild assumptions.
We proved linear convergence of the algorithm w.r.t~the KL-divergence and illustrated the (potentially very slow) rates with numerical examples.

We then argued that on `geometric problems' the algorithm should converge much faster.
To confirm this experimentally, we discussed a practical, efficient version of the algorithm with reduced memory footprint, numerical safeguards for approximation errors, primal-dual certificates and a coarse-to-fine scheme.
With this we were able to compute optimal transport between images with $\approx$ 4 megapixels on a standard desktop computer, matching two megapixel images in approximately four minutes.
Even without parallelization, the algorithm compared favourably to a single Sinkhorn algorithm in terms of runtime, memory, solution quality and numerical reliability.
With parallelization the runtime was efficiently reduced with little communication overhead.

A practical open question for future research is the development of more sophisticated and efficient implementations, e.g.~using GPUs and/or multiple machines, with more sophisticated parallelization structure (e.g.~no single unit needs to keep track of the full problem).
This should then be tested on 3D problems and with more general cost functions.
In the course of this, parameters such as the basic cell size should be studied in more detail.

On the theoretical side the convergence analysis on `geometric problems' is an interesting challenge, to provide a more thorough understanding for the low number of required iterations. We conjecture that a `discrete Helmholtz decomposition' as discussed in Section \ref{sec:ComputEvaluatePDGap} may become relevant for this. The relation between optimal transport and the Helmholtz decomposition was already noted in \cite{MonotoneRerrangement-91} and exploited numerically in \cite{OptimalTransportWarping}.

\paragraph{Acknowledgement.} This work was supported by the Emmy Noether programme and the SFB TRR 109 of the DFG.
\appendix
\section{Appendix}

\begin{proof}[Proof of Proposition \ref{prop:EntropicOTBasic}]
	\hfill \\
	\eqref{item:EntropicOTBasic:Unique} is provided by \cite[Theorem 2.1]{csiszar1975}. Existence of $(u^\ast,v^\ast)$ in \eqref{item:EntropicOTBasic:Scaling} is established by \cite[Theorem 3]{RueschendorfThomsen-Schroedinger1993}.
	Since $\pi^\ast$ is unique, $u^\ast \otimes v^\ast$ is unique $K$-almost everywhere. Since $k$ is strictly bounded away from zero, this translates to $(\mu \otimes \nu)$-a.e.~uniqueness and finally to $\mu$-a.e.~and $\nu$-a.e.~uniqueness of $u^\ast$ and $v^\ast$ up to re-scaling.
	The integral equations in \eqref{item:EntropicOTBasic:Marginals} follow from the marginal constraints $\pi^\ast \in \Pi(\hat{\mu},\hat{\nu})$ and the definition of $\kernel$.
	Using now that $\exp(-\|c\|_\infty/\veps) \leq \sKernel(x,y) \leq 1$ for all $(x,y) \in X \times Y$, one easily deduces that for all $x \in X$,
	\[
	\exp(-\|c\|_\infty/\veps) \int_Y v^\ast(y)\,\diff\nu(y) \leq \int_Y \sKernel(x,y)v^\ast(y)\,\diff\nu(y) \leq \int_Y v^\ast(y)\,\diff\nu(y),
	\]
	which gives $v^\ast \in L^1(Y, \nu)$ and $\|v^\ast\|_{L^1(Y,\nu)}>0$ (otherwise there would exist no $u^\ast$ to satisfy the left equation of \eqref{eq:ExtremalityCondition}, recall that we assumed $\|\hat{\mu}\|>0$).
	Analogously we see that $u^\ast \in L^1(X, \mu)$ and $\|u^\ast\|_{L^1(X,\mu)}>0$, so that \eqref{eq:uBound} and \eqref{eq:vBound} follow from \eqref{eq:ExtremalityCondition}.
	Since $\RadNik{\hat{\mu}}{\mu}$ and $\RadNik{\hat{\nu}}{\nu}$ were assumed to be (essentially) bounded this implies $u^\ast \in L^\infty_+(X,\mu)$ and $\nu^\ast \in L^\infty_+(Y,\nu)$.
	Finally, by the first inequality of \eqref{eq:uBound},
	\begin{align*}
	\int_X |\log u^\ast|\,\diff \hat{\mu} = \int_X |\log u^\ast| \RadNik{\hat{\mu}}{\mu}\,\diff \mu
	\leq \|v^\ast\|_{L^1(Y,\nu)} \int_X |\log u^\ast| u^\ast \diff \mu < \infty
	\end{align*}
	where the last inequality is due to the $\mu$-essential upper bound on $u^\ast$ and the lower-boundedness of $\R_+ \ni s \mapsto s \log s$ (with limit $0$ for $s=0$). By this, and an analogous argument for $v^\ast$, we find $\log u^\ast \in L^1(X,\hat{\mu})$, $\log v^\ast \in L^1(Y,\hat{\nu})$.
\end{proof}

\begin{proof}[Proof of Proposition \ref{prop:Duality}]
	\hfill \\
	\eqref{item:Duality:Defined}. Since $u \in L^\infty_+(X,\mu)$, $\log u$ is also $\mu$-essentially bounded from above, thus the first integral is either finite or $-\infty$. The same holds for the second term. The third and fourth terms of $J$ are always finite. Hence, the value of $J$ is unambiguously defined and contained in $[-\infty,\infty)$.
	
	\noindent
	\eqref{item:Duality:Gap}. First observe that $\KL(\pi|\kernel) > -\infty$ and $J(u, v) < +\infty$ by definition. If $J(u, v) = -\infty$ or $\KL(\pi|\kernel) = +\infty$ there is nothing to prove. Thus, let us assume both values are finite and directly compute
	\begin{align*}
	\KL(\pi|\kernel) - J(u, v)
	& =  \int_{X\times Y} \big[
	\varphi\left(\RadNik{\pi}{\kernel}\right)
	+ \left(u \otimes v-1\right)
	\big]\,\diff \kernel
	- \Big( \underbrace{\int_X \log u\,\diff \hat{\mu} + \int_Y \log v\,\diff \hat{\nu}}_{=\int_{X \times Y} (\log u \oplus \log v)\,\diff \pi} \Big) \\
	&= \int_{\substack{X\times Y:\\u \otimes v>0}} \big[
	\varphi\left(\RadNik{\pi}{\kernel}\right) + \underbrace{\varphi^\ast\left(\log u \oplus \log v\right)}_{=u \otimes v-1} - (\log u \oplus \log v)\cdot \frac{\diff\pi}{\diff \kernel}  \big]\,\diff \kernel \\
	&\quad + \int_{\substack{X\times Y:\\u \otimes v = 0}} \underbrace{\big[
		\varphi\left(0\right) -1 \big]}_{=0}\,\diff \kernel \geq 0
	\end{align*}
	Here we used that $J(u,v)>-\infty$ implies $u>0$ $\hat{\mu}$-a.e.~and $v>0$ $\hat{\nu}$-a.e., so $\pi \in \Pi(\hat{\mu},\hat{\nu})$ has no mass in the region of the second integral (where $u\otimes v=0$). The integrand in the first integral is non-negative by the Fenchel--Young inequality and we added the second (vanishing) integral to emphasize that all contributions were considered.
	
	\noindent
	\eqref{item:Duality:OptimalityCondition}. Let $\pi^\ast$ be of the described form with scalings $(u^\ast,v^\ast)$. Arguing as in Proposition \ref{prop:EntropicOTBasic} we find that $\log u^\ast \in L^1(X,\hat{\mu})$, $\log v^\ast \in L^1(X,\hat{\nu})$ and consequently
	\begin{align*}
	\KL(\pi^\ast|\kernel) & = \int_X \log u^\ast\,\diff \hat{\mu} + \int_Y \log v^\ast\,\diff \hat{\nu} - \|\pi^\ast\| + \|K\|=J(u^\ast,v^\ast)
	\end{align*}
	and so by \eqref{item:Duality:Gap}, $\pi^\ast$ and $(u^\ast,v^\ast)$ are optimal for \eqref{eq:EntropicOTPrimal} and \eqref{eq:EntropicOTDual}.
	Existence of such a $\pi^\ast$ is provided by Proposition \ref{prop:EntropicOTBasic}, and therefore, maximizers for \eqref{eq:EntropicOTDual} exist.
\end{proof}

\begin{proof}[Proof of Lemma \ref{lem:ScalingLemma}]
	\hfill \\
	\eqref{item:ScalingLemma:Bounds}. The result follows from Proposition \ref{prop:EntropicOTBasic} (\ref{item:EntropicOTBasic:Marginals}) applied to $\hat{\mu} = \mu_J$ and $\hat{\nu} = \iterl{\nu}_J$. Hence, $\log \iterl{u}_J \in L^1(X, \mu_J)$ and $\log \iterl{v}_J \in L^1(Y, \iterl{\nu}_J)$, and in particular $\|\iterl{u}_J\|_{L^1(X,\mu)}>0$ and $\|\iterl{v}_J\|_{L^1(Y,\nu)}>0$. The bounds on $\iterl{u}_J$ and $\iterl{v}_J$ follow from \eqref{eq:uBound} and \eqref{eq:vBound} in the form
	\begin{align*}
	\tfrac{1}{\|\iterl{v}_J\|_1} \RadNik{\mu_J}{\mu}(x) & \leq \iterl{u}_J(x) \leq \tfrac{\exp(\|c\|_\infty/\veps)}{\|\iterl{v}_J\|_1} \RadNik{\mu_J}{\mu}(x) \quad \text{for } \mu_J\text{-a.e.~}x \in X_J, \\
	\tfrac{1}{\|\iterl{u}_J\|_1} \RadNik{\iterl{\nu}_J}{\nu}(y) & \leq \iterl{v}_J(y) \leq \tfrac{\exp(\|c\|_\infty/\veps)}{\|\iterl{u}_J\|_1} \RadNik{\iterl{\nu}_J}{\nu}(y) \quad \text{for } \nu\text{-a.e.~}y \in Y,
	\end{align*}
	combined with $(\diff\mu_J/\diff\mu)(x) = 1$ for $\mu_J$-a.e.~$x \in X_J$ and $(\diff\iterl{\nu}_J/\diff\nu)(y) \leq 1$ for $\nu$-a.e.~$y \in Y$. Observe that the bounding constants, in this case $\underbar{C} = \min({1}/{\|\iterl{v}_J\|_1}, {1}/{\|\iterl{u}_J\|_1})$ and $\bar{C} = \max({\exp(\|c\|_\infty/\veps)}/{\|\iterl{v}_J\|_1}, {\exp(\|c\|_\infty/\veps)}/{\|\iterl{u}_J\|_1})$, generally depend on the scaling choice for the duals (Proposition \ref{prop:EntropicOTBasic} (\ref{item:EntropicOTBasic:Scaling})) and are not universal over cells and iterations.
	
	\noindent
	\eqref{item:ScalingLemma:Ratio}. First observe that the scaling factors $\iterl{u}_J$, $\iterl{v}_J$ satisfy
	\begin{equation}\label{eq:ExtremalityConditionIterl}
	\left\{
	\begin{aligned}
	\iterl{u}_J(x) \int_Y \sKernel(x,y)\iterl{v}_J(y)\,\diff\nu(y) &= \RadNik{{\mu_J}}{\mu}(x) \qquad \text{for $\mu$-a.e.~$x \in X$,} \\
	\iterl{v}_J(y) \int_X \sKernel(x,y)\iterl{u}_J(x)\,\diff\mu(x) &= \RadNik{\iterl{\nu_J}}{\nu}(y) \qquad \text{for  $\nu$-a.e.~$y \in Y$,}
	\end{aligned}
	\right.
	\end{equation}
	thanks to Proposition \ref{prop:EntropicOTBasic} \eqref{item:EntropicOTBasic:Marginals} with $\hat{\mu} = \mu_J$ and $\hat{\nu} = \iterl{\nu}_J$. The first equation in \eqref{eq:ExtremalityConditionIterl} yields
	\begin{align*}
	\frac{\iterl{u_J}(x_1)}{\iterl{u_J}(x_2)} & =
	\frac{\int_Y \exp(-c(x_2,y)/\veps)\,\iterl{v}_J(y)\,\diff \nu(y)}{\int_Y \exp(-c(x_1,y)/\veps)\,\iterl{v}_J(y)\,\diff \nu(y)}
	\end{align*}
	for $\mu_J$-a.e.~$x_1, x_2 \in X_J$, again using that $(\diff\mu_J/\diff\mu)(x_1) = (\diff\mu_J/\diff\mu)(x_2) = 1$. Thus, \eqref{item:ScalingLemma:Ratio} follows by noting that $c(x_2,y)\geq c(x_1,y) - \|c\|$.
	
	\noindent
	\eqref{item:DensityAbsBound:Abs}. For the upper bound we combine \eqref{eq:ExtremalityConditionIterl} and \eqref{eq:ScalingRatioBound} on the basic cell $i$, to obtain
	\begin{align*}
	\iterl{u}_i(x) \cdot \iterl{v}_i(y)
	& = \iterl{u}_{i}(x) \cdot \RadNik{\iterl{\nu}_i}{\nu}(y)\,
	\left(\int_{X} \sKernel(x',y)\,\iterl{u}_i(x')\,\diff \mu_i(x')\right)^{-1} \\
	& \leq \iterl{u}_{i}(x) \cdot \RadNik{\iterl{\nu}_i}{\nu}(y)\,
	\left(\exp(-2\|c\|/\veps) \cdot \iterl{u}_i(x) \cdot \int_{X} \diff \mu_i(x')\right)^{-1}
	\leq \frac{\exp(2\|c\|/\veps)}{\|\mu_i\|}
	\end{align*}
	where we have used $\sKernel(x',y) \geq \exp(-\|c\|/\veps)$ in the first inequality and $(\diff\iterl{\nu}_i/\diff\nu)(y)\leq 1$ for $\nu$-a.e.~$y\in Y$ in the second.
	The lower bound works almost entirely analogous by using $\sKernel(x',y)\leq 1$.
	
	\noindent
	\eqref{item:DensityAbsBound:Rel}. With $j \in J \in \iterl{\partGeneric}$, we recall
	\begin{equation}\label{eq:nuboundcell}
	\iterl{\nu}_J = \sum_{k\in J}P_Y(\iterlm{\pi} \restr (X_k \times Y)) = \sum_{k \in J} \iterlm{\nu}_{k} \geq \iterlm{\nu}_{j}.
	\end{equation}
	Since $i, j \in J \in \iterl{\partGeneric}$ our convention is that $\iterl{v}_i=\iterl{v}_{j}=\iterl{v}_J$.
	Combining \eqref{eq:ScalingRatioBound} and \eqref{eq:nuboundcell}, one obtains for $(\mu_i \otimes \mu_{j} \otimes \nu)$-a.e.~$(x,x',y) \in X_i \times X_j \times Y$
	\begin{align*}
	\iterl{u}_i(x) \cdot \iterl{v}_i(y) &= \fint_X \iterl{u}_i(x) \cdot \iterl{v}_i(y) \,\diff\mu_J(\hat{x}) \geq \exp(-\|c\|/\veps) \cdot \iterl{v}_J(y)\cdot \, \fint_{X} \iterl{u}_J(\hat{x})\, \sKernel(\hat{x},y)\,\diff \mu_J(\hat{x}) \\
	&= \tfrac{\exp(-\|c\|/\veps)}{\|\mu_J\|}\cdot \RadNik{\iterl{\nu}_J}{\nu}(y) \geq \exp(-\|c\|/\veps) \cdot \RadNik{\iterlm{\nu}_{j}}{\nu}(y) \\
	&= \exp(-\|c\|/\veps)\cdot \iterlm{v}_{j}(y)\,\int_{X} \iterlm{u}_{j}(\hat{x})\, \sKernel(\hat{x},y)\,\diff \mu_j(\hat{x}) \\
	&\geq \exp(-3\|c\|/\veps)\|\mu_j\| \cdot \iterlm{u}_{j}(x') \cdot \iterlm{v}_{j}(y).
	\end{align*}
	Eventually, \eqref{eq:RadNikLowerBound1Step} follows with similar arguments: we use the first two lines in the derivation here above and $\exp(-\|c\|/\veps) \leq \sKernel(x,y) \leq 1$ to obtain, for $\nu$-a.e.~$y \in Y$
	\begin{align*}
	\RadNik{\iterl{\nu}_i}{\nu}(y)
	& = \int_{X} \iterl{u}_i(x)\,\iterl{v}_i(y)\,\sKernel(x,y)\,\diff \mu_i(x) \geq \exp(-2\|c\|/\veps) \,\|\mu_i\| \cdot \RadNik{\iterlm{\nu}_{j}}{\nu}(y).
	\qedhere
	\end{align*}
\end{proof}

\begin{lemma}[Approximate triangle inequality for $\KL$]
	\label{lem:PhiProductDecomposition}
	Let $L \in (0,\infty)$, $n \in \N$. Then for all $(s_1,\ldots,s_n) \in [0,L]^n$ one finds
	\begin{align}
	\label{eq:PhiProductDecomposition}
	\varphi\left(\prod_{k=1}^n s_k \right) & \leq C \cdot \sum_{k=1}^n \varphi(s_k) \quad
	\text{for} \quad 
	C =n \cdot \max\{2,L^{n-1}\}.
	\end{align}
\end{lemma}

\begin{proof}
	For convenience introduce the functions $S,A, B : \R_+^n \to \R$,
	\begin{align*}
	S(s) & = \prod_{k=1}^n s_k, &
	A(s) & = \varphi(S(s)), &
	B(s) & = \sum_{k=1}^n \varphi(s_k).
	\end{align*}
	Note that $A$ and $B$ are zero if and only if $s=(1,\ldots,1)$. In this case \eqref{eq:PhiProductDecomposition} holds for any $C \in \R$. When one of the $s_k$ equals zero, the bound holds for $C \geq 1$.
	So, to obtain a feasible constant $C$ we obtain an upper bound on
	\begin{align}
	\label{eq:ProofPhiProductDecompositionSup}
	\sup\left\{ \frac{A(s)}{B(s)} \middle| s \in (0,L]^n,\,s \neq (1,\ldots,1) \right\}.
	\end{align}

	Consider now the case where $S(s) \geq 1$ and $s_k<1$ for some $k$. In this case we find that
	\begin{align*}
	\frac{\partial A}{\partial s_k}(s) & = \varphi'(S(s)) \tfrac{S(s)}{s_k} = \log(S(s)) \tfrac{S(s)}{s_k} \geq 0, &
	\frac{\partial B}{\partial s_k}(s) & = \varphi'(s_k) = \log(s_k) < 0.
	\end{align*}
	In this case we can increase $\tfrac{A(s)}{B(s)}$ by increasing $s_k$.
	Likewise, assume $0<S(s)<1$, $s_k>1$. Then we find
	\begin{align*}
	\frac{\partial A}{\partial s_k}(s) & < 0, &
	\frac{\partial B}{\partial s_k}(s) & > 0.
	\end{align*}
	So we can increase $\tfrac{A(s)}{B(s)}$ by decreasing $s_k$. In summary, we can restrict the supremum \eqref{eq:ProofPhiProductDecompositionSup} to the cases $[S(s) \geq 1,\,s_k \geq 1]$ and $[S(s)<1,\,0<s_k \leq 1]$.
	
	\smallskip
	Supremum over $[S(s) \geq 1,\,s_k \geq 1]$: with the additional constraint we can rewrite \eqref{eq:ProofPhiProductDecompositionSup} as
	\[
	\sup\left\{ \sup \left\{ \frac{A(s)}{B(s)} \middle| s \in [1,L]^n, S(s)=\tilde{S} \right\}
	\middle|
	\tilde{S} \in (1,L^n] \right\}
	\]
	where we can ignore the case $\tilde{S}=1$ since then $s=(1,\ldots,1)$. By reparametrizing $s_k = \exp(r_k)$ we find:
	\[
	\sup\left\{ \sup \left\{ \frac{\varphi(\tilde{S})}{\sum_{k=1}^n \varphi(\exp(r_k))} \middle| r \in [0,\log L]^n, \sum_{k=1}^n r_k =\log \tilde{S} \right\}
	\middle|
	\tilde{S} \in (1,L^n] \right\}
	\]
	Since $r \mapsto \sum_{k=1}^n \varphi(\exp(r_k))$ is convex on $\R_+^n$ and invariant under permuting the order of the arguments $r_k$, the inner supremum is attained when setting $r_k = \frac{1}{n} \log \tilde{S}$ and thus one finds (by reparametrizing $\tilde{S}=\tilde{s}^n$)
	\[
	\sup\left\{ \frac{\varphi(\tilde{s}^n)}{n \cdot \varphi(\tilde{s})} 
	\middle|
	\tilde{s} \in (1,L] \right\}
	\]
	Observe now that
	\[
	g(\tilde{s}) \assign \varphi(\tilde{s}^n) - n^2\tilde{s}^{n-1}\varphi(\tilde{s}) \leq 0 \quad \text{for all }\tilde{s}>1 
	\]
	because $g(1) = 0$ and $g'(\tilde{s}) = -n^2(n-1)\tilde{s}^{n-2}\varphi(\tilde{s}) \leq 0$ for all $\tilde{s} > 1$. In particular,
	\[
	\frac{\varphi(\tilde{s}^n)}{n\varphi(\tilde{s})} \leq n\tilde{s}^{n-1} \leq nL^{n-1} \quad \text{for all } \tilde{s} \in (1,L]
	\]
	which yields the conclusion.
	
	\smallskip
	Supremum over $[S(s)<1,\,0<s_k \leq 1]$: for $s \in [0,1]^n$ let $r_k \assign 1-s_k$ and set $R=\sum_{k=1}^n r_k$. Then one finds $S(s)\geq \max\{0,1-R\}$ and, since $\varphi$ is decreasing on $[0,1]$, $A(s)=\varphi(S(s))\leq \varphi(\max\{0,1-R\})$.
	Similarly for $B$, with convexity of $\varphi$ and Jensen's inequality, one has
	\begin{align*}
	B(s) = \sum_{k=1}^n \varphi(s_k) \geq n \cdot \varphi\left(\tfrac{1}{n} \sum_{k=1}^n s_k \right)=
	n \cdot \varphi\big(1-\tfrac{1}{n} R\big).
	\end{align*}
	This allows us to rewrite \eqref{eq:ProofPhiProductDecompositionSup} for the second case as
	\[
	\sup \left\{ \frac{A(s)}{B(s)} \middle| s \in [0,1]^n,\, S(s)<1 \right\}
	\leq \sup \left\{ \frac{\varphi(\max\{0,1-R\})}{n \cdot \varphi\big(1-\tfrac{1}{n} R\big)}
	\middle| R \in (0,n] \right\}
	\]
	Since the denominator is increasing in $R$ and the numerator is constant on $R \in [1,n]$, we can restrict to $R \in (0,1]$ and prove that the supremum is bounded by $2n$. To do so, we introduce
	\[
	g(R) \assign \varphi(1-R) - 2n^2\varphi\big(1-\tfrac{1}{n} R\big) \quad \text{for } R \in [0,1].
	\]
	To conclude it suffices to prove that $g(R) \leq 0$ for all $R \in [0,1]$ and any $n \geq 2$. Indeed, we have $g(0) = 0$, $g(1) < 0$ and the derivative $g'(R) = -\log(1-R) + 2n\log\big(1-\tfrac{1}{n} R\big)$ is such that $g'(0) = 0$, $g'(R) \to +\infty$ as $R \to 1^-$ and $g'$ has a unique zero on $(0,1)$ at $R_n = \frac{n}{2n-1}$. We conclude that in $(0,1)$ the function $g$ has a unique critical point (a minimum) attained at $R_n$, thus $g(R) \leq 0$ for all $R \in [0,1]$.	
\end{proof}

\bibliography{references}{}
\bibliographystyle{spmpsci}

\end{document}